\documentclass[nohyperref]{article}

\pdfoutput=1 

\usepackage{microtype}
\usepackage{graphicx}
\usepackage{subfigure}
\usepackage{booktabs} 

\usepackage[accepted]{icml2022} 

\usepackage{hyperref}       
\usepackage{url}            
\usepackage{natbib}
\usepackage{amsmath}
\usepackage{amssymb}
\usepackage{amsthm}
\usepackage{amsfonts}
\usepackage{textcomp}
\usepackage{xcolor}
\usepackage{array}
\usepackage{multirow} 
\usepackage{diagbox}
\usepackage{soul}
\usepackage{titletoc}
\usepackage{float}

\usepackage{algorithm}
\usepackage{algorithmic}
\usepackage{setspace}

\newtheorem{Thm}{Theorem}

\newtheorem{Lem}{Lemma}
\newtheorem{Ass}{Assumption}
\newtheorem{Rem}{Remark}
\newtheorem{Cor}{Corollary}

\def \hy #1{\textcolor{black}{#1}}
\def \tocheck #1{\textcolor{black}{#1}}
\def \finetune #1{\textcolor{black}{#1}}

\usepackage[capitalize,noabbrev]{cleveref}

\usepackage[textsize=tiny]{todonotes}

\icmltitlerunning{Tackling Data Heterogeneity: A New Unified Framework for Decentralized SGD}

\begin{document}

%

%





\twocolumn[
\icmltitle{Tackling Data Heterogeneity: A New Unified Framework for Decentralized SGD with Sample-induced Topology}




\begin{icmlauthorlist}
\icmlauthor{Yan Huang}{zju}
\icmlauthor{Ying Sun}{psu}
\icmlauthor{Zehan Zhu}{zju}
\icmlauthor{Changzhi Yan}{zju}
\icmlauthor{Jinming Xu}{zju}
\end{icmlauthorlist}

\icmlaffiliation{zju}{College of Control Science and Engineering, Zhejiang University, Hangzhou, China}
\icmlaffiliation{psu}{{School of Electrical Engineering and Computer Science, The Pennsylvania State University,  PA 16802, USA}}

\icmlcorrespondingauthor{Jinming Xu}{jimmyxu@zju.edu.cn}

\icmlkeywords{Machine Learning, ICML}

\vskip 0.3in
]



\printAffiliationsAndNotice{}  

\begin{abstract}

We develop a general framework unifying several gradient-based stochastic optimization methods for empirical risk minimization problems both in centralized and distributed scenarios. The framework hinges on the introduction of an augmented graph consisting of nodes modeling the samples and edges modeling both the inter-device communication and intra-device stochastic gradient computation. By designing properly the topology of the augmented graph, we are able to recover as special cases the renowned Local-SGD and DSGD algorithms, and provide a unified perspective for variance-reduction (VR) and gradient-tracking (GT) methods such as SAGA, Local-SVRG and GT-SAGA. We also provide a unified convergence analysis for smooth and (strongly) convex objectives relying on a proper structured Lyapunov function, and the obtained rate can recover the best known results for many existing algorithms. The rate results further reveal that VR and GT methods can effectively eliminate data heterogeneity within and across devices, respectively, enabling the exact convergence of the algorithm to the optimal solution. Numerical experiments confirm the findings in this paper.
\end{abstract}

\section{Introduction} \label{Introduction}

With the increasing popularity of large-scale machine learning, distributed stochastic optimization methods have sparked considerable interest to improve learning efficiency in both academia and industry \citep{lian2017can, boyd2011distributed}. {In contrast to the typical centralized/parameter-server architecture \citep{dean2012large, lian2015asynchronous, stich2019local}} where a center node coordinates the entire optimization process, which usually becomes the bottleneck, distributed structure has its unique advantage in improving computation and communication efficiency  \citep{nedic2018network}. Besides, since the data is locally maintained by each node, data privacy can be well preserved for data-sensitive application domains \citep{li2019convergence}.

We consider the prototypical empirical risk minimization problem collaboratively solved by a set of agents over a communication network. The overall objective of the agents is to seek an optimal solution $x^*\in \mathbb{R}^d$ that solves the following finite-sum problem:
\begin{equation}\label{ERM}
\underset{x\in \mathbb{R}^d}{\min}f\left( x \right) =\frac{1}{n}\sum_{i=1}^n{\left(f_i(x):=\frac{1}{m}\sum_{j=1}^m{\underset{f_{ij}\left( x \right)}{\underbrace{f\left( x,\xi _{i,j} \right) }}}\right)},
\end{equation}
where $f_i:\mathbb{R}^d\rightarrow \mathbb{R}$ is the local private loss function accessible only by the associated node
{$i\in\mathcal{N}:=\left\{ 1,2,\cdots ,n \right\}$}, and $\{\xi_{i,j}\}_{j = 1}^{m}$ denote the data samples locally stored at node $i\in\mathcal{N}$ and $\{f_{ij}\}$ denote the corresponding loss functions. 

Problem (\ref{ERM}) has been extensively studied over the last decade and enormous distributed algorithms, e.g., \citet{nedic2009distributed, nedic2010constrained, lobel2011distributed, yuan2016convergence, pu2020push}, have been proposed to solve this problem. The readers are referred to the recent survey paper \citep{nedic2018network} and the references therein. Among these algorithms, the distributed gradient decent (DGD) algorithm \citep{nedic2009distributed} is a simple yet effective method. However, it suffers from steady-state error when employing constant stepsizes due to the fact that the fixed point of DGD is inherently not consensual \citep{yuan2016convergence}. To bridge the gap between centralized gradient decent and DGD, a gradient-tracking scheme is introduced to overcome the above issue \citep{xu2015augmented, di2016next, nedic2017achieving, qu2017harnessing}, which introduces an auxiliary variable to 
estimate the full gradient $\nabla f$ (the sum of local gradients) leveraging dynamic average consensus \citep{zhu2010discrete} so as to make the fixed point consensual. The introduction of gradient tracking scheme allow us to effectively account for the heterogeneity among local data sets (also known as \emph{external} variance among data distribution of nodes).  


\textbf{Variance-reduced (VR) and local methods.} All of the above-mentioned algorithms are deterministic and require evaluating the local full  gradient $\nabla f_i$ at each iteration, leading to high computational complexity \citep{hong2017prox}. A natural way to reduce the computational complexity is to use stochastic gradients to approximate the local full gradient. To this end, a vast number of distributed stochastic optimization algorithms are proposed for the problem (\ref{ERM}), such as DSGD \citep{ram2009asynchronous}, D-PSGD \citep{lian2017can} and SGP \citep{assran2019stochastic}, and just to name a few. These stochastic optimization algorithms work quite well in practice but usually require diminishing stepsizes to attenuate the variance of the stochastic gradient, a.k.a. \emph{internal} (sample) variance. An effective solution to this is to employ the idea of variance-reduction and learn the local full gradient $\nabla f_i$ iteratively, as did in SVRG \citep{johnson2013accelerating}, SAGA \citep{defazio2014saga}, L-SVRG \citep{hofmann2015variance, qian2021svrg} and {SARAH \citep{pmlr-v70-nguyen17b}}  to eliminate the internal variance, yielding faster convergence such as D-SAGA \citep{calauzenes2017distributed}. To avoid high communication burden among nodes, one may trade computation with communication by performing several local gradient steps between two consecutive communication steps, which leads to communication-efficient algorithms, such as Local-SGD \citep{khaled2020tighter} and Local-SVRG \citep{gorbunov2021local}; \hy{or employing periodic global averaging (PGA) for speeding up consensus, such as Gossip-PGA \citep{chen2021accelerating}.}

\textbf{Gradient-tracking-based (GT) stochastic methods.} The gradient-tracking scheme have been also recently incorporated into various stochastic optimization algorithms as a key step to eliminate the \emph{external} variance. For example,  DSGT \citep{pu2020distributed} and DSA \citep{mokhtari2016dsa} can achieve higher accuracy than that of DSGD by removing the bias due to the heterogeneity among local data sets. Nevertheless, gradient-tracking, by its nature, is still unable to eliminate the data sample variance. This naturally leads to the integration of variance-reduction methods with the gradient-tracking scheme. For instance, \citet{sun2020improving} employ a scheme with both gradient-tracking and variance-reduction to solve a smooth (probably non-convex) problem and show that it converges to a stationary point sublinearly. {\citet{li2021destress} proposed a similar algorithm with a nested loop structure for the sake of improving its overall complexity.} \hy{\citet{xin2020variance} and \citet{jiang2022distributed} consider a similar GT-VR framework and obtain a linear rate for strongly convex problems and $\mathcal{O}\left( 1/k \right) $ rate for non-convex setting, respectively.} {Similar attempts have been recently made towards composite optimization problems \cite{ye2020pmgt}.}
\hy{
There are also many other efforts made to solve Problem \eqref{ERM}, such as those based on approximate Newton-type methods \citep{li2020communication} and acceleration schemes \citep{scaman2017optimal, hendrikx2021optimal}.}

\textbf{Unified framework for first-order stochastic optimization algorithms.} There have been some efforts made to unify the aforementioned algorithms. In particular, {\citet{hu2017unified}} unify several variance-reduction methods by establishing the intrinsic connection between stochastic optimization methods and dynamic jump systems. \citet{wang2021cooperative} consider to use a time-varying mixing matrix to model Cooperative SGD which can recover several existing non-variance-reduced methods. Building on this, \citet{koloskova2020unified} propose a unified framework for Decentralized (Gossip) SGD by employing changing topology and multiple local updates. To incorporate variance-reduction methods, \citet{gorbunov2020unified} study a general framework that can account for variance-reduction, importance sampling, mini-batch sampling, leading to a unified theory of variance reduced and non-variance-reduced SGD methods. The authors also extend this framework to cover
Local-SGD and variance-reduced SGD methods, which recovers the rate in \citep{koloskova2020unified}.
\hy{However, to the best of our knowledge, there is no such a framework for distributed first-order stochastic optimization algorithms that can recover all the above-mentioned VR-, GT-, Local- and PGA-based methods both in centralized and decentralized settings.}

\subsection{Our Contribution}
In this work, we develop a new unified framework based on the introduction of an augmented graph whose nodes model the data samples. Leveraging a proper sampling strategy on the augmented graph, this framework allows us to recover many existing algorithms as well as their corresponding best known rates. In contrast to the existing frameworks as mentioned above, the proposed framework not only contain them as special cases but also enable us to easily design new efficient algorithms with guaranteed rates, especially those employing gradient-tracking scheme, yielding a broader range of methods to be incorporated. The main contributions of this paper are summarized as follows:

\vspace{-0.2cm}
\begin{itemize}
 \item \textbf{New unified framework for algorithm design and analysis.} The proposed framework unify various gradient-based stochastic optimization methods both in centralized and distributed scenarios. With proper sampling strategies on the augmented graph, we can easily recover these VR-, GT-, Local- and PGA-based methods, as well as their proper combinations (see Table~\ref{Tab_comparing}). \hy{Besides, this framework also provides a new unifying perspective for GT- and VR-based schemes, which are otherwise two separate approaches before, and show an equivalence of Local-SGD, Gossip-PGA and DSDG in terms of iteration complexity once their expected topology connectivity is same.} 
 \item \textbf{Recovering various existing algorithms along with the best known rates.} A unified convergence analysis is provided, which relies on a proper Lyapunov function for smooth and (strongly) convex objectives. The obtained rates either recover the existing best known rates or are new for certain algorithms under our settings, including SAGA, Local-SGD, DSGD, Local-SVRG and GT-SAGA (see Table~\ref{Tab_comparing}). 
 \hy{These rates also show
the clear dependence of the convergence performance on the
above-mentioned schemes, such as GT, VR, Local update,
and PGA.} The theoretical results further reveal that VR- and GT-based methods are usually needed to achieve exact convergence in scenarios where data heterogeneity is a key concern.

 \item \textbf{New efficient algorithms with provable rates.} Our framework allows us to easily come up with new efficient algorithms with proper design of the sampling strategy on the augmented graph and provides the corresponding rate guarantee as well, such as Local-SAGA, PGA-SAGA, PGA-GT-SAGA, which are not formally analyzed before (see Table~\ref{Tab_comparing}). Moreover, the proposed framework provides more flexibility in design of \hy{network topology which can be of multi-layer structure in certain scenarios for communication efficiency}.
 
\end{itemize}

\section{Problem Formulation}
\label{Preliminary}
We consider solving Problem~\eqref{ERM} over a peer-to-peer network, modeled as an augmented directed graph (digraph) $\mathcal{G}=(\mathcal{V},\mathcal{E})$ where $\mathcal{V}:=\{1,2,...,M\}$ with $M=nm$ denotes the set of nodes modeling the data samples and $\mathcal{E}\subseteq\mathcal{V}\times \mathcal{V}$ denotes the set of edges consisting of ordered pairs $(i,j)$ modeling the virtual/actual communication link from $j$ to $i$. We then make the following blanket assumptions on the cost functions of problem~\eqref{ERM}.

\begin{Ass}\label{Ass_smoothness}
Each $f_{i}: \mathbb{R}^d\rightarrow \mathbb{R}$ is $\mu$-strongly convex and $L$-smooth, i.e., for any $x, x'\in \mathbb{R}^d$
\begin{equation}
\begin{aligned}
\left< \nabla f_i\left( x \right) -\nabla f_i\left( x^{\prime} \right), x-x^{\prime} \right> \geqslant \mu \left\| x-x^{\prime} \right\| ^2,
\end{aligned}
\end{equation}
\vspace{-0.25cm}
\begin{equation}
\begin{aligned}
& \left\| \nabla f_i\left( x \right) -\nabla f_i\left( x' \right) \right\|\leqslant L\left\| x-x' \right\|.
\end{aligned}
\end{equation}
\end{Ass}

\begin{Ass}\label{Ass_sampling}
(Bounded data heterogeneity at optimum) Let $x^* \in \mathrm{arg}\min_{x} f\left( x \right) $.
For each $f_{ij}: \mathbb{R}^d\rightarrow \mathbb{R}$, $i\in [n], j\in [m] $, there exist positive constants $\sigma ^*, \zeta ^*$ such that
\begin{equation}\label{Def_noise}
\begin{aligned}
&\frac{1}{M}\sum_{i=1}^n{\sum_{j=1}^m{\left\| \nabla f_{ij}\left( x^* \right) -\nabla f_i\left( x^* \right) \right\| ^2}}\leqslant \sigma ^*,
\end{aligned}
\end{equation}
\vspace{-0.3cm}
\begin{equation}\label{Def_heterogeny}
\begin{aligned}
&\frac{1}{n}\sum_{i=1}^n{\left\| \nabla f_i\left( x^* \right) \right\| ^2}\leqslant \zeta ^*.
\end{aligned}
\end{equation}
\end{Ass}

\begin{Rem}
The parameter $\sigma ^*, \zeta ^*$ are  defined to measure the local gradient sampling variance and data heterogeneity across devices, respectively. 
Notice that we do not require the data heterogeneity be bounded uniformly for all $x$ but only at the optimum $x^*$, which is weaker than the previous works such as {\citep{lian2017can, wang2021cooperative}}.
\end{Rem}

\begin{Ass}\label{Expected smoothness}
({Averaged smoothness})
For all $i\in[n]$ and $\forall x, x^\prime\in \mathbb{R}^d$, we have
\begin{equation}\label{Eq_Expe_smooth}
\begin{aligned}
&\frac{1}{m}\sum_{j=1}^m{\left\| \nabla f_{ij}\left( x \right) -\nabla f_{ij}\left( x^{\prime} \right) \right\| ^2}
\\
&\leqslant 2L\left( f_i\left( x \right) -f_i\left( x^{\prime} \right) -\left< \nabla f_i\left( x^{\prime} \right) , x-x^{\prime} \right> \right). 
\end{aligned}
\end{equation}
\end{Ass}
\vspace{-0.25cm}
Assumption \ref{Expected smoothness} is automatically satisfied if we assume that each $f_{ij}$ is convex and $L$-smooth.

\section{Sample-wise Push-Pull Framework}

In this section, we introduce the sample-wise Push-Pull framework for the finite-sum problem (\ref{ERM}). To this end, we use an augmented graph $\mathcal{G}$ with a two-level structure as depicted in Figure~\ref{Fig_augmented_graph} to illustrate the key ideas. 

\begin{figure}[h] 
    \centering
    {
        \begin{minipage}[t]{0.45\textwidth}
            \centering          
            \includegraphics[width=\textwidth]{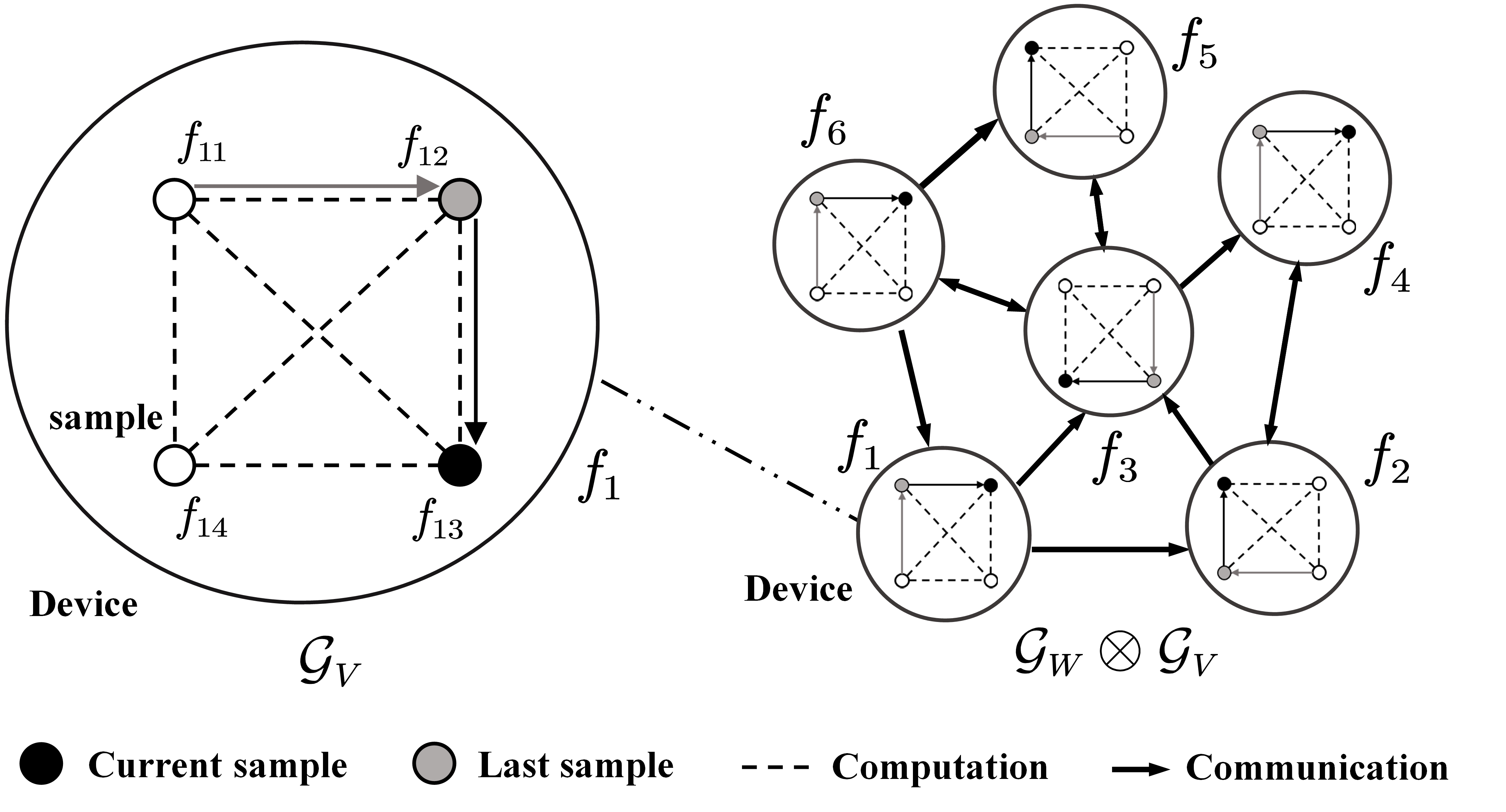}   
        \end{minipage}%
    }
    \vspace{-0.2cm}
     \caption{A two-level augmented graph for $n=6$ and $m=4$. The graph $\mathcal{G}_V$ on the left refers to a fully connected underlying graph within one device, and the solid arrows represent the \emph{local} stochastic gradient computation for updating parameters. The graph $\mathcal{G} = \mathcal{G}_W\otimes \mathcal{G}_V$ on the right represents the augmented graph induced by the samples locally stored at all the devices.}\label{Fig_augmented_graph}
\end{figure}
We first model each device as a super-node connected by the actual communication network $\mathcal{G}_W$. Then, we replace each device $i$ with a virtual fully connected graph $\mathcal{G}_V$ whose nodes model the samples $\xi_{i,j}$, yielding an augmented graph $\mathcal{G} = \mathcal{G}_W\otimes \mathcal{G}_V$. 
\textcolor{black}{As we will show shortly, implementation of stochastic optimization algorithms can be deemed as recursively sampling the edges of $\mathcal{G}$, which  consists in: \textbf{i)} locally updating the parameter of sample node $j$ in device $i$ using the gradient of $f_{il}$ when the edge $(j, l)\in\mathcal{G}_V$ associated with device $i$; \textbf{ii)}  performing an actual communication step between devices $i$ and $s$ when the edge $(i, s) \in\mathcal{G}_W$. }

As a convention in decentralized optimization, we associate each node in $\mathcal{G}$ a local variable $x_i \in \mathbb{R}^d$ that serves as a copy of the global variable $x$. Besides, each node $i$ also maintains an auxiliary variable $y_i \in \mathbb{R}^d$ that estimates the full gradient $\nabla f(x)$. For brevity, we use $X, Y\in \mathbb{R}^{M\times d}$ to denote matrices stacking the $x_i$'s and $y_i$'s, respectively:
\begin{equation*}
\begin{aligned}
X:=\left[ x_{1}, x_{2}, \cdots, x_{M} \right] ^T,\,\,
Y:=\left[y_{1}, y_{2}, \cdots , y_{M} \right] ^T.
\end{aligned}
\end{equation*}
Accordingly, the collective gradient vectors of all local objective functions at $X$ is denoted as
$$
\nabla F\left( X \right) :=\left[ \nabla f_1\left( x_{1} \right) ,\cdots ,\nabla f_M\left( x_{M} \right)  \right] ^T.
$$
Now, we are ready to introduce our algorithmic framework, termed sample-wise Push-Pull (SPP), for Problem (\ref{ERM}):
\begin{subequations}\label{SPP_framework}
\begin{align}
&X_{k+1}=R_{k}X_k-\alpha \varGamma _{k}Y_k \label{recu_X},
\\
&Y_{k+1}=C_{k}Y_k+\nabla F\left( X_{k+1} \right) -\nabla F\left( X_k \right) \label{recu_Y},
\end{align}
\end{subequations}
where $\alpha$ is a constant step-size, and $R_k, C_k\in \mathbb{R}^{M\times M}$ are time-varying weight matrices (to be properly designed) for reaching consensus (pull operation) among nodes and tracking the overall average gradient (push operation), respectively; $\varGamma _k \in \mathbb{R}^{M\times M}$ is a random sampling matrix for selecting edges of the augmented graph.

In what follows, we explain the role of the above three key parameters of $\varGamma _{k}, R_k$ and $C_k$ over an augment graph with the two-level structure as shown in Figure~\ref{Fig_augmented_graph}.

\textbf{Sampling on augmented graph.} In our proposed SPP framework,
both the processes of {actual communication among devices and local stochastic gradient computation for updating parameters} can be regarded as selecting a subset of nodes from the augmented graph $\mathcal{G}=\mathcal{G}_W\otimes \mathcal{G}_V$.
In particular, each device $i$ elects its samples (nodes in $\mathcal{G}_V$) participating in the update at iteration $k$, indicated by a binary-valued vector $\boldsymbol{e}_{i,k}\in \left\{ 0,1 \right\} ^{m\times 1}$: the $j$-th element $e_{i,k}^j$ is set to be one if the $\xi_{ij}$ is selected and zero otherwise. The concatenated vector $\boldsymbol{e}_k=[\boldsymbol{e}_{1,k}^{T}\,\,, \cdots , \boldsymbol{e}_{n,k}^{T}]^T \in \mathbb{R}^M$ then indicates the identities of all the samples in the network selected at iteration $k$. \textcolor{black}{Denote by $b_k$ the number of selected samples of each device at iteration $k$, i.e., $b_k=\mathbf{1}_m^T e_{i,k}, \forall i\in \left[ n \right]$.}
Inherited from the weight matrix of $\mathcal{G}$ at iteration $k$, which is $W_k \otimes \mathbf{1}\mathbf{1}^T$, the sampling matrix becomes:
\begin{equation}\label{Lab_Gamma_k}
\begin{aligned}
&\varGamma _k:=\varLambda _{k+1}\left( W_k\otimes \mathbf{1}\mathbf{1}^T \right) \frac{\varLambda _k}{b_k}, 
\end{aligned}
\vspace{-0.2cm}
\end{equation}
{where $\Lambda_k = \mathbf{diag}\left( \boldsymbol{e}_{k} \right)$ denotes the elected samples at iteration $k$ who will send messages to the nodes picked by $\Lambda_{k+1}$.} Indeed, $\varGamma _k$ models the virtual/actual message passing among sample node from iteration $k$ to $k+1$.

Then, we consider the following sampling strategy,
which is commonly used for many distributed learning problems.

\begin{Ass}\label{Ass_uniformly_sampling_without_replacement}
For all $k\geqslant0$, each device $i\in[n]$ independently and uniformly selects $b_k$ data samples from its local datasets at random without replacement.
\end{Ass}

\textbf{Intra and inter consensus guarantee.} The main purpose of the weight matrix $R_k$ is to ensure consensus of estimates within and across devices, which is designed as follows:
\begin{equation} \label{Lab_R_k}
\begin{aligned}
&R_k:=\mathbf{I}_M-\varLambda _{k+1}+\varGamma _k,
\end{aligned}
\end{equation}
where the term $\mathbf{I}_M-\varLambda _{k+1}$ represents that the parameters kept at the nodes that are not sampled remain unchanged at iteration $k+1$. The term $\varGamma_k$, as defined in (\ref{Lab_Gamma_k}), denotes the message passing from samples $\xi _{k}$ to $\xi _{k+1}$ over the augmented graph, and only samples in $\xi _{k+1}$ perform update.
Note that in such design, $R_k$ is \emph{row-stochastic}. There are indeed two consensus processes involved in the above process: 1) \emph{consensus within each device}, which is guaranteed by the fully connectivity of $\mathcal{G}_V$, meaning that the latest parameters can be always sent to the current sample node; 2) \emph{consensus across devices}, which is ensured by the proper design of weight matrix $W_k$ that has been incorporated in $\varLambda_k$.

\textbf{Accurate full-gradient estimation for tackling data heterogeneity.} In order to obtain an accurate estimate on gradient descent direction, variance-reduction and gradient-tracking methods are widely used to eliminate the variance of the stochastic gradient within and across the devices, respectively. In the proposed framework, we properly design the doubly-stochastic weight matrix $C_k$ corresponding to the augmented graph $\mathcal{G}_W\otimes \mathcal{G}_V$ as
\begin{equation}\label{Lab_C_k}
\begin{aligned}
C_k:=G_k\otimes V_k.
\end{aligned}
\end{equation}
Since $C_k$ is doubly-stochastic, we have by induction,
\begin{equation}\label{Eq_fixed_sum_Y}
\begin{aligned}
\frac{\mathbf{1}_{M}^{T}}{M}Y_k=\frac{\mathbf{1}_{M}^{T}}{M}\nabla F\left( X_k \right) , \forall k \geqslant 0,
\end{aligned}
\end{equation}
which enables all the nodes to track the full gradient $\nabla f(x)$. In fact, it will become clear that $G_k\in \mathbb{R}^{n\times n}$ is meant for gradient-tracking across devices while $V_k\in \mathbb{R}^{m\times m}$ is dedicated to variance-reduction within device.

Thus, we can properly choose these above weight matrices $\varGamma _{k}, R_k,C_k$ to recover existing algorithms with or without GT and VR operations. 
We use $\mathcal{A}\left(\varGamma _{k}, R_k,C_k \right)$ to denote algorithms generated from the proposed SPP framework. 
For ease of presentation, we will use the above same setting for $\varGamma _{k}$, $R_k$ and $C_k$ throughout the paper.

\begin{Rem}
In contrast to the existing frameworks, our proposed SPP framework provides a more general prospective for algorithm design based on sampling of an augmented graph, i.e., we allow for adopting various consensus schemes, gradient estimation methods and their combinations by different choices of the three key matrices $R_k$, $C_k$ and $\varGamma_k$. Our framework covers several recently proposed frameworks, such as Local-SGD {\citep{gorbunov2021local}}, {Cooperative SGD} {\citep{wang2021cooperative}}, {Decentralized (Gossip) SGD} {\citep{koloskova2020unified}}. Indeed, none of these works consider both gradient-tracking and variance-reduction with changing topology and local updates. Besides, the proposed framework has the potential to recover {Multi-Level Local-SGD} \citep{castiglia2020multi} by properly designing a topology of hierarchical structures.
\end{Rem}

\section{Existing Algorithms as Special Cases and Beyond}\label{Sec_Recover}
In this section, we show how the proposed SPP framework can, indeed, recover a large number of existing algorithms as special cases. 
To this end, we introduce a projection matrix $S_k\in \mathbb{R}^{n\times M}$, which is defined as
\begin{equation} \label{Def_projection_matrix}
\begin{aligned}
S_k:=\left( \mathbf{I}_n\otimes \mathbf{1}_{m}^{T} \right) \frac{\varLambda _k}{b_k}, 
\end{aligned}
\end{equation}
{that can reduce the dimension of the general SPP algorithm from $M$ (number of samples) to  $n$ (number of physical devices), yielding an algorithm that can be implemented efficiently on actual devices with new updating variables:}
\begin{equation}\label{Def_X_Y_hat}
\begin{aligned}
\hat{X}_k:=S_kX_k, \quad \hat{Y}_{k}:=S_{k}Y_{k}.
\end{aligned}
\end{equation}
Now, we will use the cases of SAGA/L-SVRG ($n=1$), and GT-SAGA ($n > 1$) to illustrate the equivalence between the {recovered} algorithm  and the general SPP algorithm. More details on the recovery of various existing algorithms can be found in Appendix \ref{Appendix_recover}.

\paragraph{Recovering SAGA/L-SVRG.}
Under the proposed SPP framework, we choose the parameters $\left(\varGamma _{k}, R_k,C_k \right)$ as defined in (\ref{Lab_Gamma_k}), (\ref{Lab_R_k}) and (\ref{Lab_C_k}) for SAGA with:
\begin{equation}
\begin{aligned}
W_k=G_k=1,\,\, V_k=\mathbf{J}_m,\,\, \forall k\geqslant 1.
\end{aligned}
\end{equation}
\tocheck{We denote by $\boldsymbol{s}_k$ the mini-batch of randomly selected $b\in[1, m]$ sample nodes at iteration $k$,} then we can derive the recursion of decision variable:
\begin{equation}
\begin{aligned}
\hat{x}_{k+1}&=\hat{x}_k-\alpha \hat{y}_k,
\end{aligned}
\vspace{-0.2cm}
\end{equation}
where $\left( \hat{x}_k, \hat{y}_k \right) $ is an instance of $\left( \hat{X}_k, \hat{Y}_k \right) $ with $n=1$. Then, by noticing that only \textcolor{black}{the randomly selected sample nodes performs update, i.e,  $x_{s,k+1}=\hat{x}_{k+1}$, $\forall s\in \boldsymbol{s}_{k+1}$, while the other sample nodes remain unchanged,} i.e., $x_{j,k+1}=x_{j,k}, \forall j \notin  \boldsymbol{s}_{k+1}$, 
we can derive the recursion of full-gradient estimation variable:
\begin{equation}
\begin{aligned}
\hat{y}_{k+1}&=\frac{\mathbf{1}_{m}^{T}}{m}\nabla F\left( X_k \right) +S_{k+1}\left( \nabla F\left( X_{k+1} \right) -\nabla F\left( X_k \right) \right) 
\\
&=\frac{1}{m}\sum_{j=1}^m{\nabla f_j\left( x_{j,k} \right)}
\\
&+\frac{1}{b}\sum_{s\in \boldsymbol{s}_{k+1}}{\left( \nabla f_s\left( \hat{x}_{k+1} \right) -\nabla f_s\left( x_{s,k} \right) \right)},
\end{aligned}
\end{equation}

The above procedures imply that
$\nabla F\left( X_k \right)$ actually plays the role of gradient table \citep{defazio2014saga}. Thus, the original SAGA is recovered.

Different from SAGA, L-SVRG, {\citep{qian2021svrg}} reduces the stochastic gradient variance by performing full gradient update with a certain probability, it can be also recovered by SPP with the following stochastic matrix $V_k$ varies as:
\begin{equation}
\begin{aligned}
\begin{cases}
	V_k=\mathbf{I}_m,\quad\, b_k=b, \quad\,\, w.p.\quad 1-p,\\
	V_k=\mathbf{J}_m,\quad b_k=m, \quad w.p.\quad p.\\
\end{cases}
\end{aligned}
\end{equation}
Then, we obtain the recovered L-SVRG algorithm by $S_k$,
\begin{equation}
\begin{aligned}
\hat{x}_{k+1}&=\hat{x}_k-\alpha \hat{y}_k,
\\
\hat{y}_{k+1}&=\frac{1}{m}\sum_{j=1}^m{\nabla f_j\left( x_{j,t_{k+1}} \right)}
\\
&+\frac{1}{b_{k+1}}\sum_{s\in \boldsymbol{s}_{k+1}}{\left( \nabla f_s\left( \hat{x}_{k+1} \right) -\nabla f_s\left( x_{s, t_{k+1}} \right) \right)},
\end{aligned}
\end{equation}
where $t_{k+1} < k+1$ denotes the {latest} iteration before $k+1$ performing full gradient update, i.e., $b_{t_{k+1}}=m$. 

\paragraph{Recovering GT-SAGA.}
In contrast to the centralized setting, there exists extra data heterogeneity among devices ($n>1$). Gradient tracking methods are proposed to further eliminating the global data heterogeneity. We recover GT-SAGA \cite{xin2020variance} by choosing $\left(\varGamma _{k}, R_k,C_k \right)$ as defined in (\ref{Lab_Gamma_k}), (\ref{Lab_R_k}) and (\ref{Lab_C_k}) with:
\begin{equation}
\begin{aligned}
W_k=G_k=W,\,\,V_k=\mathbf{J}_m,\,\, \forall k\geqslant 1.
\end{aligned}
\end{equation}
Then, multiplying both sides of \eqref{SPP_framework} with $S_{k+1}$ and {using the fact that $S_{k+1}R_k=S_{k+1}\varGamma _k=W_kS_k$ (c.f., Lemma~\eqref{Lem_recover})}, we can obtain a reduced algorithm as follows:
\begin{equation}
\begin{aligned}
&\hat{X}_{k+1}=W\left( \hat{X}_k-\alpha \hat{Y}_k \right) ,
\\
&\hat{Y}_{k+1}={WS_k}Y_k+S_{k+1}\left( \nabla F\left( X_{k+1} \right) -\nabla F\left( X_k \right) \right) 
\\
&=W\hat{Y}_k-\left( \mathbf{I}_n\otimes \frac{\mathbf{1}_{m}^{T}}{m} \right) \left( \nabla F\left( X_{k+1} \right) -\nabla F\left( X_k \right) \right) .
\end{aligned}
\end{equation}
Further, noticing that $\nabla F\left( X_k \right)$ resembles the gradient table of all samples at iteration $k$, GT-SAGA is thus recovered under the proposed SPP framework. See Appendix~\ref{Appendix_recover} for more detailed recovery of various other algorithms.

\paragraph{New efficient algorithms.}
It should be noted that our proposed SPP framework can further recover a large family of algorithms by proper combination of existing algorithms, such as Local-SVRG, Gossip-PGA. In so doing, we can easily design some new efficient algorithms that have not been formally proposed before with provable rates, such as Local-SAGA, PGA-SAGA, PGA-GT-SAGA (see Table \ref{Tab_comparing}).

\section{Convergence Analysis}\label{Sec_Con_analysis}
\begin{table*}[h!]
    \small
	\begin{center}
		\caption{Topology design and complexity comparison for the algorithms that can be recovered and analysed by our SPP framework under \textbf{strongly convex} setting. The choices of $W_k$ (consensus), $V_k$ (variance-reduction) and $G_k$ (gradient-tracking) are presented with corresponding values of $\left( \rho _{W},r,p,q \right)$. Notice that
		``$W_k = 1$'' implies centralized scenario with only one device.
		``$\tilde{\mathcal{O}}$'' implies that the $\log \frac{1}{\varepsilon}$ factor is omitted when the sub-linear terms are dominant. The mark ``$\textbf{\textcolor{blue}{*}}$" implies that the best-known rate is recovered by our convergence analysis under the same settings;
		``\textbf{\textcolor{blue}{$\dagger$}}" denotes that the obtained rate is new under our settings.}
		\label{Tab_comparing}
	\resizebox{\textwidth}{!}{%
		\begin{tabular}{c|c|c|c|c|c|c} 
			\hline
			\rule {0pt}{10pt}
			\textbf{Algorithm} & $W_k$ & \textbf{$V_k$} & $G_k$ & $\left( \rho _{W},r,p,q \right) $ & \textbf{Obtained Complexity ($\tilde{\mathcal{O}}(\cdot)$)}
			& \textbf{Results}
			\rule {0pt}{10pt}
			\\
			\hline
			\rule {0pt}{10pt}
			SAGA & \multirow{2}{*}{1} & \multirow{2}{*}{$\mathbf{J}_m$} & \multirow{2}{*}{1}& \multirow{2}{*}{$\left\{ 0, 1, 1, \frac{b}{m} \right\} $}            &  \multirow{2}{*}{$ \left( \frac{L}{\mu}+\frac{m}{b} \right) \log \frac{1}{\varepsilon} 
				$}
			& \multirow{2}{*}{Cor.~\ref{Cor_Complexcity_2}}
				\\
			\rule {0pt}{0pt}
			{\tiny\citep{defazio2014saga}} \textbf{\textcolor{blue}{*}} & & & & & \\
			\hline
			\rule {0pt}{10pt}
			L-SVRG & \multirow{2}{*}{1} &\multirow{2}{*}{$\left\{ \mathbf{I}_m, \mathbf{J}_m \right\} $} & \multirow{2}{*}{1}  & \multirow{2}{*}{$\left\{ 0, 1, p, 1 \right\} $}  &\multirow{2}{*}{$\left( \frac{L}{\mu}+\frac{1}{p} \right) \log \frac{1}{\varepsilon}
				$}
			& \multirow{2}{*}{Cor.~\ref{Cor_Complexcity_2}}
				\\
			\rule {0pt}{0pt}
			{\tiny\citep{qian2021svrg}} \textbf{\textcolor{blue}{*}} & & & & & \\
			\hline
			\rule {0pt}{10pt}
			Local-SGD & \multirow{2}{*}{$\left\{ \mathbf{I}_n, \mathbf{J}_n \right\} $}      &\multirow{2}{*}{$\mathbf{I}_m$}     &\multirow{2}{*}{$\mathbf{I}_n$}      &\multirow{2}{*}{$\left\{ 1, r, 0, 0 \right\} $}     &\multirow{2}{*}{$ \frac{L}{\mu r}+\frac{\sigma ^*}{nb\mu ^2\varepsilon}+\sqrt{\frac{(1-r)L}{\mu ^3r\varepsilon}\left( \frac{\zeta ^*}{r}+\frac{\sigma ^*}{b} \right)}
				$} 
			& \multirow{2}{*}{Cor.~\ref{Cor_Complexcity_1}}
			\\
			\rule {0pt}{0pt}
			{\tiny  {\citep{khaled2020tighter}}} \textbf{\textcolor{blue}{*}} & & & & & \\
			\hline
			\rule {0pt}{10pt}
			DSGD &\multirow{2}{*}{ $W$}              &\multirow{2}{*}{$\mathbf{I}_m$}            &\multirow{2}{*}{$\mathbf{I}_n$}                   &\multirow{2}{*}{$\left\{ \rho _{W}, 0, 0, 0 \right\} $}                  &\multirow{2}{*}{$\frac{L}{\mu \left( 1-\rho _W \right)}+\frac{\sigma ^*}{nb\mu ^2\varepsilon}+\sqrt{\frac{\rho _WL}{\mu ^3\left( 1-\rho _W \right) \varepsilon}\left( \frac{\zeta ^*}{1-\rho _W}+\frac{\sigma ^*}{b} \right)}
				$} 
			& \multirow{2}{*}{Cor.~\ref{Cor_Complexcity_1}}
				\\
			\rule {0pt}{10pt}
			\rule {0pt}{10pt}
			{\tiny\citep{lian2017can}}\ \textbf{\textcolor{blue}{*}} & & & & & \\
			\hline
			\rule {0pt}{10pt}
			Gossip-PGA &\multirow{2}{*}{$\left\{ W, \mathbf{J}_n \right\} $}                &\multirow{2}{*}{$\mathbf{I}_m$}              &\multirow{2}{*}{$\mathbf{I}_n$}            &\multirow{2}{*}{$\left\{ \rho _{W}, r, 0, 0 \right\} $}                &\multirow{2}{*}{$\frac{L}{\mu \left( 1-\rho _{r,W} \right)}+\frac{\sigma ^*}{nb\mu ^2\varepsilon}+\sqrt{\frac{\rho _{r,W}L}{\mu ^3\left( 1-\rho _{r,W} \right) \varepsilon}\left( \frac{\zeta ^*}{1-\rho _{r,W}}+\frac{\sigma ^*}{b} \right)}
				$}          
			& \multirow{2}{*}{Cor.~\ref{Cor_Complexcity_1}}
				\\
			\rule {0pt}{10pt}
			{\tiny\citep{chen2021accelerating}} & & & & & \\
			\hline
			\rule {0pt}{10pt}
			Local-SAGA &\multirow{2}{*}{$\left\{ \mathbf{I}_n, \mathbf{J}_n \right\} $}  &\multirow{2}{*}{$\mathbf{J}_m$}              &\multirow{2}{*}{$\mathbf{I}_n$}               &\multirow{2}{*}{$\left\{1, r, 1, \frac{b}{m} \right\}$}             &\multirow{2}{*}{$\frac{L}{\mu r}+\frac{m}{b}+\sqrt{\frac{(1-r)L\zeta ^*}{\mu ^3 r^2\varepsilon}}
			$}          
			& \multirow{2}{*}{Cor.~\ref{Cor_Complexcity_2}}
			\\
			\rule {0pt}{0pt}
			{\tiny\textcolor{blue}{(New)} }& & & & & \\
			\hline
			\rule {0pt}{10pt}
			Local-SVRG &\multirow{2}{*}{$\left\{ \mathbf{I}_n, \mathbf{J}_n \right\} $}             &\multirow{2}{*}{$\left\{ \mathbf{I}_m, \mathbf{J}_m \right\} $}              &\multirow{2}{*}{$\mathbf{I}_n$}               &\multirow{2}{*}{$\left\{
					1, r, p, 1\right\} $}               &\multirow{2}{*}{ $\frac{L}{\mu r}+\frac{1}{p}+\sqrt{\frac{(1-r)L\zeta ^*}{\mu ^3 r^2\varepsilon}}
			$}        
			& \multirow{2}{*}{Cor.~\ref{Cor_Complexcity_2}}
			\\
			\rule {0pt}{10pt}
			{\tiny\citep{gorbunov2021local}} \textbf{\textcolor{blue}{$\dagger$}}& & & & &\\
			\hline
			\rule {0pt}{10pt}
			D-SAGA &\multirow{2}{*}{$W $}           &\multirow{2}{*}{$\mathbf{J}_m$}                 &\multirow{2}{*}{$\mathbf{I}_n$}                 &\multirow{2}{*}{$\left\{ \rho _{W}, 0, 1, \frac{b}{m} \right\} $}              &\multirow{2}{*}{$ \frac{L}{\mu \left( 1-\rho _W \right)}+\frac{m}{b}+\sqrt{\frac{\rho _WL\zeta ^*}{\mu ^3 \left( 1-\rho _W \right)^2\varepsilon}}
				$}         
			& \multirow{2}{*}{Cor.~\ref{Cor_Complexcity_2}}
				\\
			\rule {0pt}{10pt}
			{\tiny \citep{calauzenes2017distributed}}\textbf{\textcolor{blue}{$\dagger$}} & & & & &\\
			\hline
			\rule {0pt}{10pt}
			PGA-SAGA &\multirow{2}{*}{$\left\{ W, \mathbf{J}_n \right\} $}           &\multirow{2}{*}{$\mathbf{J}_m$}                 &\multirow{2}{*}{$\mathbf{I}_n$}                 &\multirow{2}{*}{$\left\{ \rho _{W}, r, 1, \frac{b}{m} \right\} $}              &\multirow{2}{*}{$\frac{L}{\mu \left( 1-\rho _{r,W} \right)}+\frac{m}{b}+\sqrt{\frac{\rho _{r,W}L\zeta ^*}{\mu ^3\left( 1-\rho _{r,W} \right)^2\varepsilon }}
				$}            
			& \multirow{2}{*}{Cor.~\ref{Cor_Complexcity_2}}
				\\
			\rule {0pt}{10pt}
			{\tiny\textcolor{blue}{(New)} }& & & & & \\
			\hline
			\rule {0pt}{10pt}
			GT-SAGA &\multirow{2}{*}{$W$}     &\multirow{2}{*}{$\mathbf{J}_m$}     &\multirow{2}{*}{$W$}            &\multirow{2}{*}{$\left\{ 
					\rho _{W}, 0, 1,\frac{b}{m} \right\} $}              &\multirow{2}{*}{$\left( \frac{L}{\mu \left( 1-\rho _{W} \right) ^2}+\frac{m}{b} \right) \log \frac{1}{\varepsilon}
			$}      
			& \multirow{2}{*}{Cor.~\ref{Cor_Complexcity_3}}
			\\
			\rule {0pt}{0pt}
			{\tiny\citep{xin2020variance}} \textbf{\textcolor{blue}{$\dagger$}} & & & & & \\
			\hline
			\rule {0pt}{10pt}
			PGA-GT-SAGA &\multirow{2}{*}{$\left\{ W,\mathbf{J}_n \right\} $}             &\multirow{2}{*}{$\mathbf{J}_m$}                 &\multirow{2}{*}{$W_k$}              &\multirow{2}{*}{$\left\{ 
					\rho _{W}, r, p, \frac{b}{m}
				 \right\} $}             &\multirow{2}{*}{$\left( \frac{L}{\mu \left( 1-\rho _{r,W} \right) ^2}+\frac{m}{b} \right) \log \frac{1}{\varepsilon}
			$}          
			& \multirow{2}{*}{Cor.~\ref{Cor_Complexcity_3}}
			\\
		\rule {0pt}{10pt}
		{\tiny\textcolor{blue}{(New)} } & & & & & \\
			\hline
		\end{tabular}		}
	\end{center}
	\vspace{-0.5cm}
\end{table*}

\hy{In this section, we provide a unified convergence analysis of the proposed SPP framework for both strongly convex and general convex cases. To this end, we first define $p:=P\left( V_k=\mathbf{J}_m \right)$ as the probability of performing local variance-reduction; $q:=\mathbb{E}\left[ b_k/m|V_k=\mathbf{J}_m \right]$ the expected ratio of batch-size while performing variance-reduction; $r:=P\left( W_k=\mathbf{J}_n \right)$ the probability of adopting global averaging. In order to characterize the algorithm to be incorporated into the proposed framework, we also need to specify the non-negative matrices $W_k, G_k, V_k$ which correspond to mixing, gradient-tracking and variance-reduction, respectively, as given by the following assumption.}

\begin{Ass}\label{Ass_algoirthm}
{ The non-negative matrices $W_k, G_k, V_k$  are independently and randomly chosen as
\begin{equation*}
W_k\in \left\{ W, \mathbf{J}_n \right\} ,\,\,
G_k\in \left\{ \mathbf{I}_n, W_k \right\} , \,\,
V_k\in \left\{ \mathbf{I}_m, \mathbf{J}_m \right\},
\end{equation*}
for all $ k\geqslant 0$, and $W_k$ is doubly stochastic, i.e., $\mathbf{1}_{n}^{T}W_k=\mathbf{1}_{n}^{T}, W_k\mathbf{1}_n=\mathbf{1}_n$, and satisfies
\begin{equation}
\rho _{r,W}:=\mathbb{E}\left[ \left\| W_k-\mathbf{J}_n \right\| _{2}^{2} \right] =\left( 1-r \right) \rho _W<1,
\end{equation}
where $\rho _W:=\left\| W-\mathbf{J}_n \right\|_2^2$.
}
\end{Ass}

\begin{Rem}
{Assumption \ref{Ass_algoirthm} does not require a contraction property of $W_k$ at each iteration but in the sense of expectation. For instance, $W_k$ in Local-SGD can be randomly or periodically chosen from $\left\{ \mathbf{I}_n, \mathbf{J}_n \right\}$, or simply set as $W_k=W$ with $\rho _W<1$ for $k\geqslant 0$ in DSGD (c.f., Appendix \ref{Appendix_recover}).
}
\end{Rem}

\textbf{Main results}. For convergence analysis, we define the average variables of $X_k$ and $Y_k$ as follows:
\begin{equation}\label{Def_x_y_bar}
\begin{aligned}
&\bar{x}_k:=\frac{\mathbf{1}_n^T}{n}S_kX_k,\quad
\bar{y}_k:=\frac{\mathbf{1}_n^T}{n}S_kY_k.
\end{aligned}
\end{equation}
Then, we are ready to establish the convergence rates of SPP for smooth and strongly convex objective functions under both centralized and distributed settings. To this end, we first construct a general Lyapunov function consisting of several error terms as follows:

\vspace{-0.5cm}
\begin{equation}\label{Lyapunov_func}
\begin{aligned}
T_{k+1}&:=c_0\underset{\mathrm{optimal}\,\,\mathrm{gap}}{\underbrace{\left\| \bar{x}_{k+1}-x^* \right\| ^2}}+c_1\underset{\mathrm{consensus} \,\, \mathrm{error}}{\underbrace{\left\| \hat{X}_{k+1}-\mathbf{1}_n\bar{x}_{k+1} \right\| ^2}}
\\
&+c_2\underset{\mathrm{delayed} \,\,\mathrm{VR}\,\, \mathrm{error}}{\underbrace{\left\| \nabla F\left( X_{t_k} \right) -\nabla F\left( \mathbf{1}_Mx^* \right) \right\| ^2}}
\\
&+c_3\underset{\mathrm{VR}\,\, \mathrm{error}}{\underbrace{\left\| \nabla F\left( X_k \right) -\nabla F\left( \mathbf{1}_Mx^* \right) \right\| ^2}}
\\
&+c_4\underset{\mathrm{GT}\,\, \mathrm{error}}{\underbrace{\left\| \hat{Y}_{k+1}-\mathbf{1}_n\bar{y}_{k+1} \right\| ^2}},
\end{aligned}
\vspace{-0.1cm}
\end{equation}

where $c_0, c_1, c_2, c_3, c_4 \geqslant 0$ are constants to be properly determined (see Appendix~\ref{tech_prelim} for more details). {Note that the vanishing of the Lyapunov function implies the attainment of the optimum for strongly convex objective functions.}

Now, we proceed to present our main results\footnote{All proofs can be found in Appendix \ref{Appe_Con_analysis}} for algorithms $\mathcal{A}\left(\varGamma _{k}, R_k,C_k \right)$ under different parameter settings and we summarize their complexity in Table~{\ref{Tab_comparing}} for  comparison.

\begin{Thm}\label{Thm_Without_GT_VR}
Consider algorithms $\mathcal{A}\left( \cdot , \cdot , C_k\equiv \mathbf{I}_M \right)$ generated from the SPP framework with a constant batch-size of $b$. Suppose Assumption \ref{Ass_smoothness}-\ref{Ass_algoirthm} hold and $\mu >0$. Let
\[
c_0=1,\,\,c_1=\left( 1-r \right) \frac{8\alpha L\left( 4\alpha L+1 \right)}{n\left( 1-\rho _{r,W} \right)},c_2=c_3=c_4=0
\]
and the step-size satisfy 
\begin{equation}\label{Step_size_WO}
\begin{aligned}
{\alpha = \min \left\{ \mathcal{O}\left( \frac{1}{L} \right) , \mathcal{O}\left( \frac{1-\rho _{r,W}}{L\sqrt{\rho _{r,W}}} \right) \right\}.}
\end{aligned}
\end{equation}
Then, we have for all $k\geqslant 0$
\begin{equation}
\begin{aligned}
\mathbb{E}\left[ T_{k+1} \right] &\leqslant \left(1-\min \left\{ \alpha \mu ,\frac{1-\rho _{r,W}}{8} \right\}\right) \mathbb{E}\left[ T_k \right] +\frac{2\alpha ^2\sigma ^*}{nb}
\\
&+\left( 1-r \right) \frac{16\alpha ^3L\rho _{r,W}}{1-\rho _{r,W}}\left( \frac{4\zeta ^*}{1-\rho _{r,W}}+\frac{\sigma ^*}{b} \right). 
\end{aligned}
\end{equation}
\end{Thm}

\begin{Rem}
Theorem \ref{Thm_Without_GT_VR} yields a known tightest convergence rate as \citet{koloskova2020unified} for a class of decentralized optimization algorithms without adopting VR or GT schemes, i.e., Local-SGD, DSGD (see Table.~{\ref{Tab_comparing}}).
\end{Rem}

\begin{Cor}\label{Cor_Complexcity_1}
Under the conditions in Theorem \ref{Thm_Without_GT_VR},
there exists a suitable upper-bound of step-size $\alpha$ (refer to supplementary) such that
{$\mathbb{E}\left[ T_K \right] \leqslant \varepsilon $} after at most the following number of iterations $K$:
\begin{equation}
\begin{aligned}
&K\geqslant {\mathcal{O}}\left( \frac{L}{\mu \left( 1-\rho _{r,W} \right)}\log \frac{\mathbb{E}\left[ T_0 \right]}{\varepsilon} \right) 
\\
&+{\mathcal{O}}\left( \frac{\sigma ^*}{n\mu ^2\varepsilon}+\sqrt{\frac{{\rho _{r,W}}L}{\mu ^3\left( 1-\rho _{r,W} \right) \varepsilon}\left( \frac{\zeta ^*}{1-\rho _{r,W}}+\frac{\sigma ^*}{b} \right)} \right).
\end{aligned}
\end{equation}
\end{Cor}
Then, we provide the convergence analysis for algorithms that adopt VR schemes, such as Local-SVRG, Local-SAGA (new) and D-SAGA, in order to eliminate the internal variance $\sigma^*$ due to  gradient sampling.
\begin{Thm}\label{Thm_With_VR_only}
Consider algorithms $\mathcal{A}\left( \cdot ,\cdot ,C_k\equiv \mathbf{I}_n\otimes V_k \right)$ generated from the SPP framework with $p>0$. Suppose Assumption \ref{Ass_smoothness}-\ref{Ass_algoirthm} hold and $\mu>0$. Let
\begin{equation*}
\begin{aligned}
c_0=1, c_1=\frac{20L\alpha}{n\left( 1-\rho _{r,W} \right)}, c_2=\frac{5\alpha ^2}{Mp},  c_3=\frac{16\alpha ^2}{Mq}, c_4=0
\end{aligned}
\end{equation*}
and the step-size satisfy 
\begin{equation}\label{Step_size_VR}
\begin{aligned}
{\alpha = \mathcal{O}\left( \frac{1-\rho _{r,W}}{L} \right).}
\end{aligned}
\end{equation}
Then, we have for all $k\geqslant 0$
\begin{equation}
\begin{aligned}
\mathbb{E}\left[ T_{k+1} \right] &\leqslant \left( 1-\min \left\{ \alpha \mu , \frac{pq}{2}, \frac{1-\rho _{r,W}}{8} \right\}\right) \mathbb{E}\left[ T_k \right] 
\\
&+\frac{80\alpha ^3L\rho _{r,W}}{\left( 1-\rho _{r,W} \right) ^2}\zeta^*.
\end{aligned}
\end{equation}

\end{Thm}

\begin{Cor}\label{Cor_Complexcity_2}
Under the conditions in Theorem \ref{Thm_With_VR_only}, there exists a suitable upper-bound of step-size $\alpha$ (refer to supplementary) such that we have $\mathbb{E}\left[ T_K \right] \leqslant \varepsilon $ after at most the number of iterations $K$:

\vspace{-0.2cm}
\begin{equation}
\begin{aligned}
K&\geqslant \mathcal{O}\left( \left( \frac{1}{pq}+\frac{L}{\mu \left( 1-\rho _{r,W} \right)} \right) \log \frac{\mathbb{E}\left[ T_0 \right]}{\varepsilon} \right) 
\\
&+\mathcal{O}\left( \sqrt{\frac{{\rho _{r,W}}L\zeta ^*}{\left( 1-\rho _{r,W} \right) ^2\mu ^3\varepsilon}} \right).
\end{aligned}
\end{equation}
\end{Cor}

Finally, we provide the convergence analysis for algorithms adopting both VR and GT schemes, such as GT-SAGA and PGA-GT-SAGA (new) which are capable of removing both the internal ($\sigma^*$) and external variance ($\zeta^*$) induced by the data heterogeneity within and across devices.

\begin{Thm} \label{Thm_With_GT_VR}
Consider algorithms $\mathcal{A}\left( \cdot ,\cdot ,C_k = W_k\otimes \mathbf{J}_m \right)$ generated from the SPP framework. Suppose Assumption \ref{Ass_smoothness}-\ref{Ass_algoirthm} hold and $\mu>0$. Let
\begin{equation*}
\begin{aligned}
&c_0=1, \,\,c_1=\frac{1-\rho _{r,W}}{n\rho _{r,W}\left( 1+\rho _{r,W} \right)}, \,\,c_2=0, 
\\
&c_3=\frac{20\alpha ^2}{Mq\left( 1-\rho _{r,W} \right) ^2}, \,\,c_4=\frac{8\alpha ^2}{n\left( 1-\rho _{r,W} \right)}
\end{aligned}
\end{equation*}
and the step-size satisfy 
\begin{equation} \label{Step_size_WO_GT_VR}
\begin{aligned}
\alpha = \mathcal{O}\left( \frac{\left( 1-\rho _{r,W} \right) ^2}{L} \right).
\end{aligned}
\end{equation}
Then, we have for all $k\geqslant 0$
\begin{equation}
\begin{aligned}
\mathbb{E}\left[ T_{k+1} \right] \leqslant \left(1-\min \left\{ \alpha \mu ,\,\, \frac{q}{2},\,\, \frac{1-\rho _{r,W}}{8} \right\}\right) \mathbb{E}\left[ T_k \right] .
\end{aligned}
\end{equation}

\end{Thm}

\begin{Cor}\label{Cor_Complexcity_3}
Under the same conditions in Theorem \ref{Thm_With_GT_VR}, we have $\mathbb{E}\left[ T_K \right] \leqslant \varepsilon $ after at most the number of iterations $K$:
\begin{equation}
\begin{aligned}
K\geqslant {\mathcal{O}}\left( \left( \frac{L}{\mu \left( 1-\rho _{r,W} \right) ^2}+\frac{1}{q} \right) \log \frac{\mathbb{E}\left[ T_0 \right]}{\varepsilon} \right). 
\end{aligned}
\end{equation}
\end{Cor}

\begin{Rem}
It follows from the above theorems that one can always effectively remove the data heterogeneity within and across devices by properly choosing VR and GT schemes, respectively. Besides, the above results also establish a clear dependency of the complexity on the parameters related to network, cost functions and algorithm design.
\end{Rem}

\begin{Rem}\label{Rem_tabel}
In Table~{\ref{Tab_comparing}}, we provide the complexity results for several well known existing and some new algorithms for smooth and strongly convex objectives under our proposed SPP framework.
With proper choices of $W_k$, $V_k$ and $G_k$,
the complexity induced by Theorem~{\ref{Thm_Without_GT_VR}} yields a best-known convergence rate for Local-SGD and DSGD , i.e., matching the results in {\citep{koloskova2020unified}}.
We also recover the complexity results of SAGA and L-SVRG in the centralized scenario from Theorem~{\ref{Thm_With_VR_only}}; 
{Besides, we enhance the result of GT-SAGA {\citep{xin2020variance}} in that the dependency on the condition number of objectives is improved from $(L/\mu)^2$ to $L/\mu$ (c.f., Th.~{\ref{Thm_With_GT_VR}}) in our slightly stronger settings.} 
We also provide several new algorithms with provable rates, such as Local-SAGA, PGA-SAGA and PGA-GT-SAGA, which have not been formally proposed and analyzed yet.
\end{Rem}

\begin{table*}[t]
    \small
	\begin{center}
		\caption{Topology design and complexity comparison for the algorithms that can be recovered and analysed by our SPP framework under \textbf{convex} settings ($\mu=0$). The choices of $W_k$ (consensus), $V_k$ (variance-reduction) and $G_k$ (gradient-tracking) are presented with corresponding values of $\left( \rho _{W},r,p,q \right)$. Notice that
		``$W_k = 1$'' implies centralized scenario with only one device. The mark ``$\textbf{\textcolor{blue}{*}}$" implies that the best-known rate is recovered by our analysis under the same settings;
		``\textbf{\textcolor{blue}{$\dagger$}}" denotes that the obtained rate is new under our settings;
		$\mathcal{O}(\cdot)$ hides 
	    the initial error $\left\| x_0-x^* \right\| ^2$.}
		\label{Tab_comparing_convex}
\resizebox{\textwidth}{!}{%
		\begin{tabular}{c|c|c|c|c|c|c} 
			\hline
			\rule {0pt}{10pt}
			\textbf{Algorithm} & $W_k$ & \textbf{$V_k$} & $G_k$ & $\left( \rho _{W},r,p,q \right) $ & \textbf{Obtained Complexity ($\mathcal{O}(\cdot)$)} & \textbf{Results}
			\rule {0pt}{10pt}
			
			\\
			\hline
			\rule {0pt}{10pt}
			SAGA & \multirow{2}{*}{1} & \multirow{2}{*}{$\mathbf{J}_m$} & \multirow{2}{*}{1}& \multirow{2}{*}{$\left\{ 0, 1, 1, \frac{b}{m} \right\} $}            &  \multirow{2}{*}{$  \frac{mL}{b\varepsilon}
				$} 
			& \multirow{2}{*}{Cor.~\ref{Cor_5}}
				\\
			\rule {0pt}{0pt}
			{\tiny\citep{defazio2014saga}} \textbf{\textcolor{blue}{*}} & & & & & \\
			\hline
			\rule {0pt}{10pt}
			L-SVRG & \multirow{2}{*}{1} &\multirow{2}{*}{$\left\{ \mathbf{I}_m, \mathbf{J}_m \right\} $} & \multirow{2}{*}{1}  & \multirow{2}{*}{$\left\{ 0, 1, p, 1 \right\} $}  &\multirow{2}{*}{$ \frac{L}{p\varepsilon}
				$}  
			& \multirow{2}{*}{Cor.~\ref{Cor_5}}
				\\
			\rule {0pt}{10pt}
			{\tiny\citep{qian2021svrg}} \textbf{\textcolor{blue}{*}} & & & & & \\
			\hline
			\rule {0pt}{10pt}
			Local-SGD & \multirow{2}{*}{$\left\{ \mathbf{I}_n, \mathbf{J}_n \right\} $}      &\multirow{2}{*}{$\mathbf{I}_m$}     &\multirow{2}{*}{$\mathbf{I}_n$}      &\multirow{2}{*}{$\left\{ 1, r, 0, 0 \right\} $}     &\multirow{2}{*}{$\frac{L}{r\varepsilon}+\frac{\sigma ^*}{nb\varepsilon ^2}+\frac{1}{\varepsilon ^{3/2}}\sqrt{\frac{(1-r)L}{r}\left( \frac{\zeta ^*}{r}+\frac{\sigma ^*}{b} \right)}
				$}   
			& \multirow{2}{*}{Cor.~\ref{Cor_4}}	
			\\
			\rule {0pt}{10pt}
			{\tiny\citep{khaled2020tighter}} \textbf{\textcolor{blue}{*}} & & & & & \\
			\hline
			\rule {0pt}{10pt}
			DSGD &\multirow{2}{*}{ $W$}              &\multirow{2}{*}{$\mathbf{I}_m$}            &\multirow{2}{*}{$\mathbf{I}_n$}                   &\multirow{2}{*}{$\left\{ \rho _{W}, 0, 0, 0 \right\} $}                  &\multirow{2}{*}{$\frac{L}{\left( 1-\rho _W \right) \varepsilon}+\frac{\sigma ^*}{nb\varepsilon ^2}+\frac{1}{\varepsilon ^{3/2}}\sqrt{\frac{\rho _W L}{1-\rho _W}\left( \frac{\zeta ^*}{1-\rho _W}+\frac{\sigma ^*}{b} \right)}
				$}        
			& \multirow{2}{*}{Cor.~\ref{Cor_4}}		
			\\
			\rule {0pt}{10pt}
			\rule {0pt}{10pt}
			{\tiny\citep{lian2017can}}\ \textbf{\textcolor{blue}{*}} & & & & & \\
			\hline
			\rule {0pt}{10pt}
			Gossip-PGA &\multirow{2}{*}{$\left\{ W, \mathbf{J}_n \right\} $}                &\multirow{2}{*}{$\mathbf{I}_m$}              &\multirow{2}{*}{$\mathbf{I}_n$}            &\multirow{2}{*}{$\left\{ \rho _{W}, r, 0, 0 \right\} $}                &\multirow{2}{*}{$\frac{L}{\left( 1-\rho _{r,W} \right) \varepsilon}+\frac{\sigma ^*}{nb\varepsilon ^2}+\frac{1}{\varepsilon ^{3/2}}\sqrt{\frac{\rho _{r,W} L}{1-\rho _{r,W}}\left( \frac{\zeta ^*}{1-\rho _{r,W}}+\frac{\sigma ^*}{b} \right)}
				$}          
			& \multirow{2}{*}{Cor.~\ref{Cor_4}}		
			\\
			\rule {0pt}{10pt}
			{\tiny\citep{chen2021accelerating}} & & & & & \\
			\hline
			\rule {0pt}{10pt}
			Local-SAGA &\multirow{2}{*}{$\left\{ \mathbf{I}_n, \mathbf{J}_n \right\} $}  &\multirow{2}{*}{$\mathbf{J}_m$}              &\multirow{2}{*}{$\mathbf{I}_n$}               &\multirow{2}{*}{$\left\{1, r, 1, \frac{b}{m} \right\}$}             &\multirow{2}{*}{$\frac{mL}{br\varepsilon}+\frac{m\sqrt{(1-r)L\zeta^*}}{br\varepsilon ^{3/2}}
			$}         
			& \multirow{2}{*}{Cor.~\ref{Cor_5}}	
			\\
			\rule {0pt}{0pt}
			{\tiny\textcolor{blue}{(New)} }& & & & & \\
			\hline
			\rule {0pt}{10pt}
			Local-SVRG &\multirow{2}{*}{$\left\{ \mathbf{I}_n, \mathbf{J}_n \right\} $}             &\multirow{2}{*}{$\left\{ \mathbf{I}_m, \mathbf{J}_m \right\} $}              &\multirow{2}{*}{$\mathbf{I}_n$}               &\multirow{2}{*}{$\left\{
					1, r, p, 1\right\} $}               &\multirow{2}{*}{ $\frac{L}{rp\varepsilon}+\frac{\sqrt{(1-r)L\zeta ^*}}{rp\varepsilon ^{3/2}}
			$}         
			& \multirow{2}{*}{Cor.~\ref{Cor_5}}	
			\\
			\rule {0pt}{0pt}
			{\tiny\citep{gorbunov2021local}} \textcolor{blue}{$\dagger$} & & & & &\\
			\hline
			\rule {0pt}{10pt}
			D-SAGA &\multirow{2}{*}{$W $}           &\multirow{2}{*}{$\mathbf{J}_m$}                 &\multirow{2}{*}{$\mathbf{I}_n$}                 &\multirow{2}{*}{$\left\{ \rho _{W}, r, 1, \frac{b}{m} \right\} $}              &\multirow{2}{*}{$\frac{mL}{b\left( 1-\rho _W \right) \varepsilon}+\frac{m\sqrt{\rho _WL\zeta^*}}{b\left( 1-\rho _W \right) \varepsilon ^{3/2}}
				$}              
			& \multirow{2}{*}{Cor.~\ref{Cor_5}}		
			\\
			\rule {0pt}{10pt}
			{\tiny\citep{calauzenes2017distributed}}\textbf{\textcolor{blue}{$\dagger$}} & & & & &\\
			\hline
			\rule {0pt}{10pt}
			PGA-SAGA &\multirow{2}{*}{$\left\{ W, \mathbf{J}_n \right\} $}           &\multirow{2}{*}{$\mathbf{J}_m$}                 &\multirow{2}{*}{$\mathbf{I}_n$}                 &\multirow{2}{*}{$\left\{ \rho _{W}, r, 1, \frac{b}{m} \right\} $}              &\multirow{2}{*}{$\frac{mL}{b\left( 1-\rho _{r,W} \right) \varepsilon}+\frac{m\sqrt{\rho _{r,W}L\zeta^*}}{b\left( 1-\rho _{r,W} \right) \varepsilon ^{3/2}}
				$}               
			& \multirow{2}{*}{Cor.~\ref{Cor_5}}		
			\\
			\rule {0pt}{10pt}
			{\tiny\textcolor{blue}{(New)} }& & & & & \\
			\hline
			\rule {0pt}{10pt}
			GT-SAGA &\multirow{2}{*}{$W$}     &\multirow{2}{*}{$\mathbf{J}_m$}     &\multirow{2}{*}{$W$}            &\multirow{2}{*}{$\left\{ 
					\rho _{W}, 0, 1,\frac{b}{m} \right\} $}              &\multirow{2}{*}{$\frac{mL}{b\left( 1-\rho _W \right) ^2\varepsilon}
			$}       
			& \multirow{2}{*}{Cor.~\ref{Cor_6}}	
			\\
			\rule {0pt}{10pt}
			{\tiny\citep{xin2020variance}} \textbf{\textcolor{blue}{$\dagger$}} & & & & & \\
			\hline
			\rule {0pt}{10pt}
			PGA-GT-SAGA &\multirow{2}{*}{$\left\{ W,\mathbf{J}_n \right\} $}             &\multirow{2}{*}{$\mathbf{J}_m$}                 &\multirow{2}{*}{$W_k$}              &\multirow{2}{*}{$\left\{ 
					\rho _{W}, r, p, \frac{b}{m}
				 \right\} $}             &\multirow{2}{*}{$\frac{mL}{b\left( 1-\rho _{r,W} \right) ^2\varepsilon}
			$}          
			& \multirow{2}{*}{Cor.~\ref{Cor_6}}	
			\\
		\rule {0pt}{10pt}
		{\tiny\textcolor{blue}{(New)} } & & & & & \\
			\hline
		\end{tabular}
		}
	\end{center}
\vspace{-0.3cm}
\end{table*}

\tocheck{For general convex cases ($\mu = 0$), we provide the obtained complexity of the algorithms that can be recovered and analysed by the SPP framework in Table~\ref{Tab_comparing_convex}. In particular, we can also recover the best-known sub-linear rates of SAGA, L-SVRG, DSGD and Local-SGD. Moreover, we establish the sublinear rate of GT-SAGA in general convex settings, which is not yet considered in {\citet{xin2020variance}} (c.f., Theorem~{\ref{Thm_4}}-{\ref{Thm_6}}).
All proofs can be found in Appendix \ref{App_sublinear_results}. 
}

\section{Experimental Results}\label{Sec_expe}
In this section, we report some experiments to verify the theoretical findings for the proposed framework under different settings of data heterogeneity and topology\footnote{More experimental results can be found in Appendix \ref{Appe_add_experiments}.}. 

\begin{table}[ht]
    \scriptsize
 \begin{center}
    \vspace{-0.2cm}
  \caption{Summary of the experimental setup.}
  \label{Exp_Setting}
  \begin{tabular}{|c|c|c|c|c|c|} 
   \hline
   \rule {0pt}{10pt}
   \textbf{Dataset} & \textbf{Node} ($n$) & \#\textbf{Train} & \#\textbf{Test} & \textbf{BS ($n\times b$)} & \textbf{SS ($\alpha$)}
   \rule {0pt}{10pt}
   \\
   \hline
   \rule {0pt}{10pt}
   {F-MNIST} & \{8, 50\} & 60000 & 10000 & 200 &  0.05 \\
   \hline
   \rule {0pt}{10pt}
   CIFAR-10 & \{8, 50\} & 50000 & 10000 & 400 &  0.008 \\
   \hline
  \end{tabular}
 \end{center}
 {\quad BS: Batch-size;  SS: Step-size.}
 \vspace{-0.5cm}
\end{table}

\textbf{Experiment settings.} 
We train a regularized logistic regression classifier on both CIFAR-10 and Fashion-MNIST (F-MNIST) datasets over a network of $n$ nodes each of which locally stores $m$ data samples, which reads:
\begin{equation}\label{prob:logistic}
\begin{aligned}
\vspace{-0.2cm}
\underset{x\in \mathbb{R}^{d}}{\min}f\left( x \right) :=\frac{1}{n}\sum_{i=1}^n{\frac{1}{m}\sum_{j=1}^m{\underset{f_{ij}}{\left( \underbrace{\ell \left( x,\xi _{ij} \right) +\frac{\lambda}{2}\left\| x \right\| ^2} \right)}}}
\end{aligned}
\vspace{-0.2cm}
\end{equation}
with the cross-entropy loss $\ell$ defined as:
\begin{equation}\label{cross-entropy loss}
\begin{aligned}
\vspace{-0.3cm}
\ell \left( x,\xi _{ij} \right) :=-\sum_{c=1}^{10}{\phi _{ij}^{c}}\log \left( 1+\exp \left( -x^T\theta _{i,j} \right) \right) ^{-1},
\vspace{-0.3cm}
\end{aligned}
\end{equation}
where $\lambda=0.001$ is the regularization parameter, $\theta _{i,j} \in \mathbb{R}^d$ and $\phi _{ij}^{c}\in \left\{ -1, 1 \right\}$ is the feature vector and the label of class $c\in[10]$ associated with sample  $\xi _{ij}$, respectively. 
The datasets and parameters we used are summarized in Table~\ref{Exp_Setting} with $r=0.05$ for certain algorithms with global averaging.

\textbf{Data heterogeneity.} \textcolor{black}{We consider datasets with unbalanced label distribution, which is an important type of data heterogeneity for  classification problems in distributed settings \citep{hsieh2020non}, and control the level of data heterogeneity through an arithmetic sequence with a difference of $h$ to allocate training samples of each class (c.f., Eq.~(\ref{Eq_data_splite_1}, \ref{Eq_data_splite_2}) in Appendix~\ref{Appe_add_experiments}).
Fig.~\ref{Expe} plots the testing accuracy of the algorithms to be compared under different settings of heterogeneous label distributions: i) independent and identically (label) distributed (IID) datasets on the top row, i.e., $h=0$; ii) non-IID datasets with $h=124$ on the middle row; iii) A highly unbalanced label distribution denoted by $h=h_{\max}$ on the bottom row (c.f., Table \ref{data_split_h20}).} \finetune{It follows from Fig.~\ref{Expe} that, when the labels are evenly distributed (top row), the VR-based algorithms (such as SAGA, D-SAGA, Local-SAGA, denoted in solid lines) outperform the others without VR schemes. Furthermore, when the level of data heterogeneity increases (middle and bottom rows), GT-SAGA and PGA-GT-SAGA that adopt GT schemes can maintain relatively high testing accuracy while those without GT will degrade dramatically, especially for those that also do not employ VR schemes, which implies that GT-based methods are more robust against the data heterogeneity. The above experimental results verify the effectiveness of adopting VR and GT schemes in scenarios where datasets are heterogeneous, which corroborates the Theorem~\ref{Thm_Without_GT_VR}-\ref{Thm_With_GT_VR}. }
\begin{figure}[h]
\vspace{-0.2cm}
    \centering
    \subfigure 
    {
        \begin{minipage}[t]{0.22\textwidth}
            \centering          
            \includegraphics[width=\textwidth]{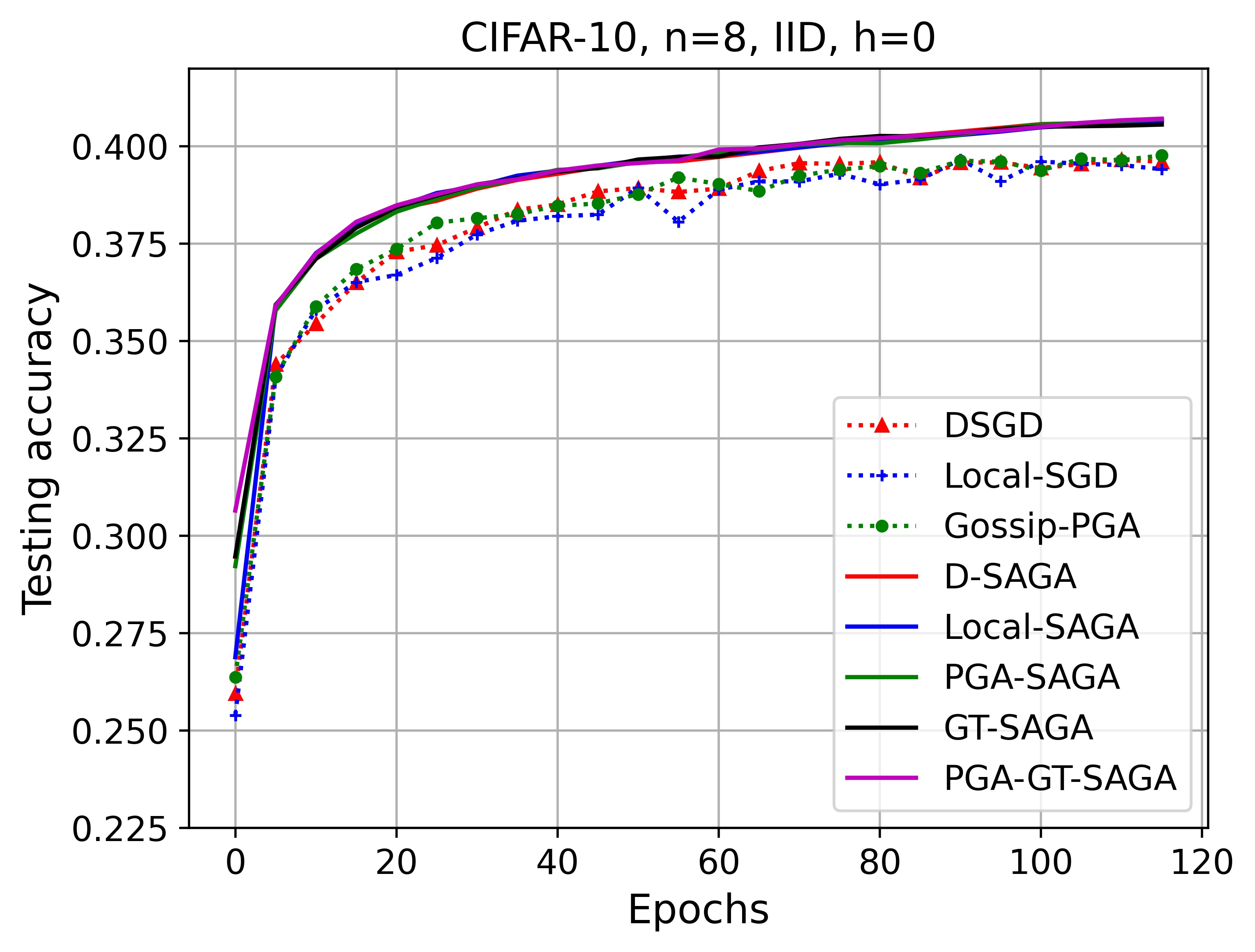}   
        \end{minipage}%
    }
    \subfigure 
    {
        \begin{minipage}[t]{0.22\textwidth}
            \centering      
            \includegraphics[width=\textwidth]{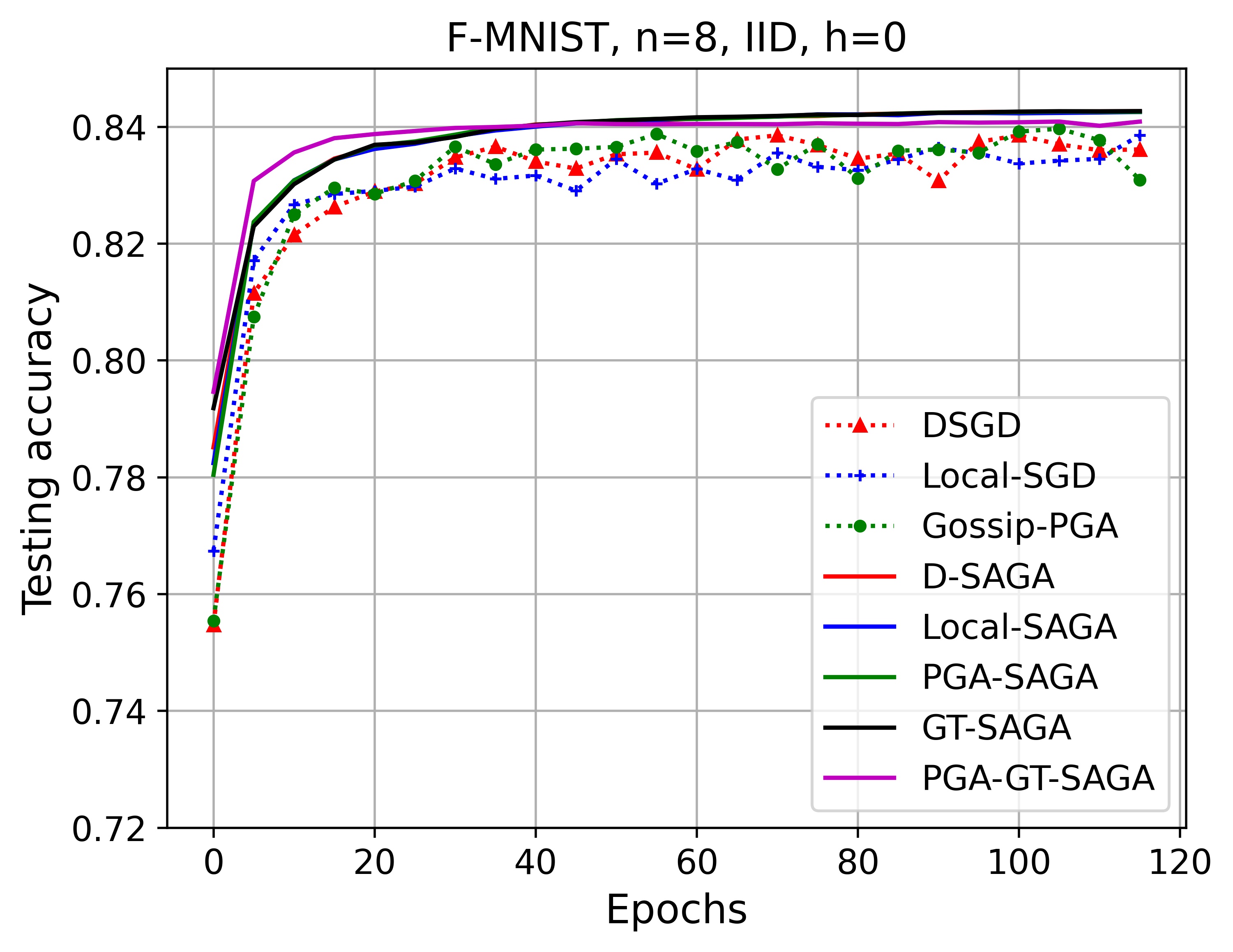}   
        \end{minipage}
    }\\ \vspace{-10pt}%
    \subfigure 
    {
    \begin{minipage}[t]{0.22\textwidth}
            \centering      
            \includegraphics[width=\textwidth]{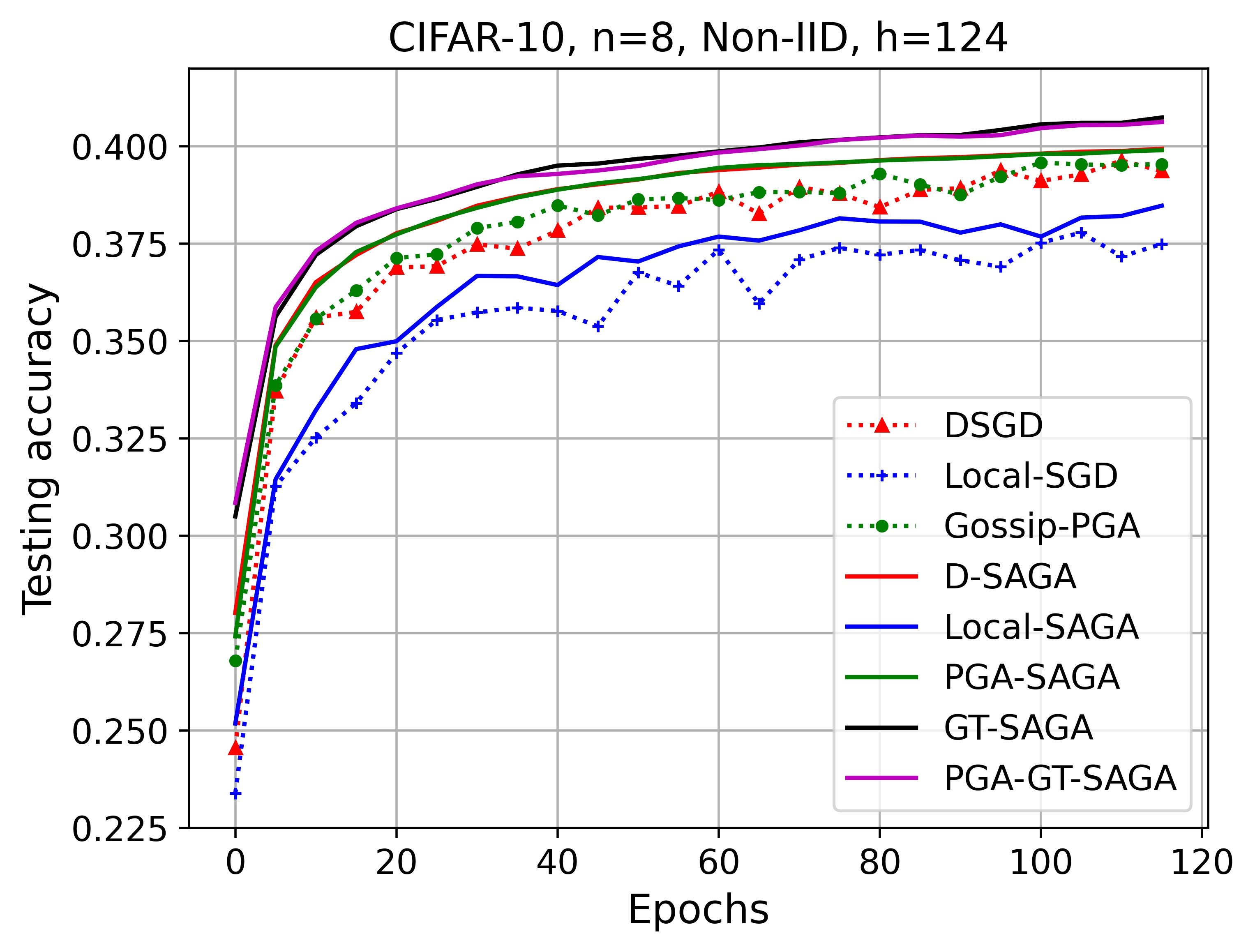}   
        \end{minipage}
    }%
    \subfigure 
    {
        \begin{minipage}[t]{0.22\textwidth}
            \centering      
            \includegraphics[width=\textwidth]{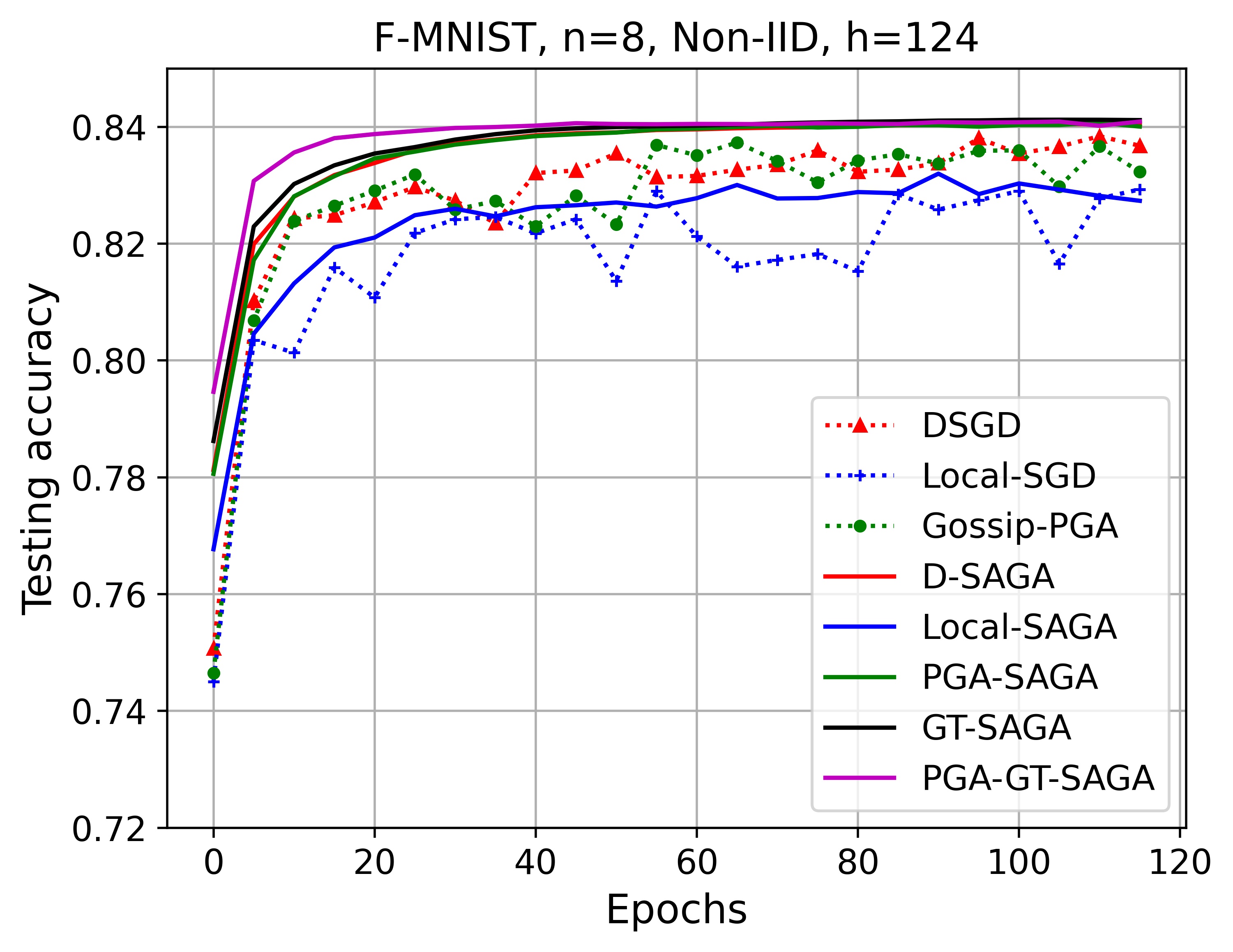}   
        \end{minipage}
    }\\ \vspace{-10pt}%
    \subfigure 
    {
    \begin{minipage}[t]{0.22\textwidth}
            \centering      
            \includegraphics[width=\textwidth]{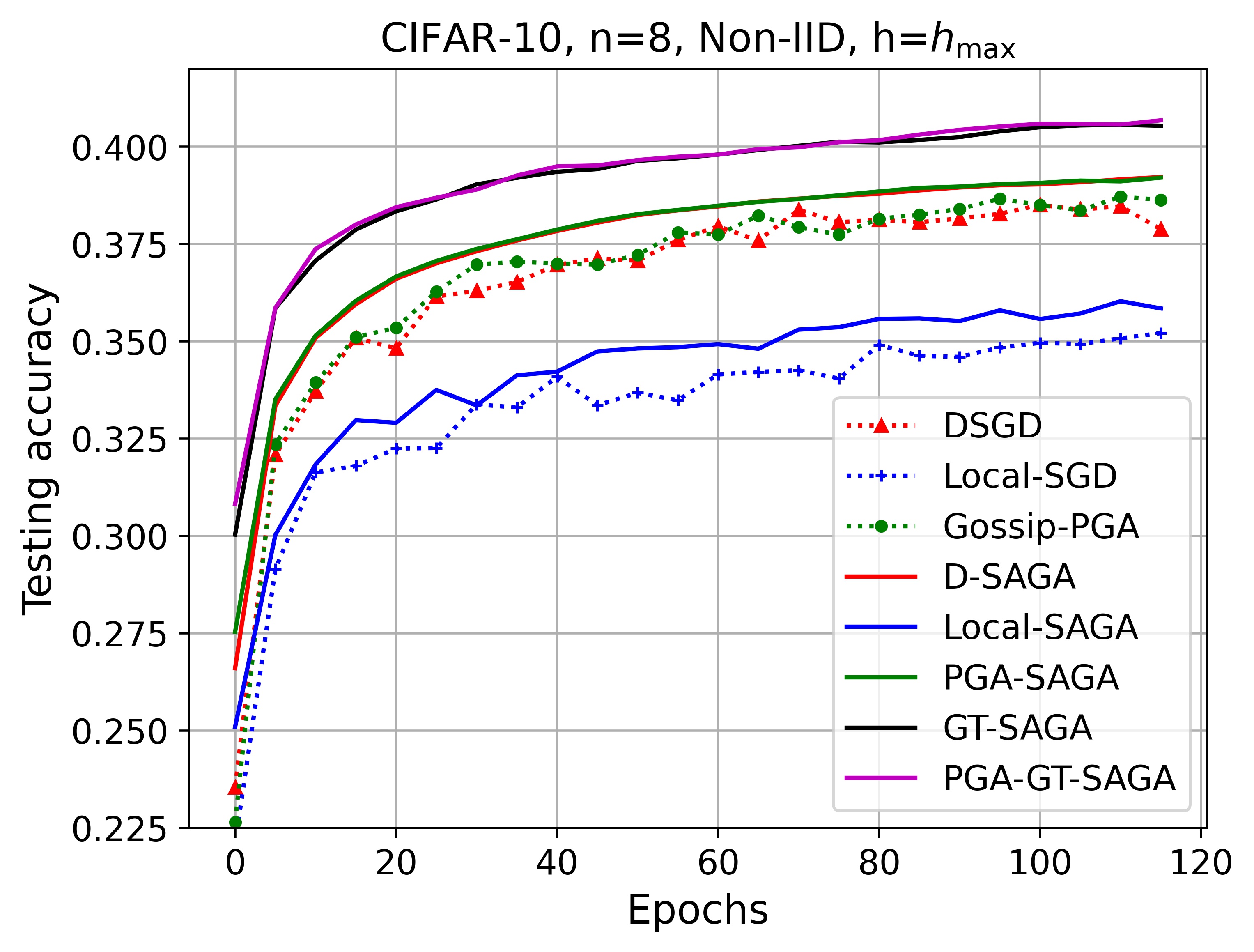}   
        \end{minipage}
    }%
    \subfigure 
    {
        \begin{minipage}[t]{0.22\textwidth}
            \centering      
            \includegraphics[width=\textwidth]{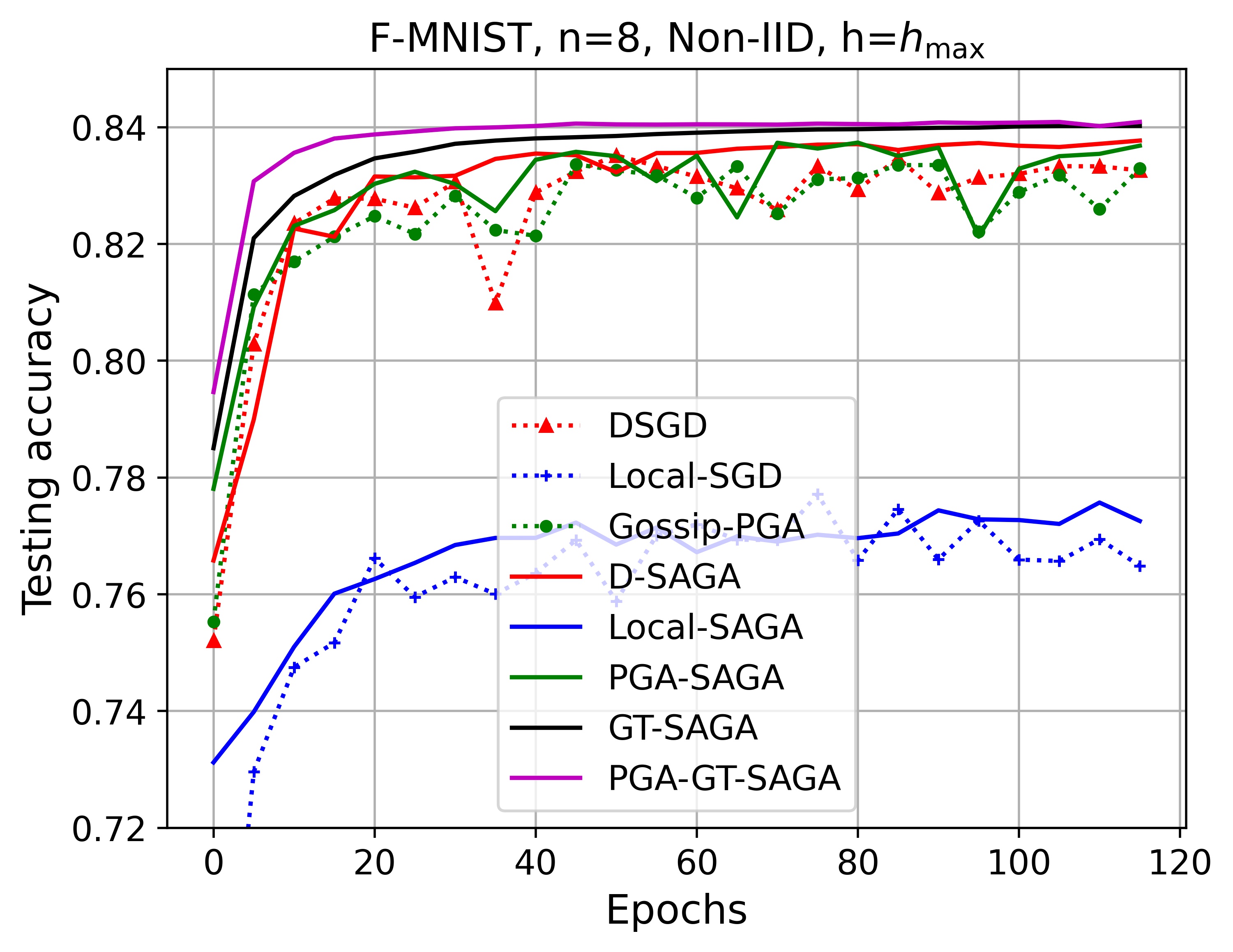}   
        \end{minipage}
    }%
    \vspace{-0.25cm}
    \caption{Performance comparison of DSGD, Local-SGD, Gossip-PGA, D-SAGA, Local-SAGA, PGA-SAGA, GT-SAGA and PGA-GT-SAGA on directed ring graph with $n=8$ under different settings of data heterogeneity.}
    \vspace{-0.25cm}
    \label{Expe}
\end{figure}

\vspace{-0.1cm}
\textbf{Topology dependence.}
\finetune{To verify the dependency of the impact of data heterogeneity on the performance against the connectivity of the topology}, we conduct several experiments (with a fixed value of $h=20$) on different graphs: directed ring with $n=8,\, \rho _W\approx 0.92$ ; exponential graph with $n=50, \,\rho _W\approx 0.99$ for DSGD, D-SAGA and GT-SAGA, respectively, whose testing accuracy results are plotted in Fig.~\ref{Fig_topo_dependence}. \finetune{It follows from  Fig.~\ref{Fig_topo_dependence} that the performance of both DSGD and D-SAGA will be degraded when the connectivity of the graph becomes worse while GT-SAGA maintains relatively high testing accuracy since it removes data heterogeneity $\zeta^*$ by employing GT-schemes.}

\begin{figure}[h]
\vspace{-0.7cm}
    \centering
    \subfigure 
    {
        \begin{minipage}[t]{0.22\textwidth}
            \centering          
            \includegraphics[width=\textwidth]{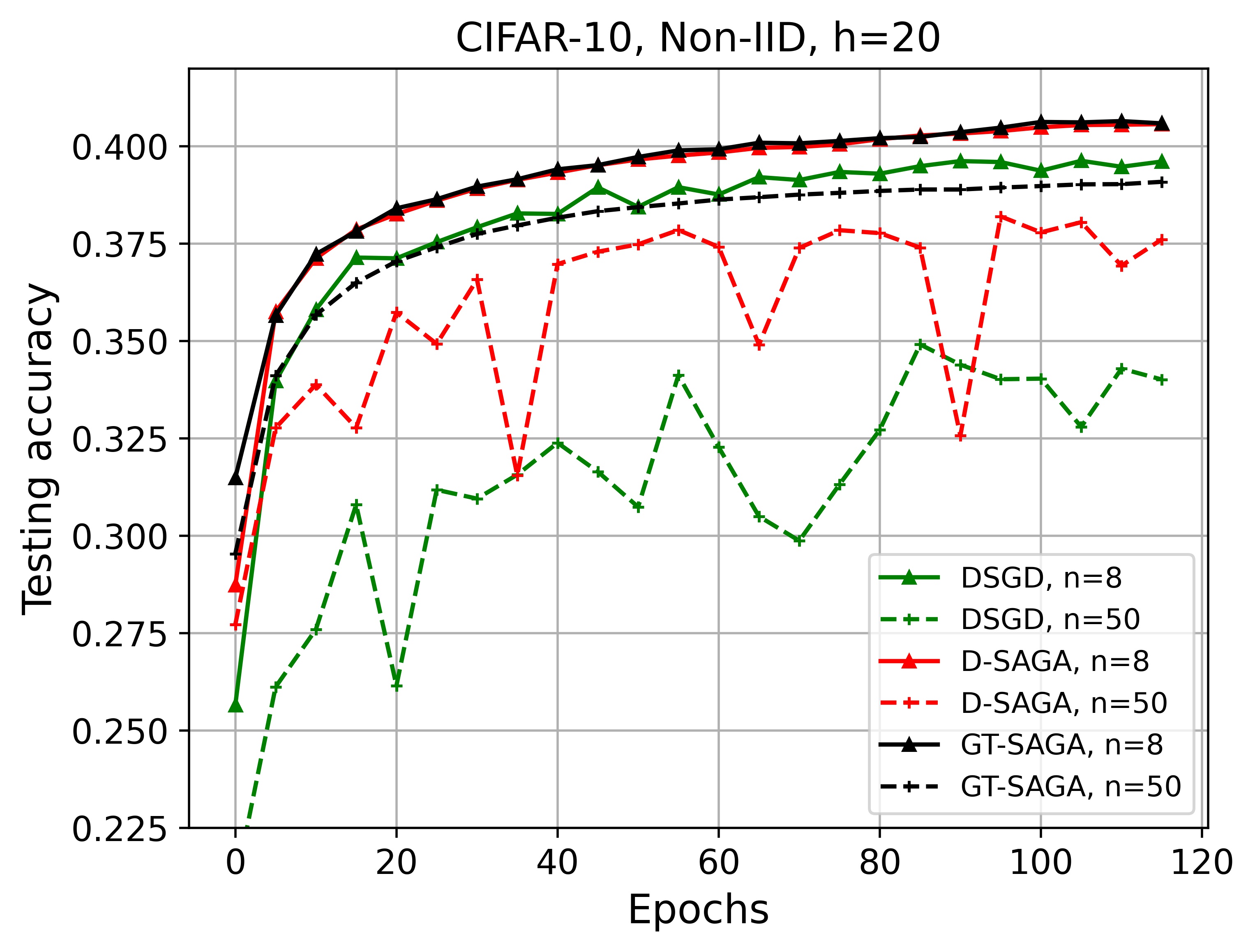}   
        \end{minipage}%
    }
    \subfigure 
    {
        \begin{minipage}[t]{0.22\textwidth}
            \centering      
            \includegraphics[width=\textwidth]{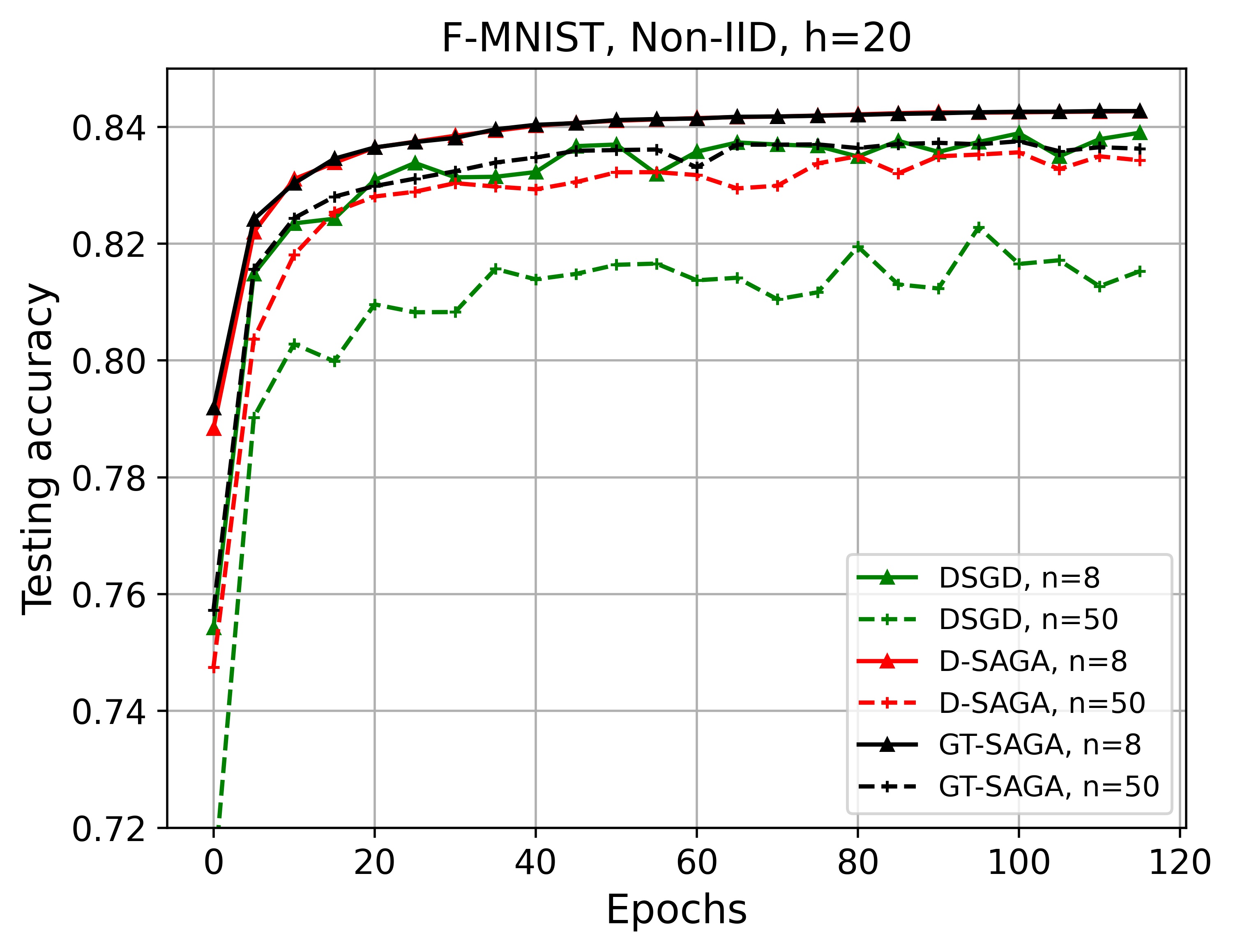}   
        \end{minipage}
    }%
    \vspace{-0.5cm}
    \caption{Performance comparison of DSGD, D-SAGA, and GT-SAGA on graphs with $n=8, 50$.}
    \label{Fig_topo_dependence}
    \vspace{-0.5cm}
\end{figure}

\section{Conclusions}
This paper develops a new unified framework for first-order stochastic gradient methods both in centralized and distributed scenarios. The proposed framework is able to recover many existing stochastic algorithms along with their corresponding rates. It also enable us to easily design new efficient algorithms by proper design of sampling strategies on the augmented graph. This framework is especially suitable for scenarios where data heterogeneity is a key concern. Since the framework heavily depends on the underlying augmented graph, it is of great interest and importance to design a proper augmented graph and sampling strategy to account for different important scenarios.

\section*{Acknowledgements}
The work of Huang, Zhu, Yan, and Xu has been supported in parts by National Natural Science Foundation of China under Grants 62003302, 62088101, 61922058 and 62173225; and in parts by the Key Laboratory of Collaborative Sensing and Autonomous Unmanned Systems of Zhejiang Province. The work of  Sun has been supported by  the Office of Naval Research, under the grant N00014-21-1-2673.



\bibliographystyle{apalike}
\bibliography{reference}

\onecolumn
\textbf{\large {Appendix}}

\vspace{-0.2cm}
{\footnotesize
\tableofcontents
}

\appendix

\newpage
\renewcommand\thesection{\Alph{section}}

\section{Recovering Existing Algorithms}\label{Appendix_recover}
In this section, we show that most of existing algorithms can be recovered by the proposed SPP framework. To this end, we first recall our proposed SPP framework (\ref{SPP_framework}) as follows:
\begin{equation*}
\begin{aligned}
&X_{k+1}=R_{k}X_k-\alpha \varGamma _{k}Y_k,
\\
&Y_{k+1}=C_{k}Y_k+\nabla F\left( X_{k+1} \right) -\nabla F\left( X_k \right) ,
\end{aligned}
\end{equation*}
where
\begin{equation*}
\begin{aligned}
\varGamma _k=\varLambda _{k+1}\left( W_k\otimes \mathbf{1}\mathbf{1}^T \right) \frac{\varLambda _k}{b_k},\quad \,\,R_k=\mathbf{I}_M-\varLambda _{k+1}+\varGamma _k,\quad C_k=G_k\otimes V_k.
\end{aligned}
\end{equation*}
Note that the matrices $R_k$ is row-stochastic and $C_k$ is doubly-stochastic such that we have by induction
\[
\frac{\mathbf{1}_{M}^{T}}{M}Y_k=\frac{\mathbf{1}_{M}^{T}}{M}\nabla F\left( X_k \right), \forall k\geq 0.
\]
Also, let us recall the notions with dimension mentioned in the main text:
\begin{equation}
\begin{aligned}
X_k&:=\left[ x_{1,k}, x_{2,k}, \cdots , x_{M,k} \right] ^T\in \mathbb{R}^{M\times d},\quad x_{i,k}\in \mathbb{R}^{d}\quad \forall i,
\\
Y_k&:=\left[ y_{1,k}, y_{2,k}, \cdots ,y_{M,k} \right] ^T\in \mathbb{R}^{M\times d}, \quad y_{i,k}\in \mathbb{R}^{d}\quad  \forall i,
\\
\nabla F\left( X_k \right) &:=\left[ \nabla f_1\left( x_{1,k} \right) ,\cdots ,\nabla f_M\left( x_{M,k} \right) \right] ^T\in \mathbb{R}^{M\times d},
\\
\hat{X}_k&:=S_kX_k=\left[ \hat{x}_{1,k}^{T},\hat{x}_{2,k}^{T},\cdots ,\hat{x}_{n,k}^{T} \right] ^T\in \mathbb{R}^{n\times d},\quad 
\\
\hat{Y}_k&:=S_kY_k=\left[ \hat{y}_{1,k}^{T},\hat{y}_{2,k}^{T},\cdots ,\hat{y}_{n,k}^{T} \right] ^T\in \mathbb{R}^{n\times d},
\\
\bar{x}_k&:=\frac{\mathbf{1}_{n}^{T}}{n}\hat{X}_k=\frac{\mathbf{1}_{n}^{T}}{n}S_kX_k\in \mathbb{R}^{1\times d}, \quad 
\\
\bar{y}_k&:=\frac{\mathbf{1}_{n}^{T}}{n}\hat{Y}_k=\frac{\mathbf{1}_{n}^{T}}{n}S_kY_k\in \mathbb{R}^{1\times d},
\\
x^*&:=\underset{x}{\mathrm{argmin}} f\left( x \right) \in \mathbb{R}^{1\times d}.
\end{aligned}
\end{equation}
Now, we show that how each algorithm mentioned in this paper can be recovered with particular choices of $\left( \varGamma_k ,\,\,R_k, C_k \right)$ as well as the projection matrix $S_k:=\left( \mathbf{I}_n\otimes \mathbf{1}_{m}^{T} \right) \frac{\varLambda _k}{b_k}$ (c.f., Eq. ~\eqref{Def_projection_matrix}).

The following lemma related to Kronecker product is crucial in recovering these algorithms.

\begin{Lem}\label{Lem_recover}
Suppose Assumption \ref{Ass_uniformly_sampling_without_replacement} holds. Then we have for all $k\geqslant 0$,
\begin{equation}
\begin{aligned}
S_{k+1}R_k=S_{k+1}\varGamma _k=W_kS_k.
\end{aligned}
\end{equation}
\end{Lem}
\begin{proof}
{Since $\varLambda _k\left( \mathbf{I}_M-\varLambda _k \right) =0$ according to the definition of $\Lambda_k$, we have $S_{k+1}R_k=S_{k+1}\varGamma _k$.} Further, we obtain that
\begin{equation}
\begin{aligned}
S_{k+1}R_k&=S_{k+1}\varGamma _k=\left( \left( \mathbf{I}_n\otimes \mathbf{1}_{m}^{T} \right) \frac{\varLambda _{k+1}}{b_{k+1}} \right) \left( \left( W_k\otimes \mathbf{1}\mathbf{1}^T \right) \frac{\varLambda _k}{b_k} \right) 
\\
&=\frac{1}{b_kb_{k+1}}\left[ \begin{matrix}
	\ddots&		&		\\
	&		\mathbf{1}_{m}^{T}\varLambda _{k+1}^{i}&		\\
	&		&		\ddots\\
\end{matrix} \right] \left[ \begin{matrix}
	\cdots&		w_{1j,k}\mathbf{1}\mathbf{1}^T\varLambda _{k}^{j}&		\cdots\\
	&		\vdots&		\\
	\cdots&		w_{nj,k}\mathbf{1}\mathbf{1}^T\varLambda _{k}^{j}&		\cdots\\
\end{matrix} \right] 
\\
&=\frac{1}{b_k}\left[ \begin{matrix}
	\cdots&		w_{1j,k}\mathbf{1}_{m}^{T}\varLambda _{k}^{j}&		\cdots\\
	&		\vdots&		\\
	\cdots&		w_{nj,k}\mathbf{1}_{m}^{T}\varLambda _{k}^{j}&		\cdots\\
\end{matrix} \right] =W_k\left( \mathbf{I}_n\otimes \mathbf{1}_{m}^{T} \right) \frac{\varLambda _k}{b_k},
\end{aligned}
\end{equation}
which completes the proof.
\end{proof}

\subsection{Recovering (centralized) SAGA, L-SVRG and SARAH}
To recover the algorithms with VR schemes, we set the matrices $\left( \varGamma_k, \,R_k,\, C_k \right)$ as follows:
\begin{equation*}
\begin{aligned}
\varGamma _k=\varLambda _{k+1}\mathbf{1}\mathbf{1}^T\frac{\varLambda _k}{b_k},\quad R_k=\mathbf{I}_m-\varLambda _{k+1}+\varLambda _{k+1}\mathbf{1}\mathbf{1}^T\frac{\varLambda _k}{b_k},\quad C_k=V_k.
\end{aligned}
\end{equation*}

\textbf{For SAGA.} The matrices $W_k$, $G_k$ and $V_k$ are set as follows (c.f., Section~\ref{Sec_Recover} for more details for these matrices):
\begin{equation}
\begin{aligned}
W_k=G_k=1,\,\, V_k=\mathbf{J}_m,\,\, k\geqslant 1.
\end{aligned}
\end{equation}
We denote by $\boldsymbol{s}_k$ the set of randomly selected sample nodes at iteration $k$ with constant mini-batch size $b\in[1, m]$. Then, multiplying $S_{k+1}$ for both sides of \eqref{SPP_framework} and invoking Lemma \ref{Lem_recover}, we can derive the following recursions:
\begin{equation}\label{recoverd SAGA}
\begin{aligned}
\hat{x}_{k+1}&=S_{k+1}\left( R_kX_k-\alpha \varGamma _kY_k \right) =\hat{x}_k-\alpha \hat{y}_k,
\\
\hat{y}_{k+1}&=S_{k+1}\left( \mathbf{J}_mY_k+\nabla F\left( X_{k+1} \right) -\nabla F\left( X_{t_k} \right) \right) 
\\
&=\frac{\mathbf{1}_{m}^{T}}{m}\nabla F\left( X_k \right) +S_{k+1}\left( \nabla F\left( X_{k+1} \right) -\nabla F\left( X_k \right) \right) 
\\
&=\frac{1}{m}\sum_{j=1}^m{\nabla f_j\left( x_{j,k} \right)}+\frac{1}{b}\sum_{s\in \boldsymbol{s}_{k+1}}{\left( \nabla f_s\left( \hat{x}_{k+1} \right) -\nabla f_s\left( x_{s,k} \right) \right)},
\end{aligned}
\end{equation}
where $\left( \hat{x}_k, \hat{y}_k \right) $ is the specific instance of $\left( \hat{X}_k, \hat{Y}_k \right) $ with $n=1$, and $s\in \boldsymbol{s}_{k+1}$ represents the index of sample randomly picked at iteration $k+1$. It is worth noting that only the picked sample nodes in $s\in \boldsymbol{s}_{k+1}$ performs updates at each iteration. As a result, the collective gradient vector $\nabla F\left( X_k \right) $ keeps the historical gradients of all samples, resembling the role of the table storing the gradients in the original SAGA algorithm \citep{defazio2014saga}.

\textbf{For L-SVRG.} The matrix $V_k$ corresponding to variance reduction and batch-size $b_k$ vary as:
\begin{equation*}
\begin{aligned}
\begin{cases}
	V_k=\mathbf{I}_m,\quad\, b_k=b, \quad\,\, w.p.\quad 1-p\\
	V_k=\mathbf{J}_m,\quad b_k=m, \quad w.p.\quad p\\
\end{cases},
\end{aligned}
\end{equation*}
which indicates that the algorithm performs full gradient update with probability $p$.
Then, by Lemma \ref{Lem_recover} and Lemma \ref{Lem_Y_t}, we derive the recursions of the recovered L-SVRG algorithm with the projection matrix $S_{k+1}$: 
\begin{equation}
\begin{aligned}
\hat{x}_{k+1}&=S_{k+1}\left( R_kX_k-\alpha \varGamma _kY_k \right) =\hat{x}_k-\alpha \hat{y}_k,
\\
\hat{y}_{k+1}&=S_{k+1}\left( V_kY_k+\nabla F\left( X_{k+1} \right) -\nabla F\left( X_k \right) \right) 
\\
&=\frac{1}{m}\sum_{j=1}^m{\nabla f_j\left( x_{j,t_{k+1}} \right)} +\frac{1}{b_{k+1}}\sum_{s\in \boldsymbol{s}_{k+1}}{\left( \nabla f_s\left( \hat{x}_{k+1} \right) -\nabla f_s\left( x_{s,t_{k+1}} \right) \right)},
\end{aligned}
\end{equation}
where $t_{k+1} < k+1$ denotes the latest iteration before $k+1$ performing the full gradient update, i.e., $b_{t_{k+1}}=m$, and $x_{j,t_{k+1}}=\hat{x}_{t_{k+1}}, \forall j\in \left[ m \right]$. The original L-SVRG algorithm \citep{qian2021svrg} is thus recovered.

{\textbf{For SARAH.} As with SAGA and L-SVRG, we can verify that SARAH can be also recovered from the SPP framework by setting:
\begin{equation}
\begin{aligned}
V_k = \mathbf{J}_m, \quad
b_k=\begin{cases}
	b, \quad \,\, w.p.\quad 1-p\\
	m, \quad w.p.\quad p\\
\end{cases},
\end{aligned}
\end{equation}
which indicates that SARAH always performs variance reduction (since $V_k = \mathbf{J}_m$) while using dynamic sampling strategy. Therefore, we can obtain the original SARAH algorithm \citep{pmlr-v70-nguyen17b} by the projection matrix $S_k$, which shares the same recursions of SAGA in \eqref{recoverd SAGA} while intermittently performing full gradient update. It can be further revealed from Table \ref{Tab_uni_SARAH} that SARAH is, indeed, a mixing of SAGA and L-SVRG.
}
\begin{table}[h!]
  \caption{{A unified perspective for SAGA, L-SVRG and SARAH under SPP framework}}
 \label{Tab_uni_SARAH}
 \renewcommand\arraystretch{1.2}
 \begin{center}
  \label{VR Methods}
  \vspace{0.2cm}
  \begin{tabular}{c|c|c|c} 
   \hline
   \rule {0pt}{8pt}
    \textbf{Algorithms}  & SAGA & L-SVRG & \textbf{SARAH} 
   \\ 
    \hline
   \rule {0pt}{8pt}
   $b_k$ & $b$ & $\left\{b, \,\, m \right\} $ & $\left\{ b, \,\, m\right\} $\\
   \hline
   \rule {0pt}{8pt}
   $V_k$ & $\mathbf{J}_m$ & $\left\{ \mathbf{I}_m, \mathbf{J}_m \right\}$ & $\mathbf{J}_m $\\
   \hline
  \end{tabular}
 \end{center}
\end{table}

\subsection{Recovering Local-SAGA, D-SAGA and PGA-SAGA}
We choose the parameters $\left( \varGamma_k, \,R_k,\, C_k \right)$ as follow:
\begin{equation*}
\begin{aligned}
\varGamma _k=\varLambda _{k+1}\left( W_k\otimes \mathbf{1}\mathbf{1}^T \right) \frac{\varLambda _k}{b_k},\quad
R_k=\mathbf{I}_M-\varLambda _{k+1}+\varLambda _{k+1}\left( W_k\otimes \mathbf{1}\mathbf{1}^T \right) \frac{\varLambda _k}{b_k},  \quad
C_k=\left( \mathbf{I}_n\otimes V_k \right) .
\end{aligned}
\end{equation*}

\textbf{For Local-SAGA.} We set $G_k=\mathbf{I}_n, V_n=\mathbf{J}_m$, and the mixing matrix $W_k$ corresponding to the actual communication topology varies as:
\begin{equation*}
\begin{aligned}
W_k=\begin{cases}
	\mathbf{I}_n,\quad w.p.\quad 1-r\\
	\mathbf{J}_n,\quad w.p.\quad r\\
\end{cases}.
\end{aligned}
\end{equation*}

From Lemma \ref{Lem_recover}, we derive the following recursion of the recovered Local-SAGA algorithm:
\begin{equation}
\begin{aligned}
\hat{X}_{k+1}&=S_{k+1}\left( R_kX_k-\alpha \varGamma _kY_k \right) =W_k\left( \hat{X}_k-\alpha \hat{Y}_k \right),
\\
\hat{Y}_{k+1}&=S_{k+1}\left( \left( \left( \mathbf{I}_n\otimes \mathbf{J}_m \right) Y_k+\nabla F\left( X_{k+1} \right) -\nabla F\left( X_k \right) \right) \right) 
\\
&=\left( \mathbf{I}_n\otimes \frac{\mathbf{1}_{m}^{T}}{m} \right) \nabla F\left( X_k \right) +S_{k+1}\left( \nabla F\left( X_{k+1} \right) -\nabla F\left( X_k \right) \right).
\end{aligned}
\end{equation}
Then, by noticing that the decision variables $X_{k+1}$ and $X_k$ are only different at the rows corresponding to samples at iteration $k+1$, we have
\begin{equation}
\begin{aligned}
S_{k+1}\left( \nabla F\left( X_{k+1} \right) -\nabla F\left( X_k \right) \right) =\left( \mathbf{I}_n\otimes \frac{\mathbf{1}_{m}^{T}}{m} \right) \left( \nabla F\left( X_{k+1} \right) -\nabla F\left( X_k \right) \right) ,
\end{aligned}
\end{equation}
which implies
\begin{equation}
\begin{aligned}
\hat{Y}_{k+1}=\left( \mathbf{I}_n\otimes \frac{\mathbf{1}_{m}^{T}}{m} \right) \nabla F\left( X_k \right).
\end{aligned}
\end{equation}
By doing so, we get the recovered Local-SAGA algorithm:
\begin{equation}
\begin{aligned}
\hat{X}_{k+1}=W_k\left( \hat{X}_k-\alpha \left( \mathbf{I}_n\otimes \frac{\mathbf{1}_{m}^{T}}{m} \right) \nabla F\left( X_k \right) \right) ,
\end{aligned}
\end{equation}
which indicates that each device performs SAGA over their local datasets, and carry out global communication at a probability of $r$ to reach consensus . 

\textbf{For D-SAGA and PGA-SAGA.} Similar to Local-SAGA, the only difference between them are the mixing matrix $W_k$, in specific, we can recover D-SAGA by choosing $W_k = W$ for $k\geqslant0$, and PGA-SAGA by choosing $W_k$ varying as:
\begin{equation}
\begin{aligned}
W_k=\begin{cases}
	W,\quad w.p.\quad \,\,1-r\\
	\mathbf{J}_n,\quad w.p.\quad \,\,r\\
\end{cases}.
\end{aligned}
\end{equation}

\subsection{Recovering Local-SVRG and D-SVRG}
\textbf{For Local-SVRG.} The matrix $W_k$ corresponding to the actual communication topology, and $V_k$ corresponding to variance reduction vary as:
\begin{equation*}
\begin{aligned}
W_k=\begin{cases}
	\mathbf{I}_n,\quad w.p.\quad 1-r\\
	\mathbf{J}_n,\quad w.p.\quad r\\
\end{cases},  \quad
\begin{cases}
	V_k=\mathbf{I}_m,\quad\, b_k=b, \quad\,\, w.p.\quad 1-p\\
	V_k=\mathbf{J}_m,\quad b_k=m, \quad w.p.\quad p\\
\end{cases},
\end{aligned}
\end{equation*}
which indicates that each device performs full gradient update with probability $p$ (L-SVRG) over their local datasets, and global communication to reach consensus with probability $r$.

Then, similar to Local-SAGA,  we can drive the recursions of the recovered Local-SVRG algorithm with $S_{k+1}$:
\begin{equation}
\begin{aligned}
\hat{X}_{k+1}&=S_{k+1}\left( R_kX_k-\alpha \varGamma _kY_k \right) =W_k\left( \hat{X}_k-\alpha \hat{Y}_k \right),
\\
\hat{Y}_{k+1}&=\left( \mathbf{I}_n\otimes \mathbf{1}_{m}^{T} \right) \frac{\varLambda _{k+1}}{b_{k+1}}\left( \left( \mathbf{I}_n\otimes V_k \right) Y_k+\nabla F\left( X_{k+1} \right) -\nabla F\left( X_k \right) \right) 
\\
&=\left( \mathbf{I}_n\otimes \frac{\mathbf{1}_{m}^{T}}{m} \right) \nabla F\left( X_{t_{k+1}} \right) +\left( \mathbf{I}_n\otimes \frac{\mathbf{1}_{m}^{T}}{m} \right) \left( \nabla F\left( X_{k+1} \right) -\nabla F\left( X_{t_{k+1}} \right) \right). 
\end{aligned}
\end{equation}
Noticing that $t_{k+1} < k+1$ denotes the iteration before $k+1$ performing full gradient update, which thus recovers the original Local-SVRG algorithm \citep{gorbunov2021local}.

\textbf{For D-SVRG.} Similar to Local-SVRG,  it is straightforward to recover D-SVRG by choosing the mixing matrix $W_k = W$ for $k\geqslant0$.

\subsection{Recovering GT-SAGA}
We choose the parameters $\left( \varGamma_k, \,R_k,\, C_k \right)$ as follow:
\begin{equation*}
\begin{aligned}
\varGamma _k=\varLambda _{k+1}\left( W_k\otimes \mathbf{1}\mathbf{1}^T \right) \frac{\varLambda _k}{b_k},\quad R_k=\mathbf{I}_M-\varLambda _{k+1}+\varLambda _{k+1}\left( W_k\otimes \mathbf{1}\mathbf{1}^T \right) \frac{\varLambda _k}{b_k},\quad C_k=\left( W_k\otimes V_k \right),
\end{aligned}
\end{equation*}
which indicates performing both gradient tracking and variance reduction. Then, using Lemma \ref{Lem_recover},  we can derive the recursions of the recovered GT-SAGA algorithm:
\begin{equation}
\begin{aligned}
\hat{X}_{k+1}&=S_{k+1}\left( R_kX_k-\alpha \varGamma _kY_k \right) =W_k\left( \hat{X}_k-\alpha \hat{Y}_k \right) ,
\\
\hat{Y}_{k+1}&=\left( \mathbf{I}_n\otimes \mathbf{1}_{m}^{T} \right) \frac{\varLambda _{k+1}}{b_{k+1}}\left( \left( W_k\otimes \mathbf{J}_m \right) Y_k+\nabla F\left( X_{k+1} \right) -\nabla F\left( X_k \right) \right) 
\\
&=W_k\left( \mathbf{I}_n\otimes \frac{\mathbf{1}_{m}^{T}}{m} \right) Y_k+S_{k+1}\nabla F\left( X_{k+1} \right) -S_{k+1}\nabla F\left( X_k \right) \,\,
\\
&=W_k\hat{Y}_k-\left( \mathbf{I}_n\otimes \frac{\mathbf{1}_{m}^{T}}{m} \right) \nabla F\left( X_{k+1} \right) -\left( \mathbf{I}_n\otimes \frac{\mathbf{1}_{m}^{T}}{m} \right) \nabla F\left( X_k \right) ,
\end{aligned}
\end{equation}
where in the last equality we have used the following facts:
\begin{equation*}
\begin{aligned}
&\left( \mathbf{I}_n\otimes \frac{\mathbf{1}_{m}^{T}}{m} \right) Y_k=\left( \mathbf{I}_n\otimes \frac{\mathbf{1}_{m}^{T}}{m} \right) \varLambda _kY_k=\hat{Y}_k, \\
&S_{k+1}\left( \nabla F\left( X_{k+1} \right) -\nabla F\left( X_k \right) \,\, \right) =\left( \mathbf{I}_n\otimes \frac{\mathbf{1}_{m}^{T}}{m} \right) \left( \nabla F\left( X_{k+1} \right) -\nabla F\left( X_k \right) \right).
\end{aligned}
\end{equation*}
Noticing that $\nabla F\left( X_k \right)$, indeed, plays the role of the table of SAGA storing the historical gradients, we thus recover the original GT-SAGA algorithm \citep{xin2020variance}.

\textbf{New efficient algorithms.} It should be noted that we can further design various new algorithms by properly choosing the matrices $W_k$, $G_k$ and $V_k$ to obtain suitable performance for different scenarios.

\section{Technical Preliminaries}
\label{tech_prelim}
The key idea of the proofs for the main results in Section \ref{Sec_Con_analysis} is based on the proper design of the following Lyapunov function as defined in (\ref{Lyapunov_func}), which we recall here:
\begin{equation*}
\begin{aligned}
T_{k+1}&:=c_0\left\| \bar{x}_{k+1}-x^* \right\| ^2+c_1\left\| \hat{X}_{k+1}-\mathbf{1}_n\bar{x}_{k+1} \right\| ^2
\\
&+c_2\left\| \nabla F\left( X_{t_k} \right) -\nabla F\left( \mathbf{1}_Mx^* \right) \right\| ^2+c_3\left\| \nabla F\left( X_k \right) -\nabla F\left( \mathbf{1}_Mx^* \right) \right\| ^2+c_4\left\| \hat{Y}_k-\mathbf{1}_n\bar{y}_k \right\| ^2,
\end{aligned}
\end{equation*}
where $c_0, c_1, c_2, c_3, c_4 \geqslant 0$ are parameters to be properly determined.
In particular, the above Lyapunov function consists of the following five error terms in the sense of expectation:
\begin{itemize}
    \item  \textbf{Optimality gap}: $\mathbb{E}\left[ \left\| \bar{x}_k-x^* \right\|^{2} \right] $;
    \item \textbf{Consensus error across proxy nodes}: $\mathbb{E}\left[\left\| \hat{X}_k-\mathbf{1}_n\bar{x}_k \right\| ^2 \right]$;
    \item \textbf{Delayed variance-reduction error}: $\mathbb{E}\left[ \left\| \nabla F\left( X_{t_k} \right) -\nabla F\left( \mathbf{1}_Mx^* \right) \right\|^{2} \right]$;
    \item \textbf{Variance-reduction error }: $\mathbb{E}\left[ \left\| \nabla F\left( X_k \right) -\nabla F\left( \mathbf{1}_Mx^* \right) \right\| ^2 \right] $;
    \item \textbf{Gradient tracking error}: $\mathbb{E}\left[ \left\| \hat{Y}_k-\mathbf{1}_n\bar{y}_k \right\| ^2 \right] $.
\end{itemize}

\textbf{Proof Sketch}: To obtain the main results, we first need to establish the evolution of each of the above error terms denoted as $\boldsymbol{e}_k$ (c.f., Lemma \ref{Lem_opt_gap}-\ref{Lem_GT_error}) to come up with a recursive dynamics: $\boldsymbol{e}_{k+1}\leqslant A\boldsymbol{e}_k+\boldsymbol{b}$. Then, by properly choosing a non-negative non-zero coefficient vector $\boldsymbol{c}:=\left[ c_0,\cdots,c_4\right] ^T$ such that  we have $\boldsymbol{c}^TA\leqslant \varrho  \boldsymbol{c}^T$ with $\varrho  <1$, we are able us to obtain the convergence results as stated in Theorem (Corollary) \ref{Thm_Without_GT_VR}-\ref{Thm_With_GT_VR} for smooth and strongly convex objectives ($\mu>0$). Similar procedures can be applied to obtain the sub-linear rates for general convex cases ($\mu = 0$).

Before proceeding to the main proofs, we first introduce some lemmas that will be crucial in the subsequent analysis. Besides, we denote by $\mathcal{F}_k$ the $\sigma$-algebra generated by $
\left\{ \varLambda _0, R_0, C_0, \cdots , \varLambda _{k-1}, R_{k-1}, C_{k-1} \right\} $, and define $\mathbb{E}\left[ \cdot |\mathcal{F}_k \right] $ as the conditional expectation given $\mathcal{F}_k$. We also recall the following definitions:
$$
\rho _W:=\left\| W-\mathbf{J}_n \right\|_2^2,\,\,r:=P\left( W_k=\mathbf{J}_n \right) ,\,\, p:=P\left( V_k=\mathbf{J}_m \right) ,\,\,q:=\mathbb{E}\left[ b_k/m|V_k=\mathbf{J}_m \right],
$$
where $\rho_W$ denotes the spectral gap of the fixed mixing matrix $W$; $r$ represents the probability of adopting global averaging; $p$ represents the probability of performing local variance reduction (i.e., updating the gradient table kept by each device); $q$ represents the expected batch-size of samples while performing variance reduction.

\subsection{Supporting Lemmas}
In this section, we first provide some lemmas that will be used in the subsequent analysis.

\begin{Lem}
{\citep{nesterov2003introductory}} Suppose Assumption \ref{Ass_smoothness} hold. Then, for any $x, x'\in \mathbb{R}^d$, we have
\begin{equation}\label{Eq_Convex_smooth}
\begin{aligned}
&\left\| \nabla f_i\left( x \right) -\nabla f_i\left( x' \right) \right\| ^2\leqslant 2L\left( f_i\left( x \right) -f_i\left( x' \right) -\left< \nabla f_i\left( x' \right) , x-x' \right> \right).
\end{aligned}
\end{equation}
\end{Lem}

\begin{Lem}\label{Lem_Exp2Raidus} 
Let $\boldsymbol{x}\in \mathbb{R}^n$ be a random vector and $A\in \mathbb{R}^{n\times n}$ be a random matrix. Then, if $A$ is independent on $\boldsymbol{x}$, we have
\begin{equation}
\begin{aligned}
\mathbb{E}\left[ \left\| A\boldsymbol{x} \right\| ^2 \right] =\mathbb{E}\left[ \boldsymbol{x}^TA^TA\boldsymbol{x} \right] =\mathbb{E}\left[ \boldsymbol{x}^T\mathbb{E}\left[ A^TA|\boldsymbol{x} \right] \boldsymbol{x} \right] \leqslant \rho \left( \mathbb{E}\left[ A^TA \right] \right) \mathbb{E}\left[ \left\| \boldsymbol{x} \right\| ^2 \right].
\end{aligned}
\end{equation}
\end{Lem}
The above lemma follows directly from the smoothing lemma {\citep{gut2005probability}} and thus its proof is omitted.

\begin{Lem}\label{Lem_expectation}
{Suppose Assumption \ref{Ass_uniformly_sampling_without_replacement} hold.} Then, we have for all $k\geqslant0$,
\begin{equation}\label{Eq_rho_S_k}
\begin{aligned}
\rho(\mathbb{E}\left[ \left( S_k \right) ^TS_k \right]) \leqslant \frac{1}{m},
\end{aligned}
\end{equation}
and
\begin{equation}\label{Eq_rho_bar_S_k}
\begin{aligned}
\rho \left( \mathbb{E}\left[ \left( \frac{\mathbf{1}_{n}^{T}}{n}S_k \right) ^T\left( \frac{\mathbf{1}_{n}^{T}}{n}S_k \right) \right] \right) \leqslant \frac{1}{M},
\end{aligned}
\end{equation}
where $\rho(\cdot)$ denotes the spectral radius of a matrix.
\end{Lem}

\begin{proof}
Under Assumption \ref{Ass_uniformly_sampling_without_replacement}, we know that each node $i$ performs independent identically distributed mini-batch sampling without replacement from finite $m$ data samples locally. 
Therefore, by Lemma \ref{Lem_Exp2Raidus}, we have
\begin{equation*}
\begin{aligned}
&\mathbb{E}\left[ \left( S_k \right) ^TS_k \right] =\mathbb{E}\left[ \frac{\varLambda _k}{b_k}\left( \mathbf{I}_n\otimes \mathbf{1}_m\mathbf{1}_{m}^{T} \right) \frac{\varLambda _k}{b_k} \right]
\\
&=\mathbb{E}\left[ \frac{1}{\left( b_k \right) ^2}\left( \frac{C_{m-1}^{b_k-1}}{C_{m}^{b_k}}\mathbf{I}_M+\frac{C_{m-2}^{b_k-2}}{C_{m}^{b_k}}\left( \mathbf{I}_n\otimes \left( \mathbf{1}_m\mathbf{1}_{m}^{T}-\mathbf{I}_m \right) \right) \right) \right] 
\\
&=\mathbb{E}\left[ \frac{1}{\left( b_k \right) ^2}\left( \frac{b_k}{m}\left( 1-\frac{b_k-1}{m-1} \right) \mathbf{I}_M+\frac{b_k\left( b_k-1 \right)}{m\left( m-1 \right)}\left( \mathbf{I}_n\otimes \mathbf{1}_m\mathbf{1}_{m}^{T} \right) \right) \right] ,
\end{aligned}
\end{equation*}
where $C_m^{b_k}$ denote the number of $b_k$ combinations from the set $\left[ m \right]$.
Then, using Gershgorin circle theorem, we have
$$
\rho \left( \mathbb{E}\left[ \left( S_k \right) ^TS_k \right] \right) \leqslant \frac{1}{m},
$$
and
\begin{equation*}
\begin{aligned}
\rho \left( \mathbb{E}\left[ \left( \frac{\mathbf{1}_{n}^{T}}{n}S_k \right) ^T\left( \frac{\mathbf{1}_{n}^{T}}{n}S_k \right) \right] \right) \leqslant \rho \left( \frac{\mathbf{1}_n\mathbf{1}_{n}^{T}}{n^2} \right) \rho \left( \mathbb{E}\left[ \left( S_k \right) ^TS_k \right] \right) =\frac{1}{M},
\end{aligned}
\end{equation*}
which completes the proof.
\end{proof}

\begin{Lem}\label{Lem_Y_t}
Suppose $C_{k} = \mathbf{I}_M$ hold for $k>t_k$. Then, we have by induction that
\begin{equation}
\begin{aligned}
Y_k=C_{t_k}Y_{t_k}+\nabla F\left( X_k \right) -\nabla F\left( X_{t_k} \right).
\end{aligned}
\end{equation}
\end{Lem}
The above lemma shows that there will be a delay in the recursion of $Y_k$ when the estimation on the full gradient is not adopted at each iteration (i.e., $C_k=\mathbf{I}_M$), such as SVRG and L-SVRG algorithms.

\begin{Lem}\label{Lem_X_by_RW}
Under the proposed SPP framework, we have for $k>0$
\begin{equation}\label{Eq_X_by_X_hat}
\begin{aligned}
X_{k+1}=\left( \mathbf{I}_M-\varLambda _{k+1} \right) X_k+\varLambda _{k+1}\left( \hat{X}_{k+1}\otimes \mathbf{1}_m \right),
\end{aligned}
\end{equation}
and
\begin{equation}\label{Eq_nabla_F_by_X_hat}
\begin{aligned}
\nabla F\left( X_k \right) = \left( \mathbf{I}_M-\varLambda _k \right) \nabla F\left( X_{k-1} \right) +\varLambda _k\nabla F\left( \hat{X}_k\otimes \mathbf{1}_m \right).
\end{aligned}
\end{equation}
\end{Lem}

\begin{proof}
By~(\ref{Lab_Gamma_k}), (\ref{Lab_R_k}), we have
\begin{equation*}
\begin{aligned}
X_{k+1}&=\left( \mathbf{I}_M-\varLambda _{k+1} \right) X_k+\varLambda _{k+1}\left( W_k\otimes \mathbf{1}\mathbf{1}^T \right) \frac{\varLambda _k}{b_k}\left( X_k-\alpha Y_k \right) 
\\
&=\left( \mathbf{I}_M-\varLambda _{k+1} \right) X_k+\varLambda _{k+1}\left( W_k\otimes \mathbf{1}_m \right) \left( \mathbf{I}_n\otimes \mathbf{1}_m^T \right)  \frac{\varLambda _k}{b_k}\left(X_k-\alpha Y_k \right) 
\\
&=\left( \mathbf{I}_M-\varLambda _{k+1} \right) X_k+\varLambda _{k+1}\left( \hat{X}_{k+1}\otimes \mathbf{1}_m \right),
\end{aligned}
\end{equation*}
where in the second equality we have used the property of Kronecker product that $\left( A\otimes B \right) \left( C\otimes D \right) =AC\otimes BD$. In what follows, by noticing that $\varLambda _k\nabla F\left( X_k \right)$ represents selecting the rows of $\nabla F\left( X_k \right)$ corresponding to the randomly selected samples at iteration $k$ by $\Lambda_k$, we have
\begin{equation*}
\begin{aligned}
\nabla F\left( X_k \right) &=\nabla F\left( \left( \mathbf{I}_M-\varLambda _k \right) X_{k-1}+\varLambda _k\left( \hat{X}_k\otimes \mathbf{1}_m \right) \right) 
\\
&=\left( \mathbf{I}_M-\varLambda _k+\varLambda _k \right) \nabla F\left( \left( \mathbf{I}_M-\varLambda _k \right) X_{k-1}+\varLambda _k\left( \hat{X}_k\otimes \mathbf{1}_m \right) \right) 
\\
&=\varLambda _k\nabla F\left( \hat{X}_k\otimes \mathbf{1}_m \right) +\left( \mathbf{I}_M-\varLambda _k \right) \nabla F\left( X_{k-1} \right) ,
\end{aligned}
\end{equation*}
which completes the proof.
\end{proof}

Lemma \ref{Lem_X_by_RW} illustrates the update rule of decision variables over the augmented graph, that is, only the sampled nodes (samples) perform update, while other nodes that are not sampled remain unchanged.

\begin{Lem}\label{Expt_gradient}
Suppose Assumption \ref{Ass_uniformly_sampling_without_replacement} and \ref{Ass_algoirthm} hold. Then we have for all $k>0$
\begin{equation}
\begin{aligned}
\mathbb{E}\left[ \bar{y}_k|\mathcal{F}_{k} \right] =r\nabla f\left( \bar{x}_k \right) +\left( 1-r \right) g_k,
\end{aligned}
\end{equation}
where
\begin{equation}
\begin{aligned}
g_k=\frac{\mathbf{1}_{M}^{T}}{M}\nabla F\left( \hat{X}_k\otimes \mathbf{1}_m \right),
\end{aligned}
\end{equation}
and $r=P\left( W_k=\mathbf{J}_m \right)$ denotes the probability of performing global averaging at each iteration.
\end{Lem}

\begin{proof}
By the definition of $\bar{y}_k$ in (\ref{Def_x_y_bar}),
we have $\bar{y}_k=\frac{\mathbf{1}_{n}^{T}}{n}S_k Y_k,$ and
\begin{equation}
\begin{aligned}
\mathbb{E}\left[ \bar{y}_k|\mathcal{F}_k \right] &\overset{\left( \ref{recu_Y} \right)}{=}\mathbb{E}\left[ \frac{\mathbf{1}_{n}^{T}}{n}S_k\left( C_{k-1}Y_{k-1}+\nabla F\left( X_k \right) -\nabla F\left( X_{k-1} \right) \right) |\mathcal{F}_k \right] 
\\
&{=}\mathbb{E}\left[ \left( \frac{\mathbf{1}_{M}^{T}}{M}Y_{k-1}+\frac{\mathbf{1}_{n}^{T}}{n}S_k\nabla F\left( X_k \right) -\frac{\mathbf{1}_{n}^{T}}{n}S_k\nabla F\left( X_{k-1} \right) \right) |\mathcal{F}_k \right] 
\\
&\overset{\left( \ref{Eq_fixed_sum_Y} \right)}{=}\mathbb{E}\left[ \left( \frac{\mathbf{1}_{M}^{T}}{M}\nabla F\left( X_{k-1} \right) +\frac{\mathbf{1}_{n}^{T}}{n}S_k\nabla F\left( X_k \right) -\frac{\mathbf{1}_{M}^{T}}{M}\nabla F\left( X_{k-1} \right) \right) |\mathcal{F}_k \right] 
\\
&=\mathbb{E}\left[ \frac{\mathbf{1}_{n}^{T}}{n}S_k\nabla F\left( X_k \right) |\mathcal{F}_k \right],
\end{aligned}
\end{equation}
where in the second equality we have used the fact that $\mathbb{E}\left[ \frac{\mathbf{1}_{n}^{T}}{n}S_k \right] =\frac{\mathbf{1}_{n}^{T}}{n}\left( \mathbf{I}_n\otimes \frac{\mathbf{1}_{m}^{T}}{m} \right) =\frac{\mathbf{1}_{M}^{T}}{M}$ due to Assumption \ref{Ass_uniformly_sampling_without_replacement} and $C_{k-1}$ is doubly-stochastic. Then, by the Eq.~(\ref{Eq_nabla_F_by_X_hat}) in Lemma \ref{Lem_X_by_RW}, we obtain
\begin{equation}
\begin{aligned}
\mathbb{E}\left[ \bar{y}_k|\mathcal{F}_k \right] &=\mathbb{E}\left[ \frac{\mathbf{1}_{n}^{T}}{n}S_k\left[ \nabla F\left( \varLambda _k\left( \hat{X}_k\otimes \mathbf{1}_m \right) +\left( \mathbf{I}_M-\varLambda _k \right) X_{k-1} \right) \right] |\mathcal{F}_k \right] 
\\
&=\mathbb{E}\left[ \frac{\mathbf{1}_{n}^{T}}{n}S_k\left[ \varLambda _k\nabla F\left( \hat{X}_k\otimes \mathbf{1}_m \right) +\left( \mathbf{I}_M-\varLambda _k \right) \nabla F\left( X_{k-1} \right) \right] |\mathcal{F}_k \right] 
\\
&=\mathbb{E}\left[ \frac{\mathbf{1}_{n}^{T}}{n}S_k\nabla F\left( \hat{X}_k\otimes \mathbf{1}_m \right) |\mathcal{F}_k \right] ,
\end{aligned}
\end{equation}
where in the last equality we have used the fact that $\varLambda _k\left( \mathbf{I}_M-\varLambda _k \right) =0$. 
Therefore, by the law of total probability, we obtain 
\begin{equation*}
\begin{aligned}
\mathbb{E}\left[ \bar{y}_k|\mathcal{F}_k \right] &=r\mathbb{E}\left[ \frac{\mathbf{1}_{n}^{T}}{n}S_k\nabla F\left( \hat{X}_k\otimes \mathbf{1}_m \right) |\mathcal{F}_k, W_k=\mathbf{J}_n \right] +\left( 1-r \right) \mathbb{E}\left[ \frac{\mathbf{1}_{n}^{T}}{n}S_k\nabla F\left( \hat{X}_k\otimes \mathbf{1}_m \right) |\mathcal{F}_k, W_k=W \right] 
\\
&=r\nabla f\left( \bar{x}_k \right) +\left( 1-r \right) \frac{\mathbf{1}_{M}^{T}}{M}\nabla F\left( \hat{X}_k\otimes \mathbf{1}_m \right) ,
\end{aligned}
\end{equation*}
which completes the proof.
\end{proof}

Lemma \ref{Expt_gradient} shows that the expectation of $\bar{y}_k$ is a linear combination of the full gradient evaluated at the averaged decision variable and that evaluated at the decision variable of proxy nodes.

In the next lemma, we define an error term $\sigma_k$ (resembling internal variance) which can be bounded by the variance reduction errors, and establish the upper bounds under three particular settings discussed in the main text.

\begin{Lem}\label{Lem_sigma_k}
Suppose Assumption \ref{Ass_smoothness}-\ref{Ass_algoirthm} hold. Denote
\begin{equation}\label{Sigma_k}
\sigma _k:=n\mathbb{E}\left[ \left\| \frac{\mathbf{1}_{n}^{T}}{n}S_k\nabla F\left( \mathbf{1}_Mx^* \right) +\frac{\mathbf{1}_{n}^{T}}{n}S_kC_{t_k}Y_{t_k}-\frac{\mathbf{1}_{n}^{T}}{n}S_k\nabla F\left( X_{t_k} \right) -\nabla f\left( x^* \right) \right\| ^2|\mathcal{F}_k \right], 
\end{equation}
where $t_{k}<k$ represents the latest iteration before $k$ that $C_{t_{k}}\ne \mathbf{I}_M$.
Then we have, for all $k>0$, if we choose $C_k\equiv \mathbf{I}_M$ and $b_k = b$, 
\begin{equation}\label{Bound_sigma_k_1}
\begin{aligned}
\sigma _k \leqslant \frac{\sigma ^*}{{b}},
\end{aligned}
\end{equation}

else if $C_k=\mathbf{I}_n\otimes V_k$ with $p>0$ holds,
\begin{equation}\label{Bound_sigma_k_2}
\begin{aligned}
\sigma _k\leqslant \frac{1-p}{M}\mathbb{E}\left[ \left\| \nabla F\left( X_{t_{k-1}} \right) -\nabla F\left( \mathbf{1}_Mx^* \right) \right\| ^2|\mathcal{F}_k,V_{k-1}=\mathbf{I}_m \right] 
\\
+\frac{p}{M}\mathbb{E}\left[ \left\| \nabla F\left( X_{k-1} \right) -\nabla F\left( \mathbf{1}_Mx^* \right) \right\| ^2|\mathcal{F}_k,V_{k-1}=\mathbf{J}_m \right],
\end{aligned}
\end{equation}
else if $C_k=W_k\otimes \mathbf{J}_m$ holds, 
\begin{equation}\label{Bound_sigma_k_3}
\begin{aligned}
\sigma _k\leqslant \frac{1}{M}\mathbb{E}\left[ \left\| \nabla F\left( X_{k-1} \right) -\nabla F\left( \mathbf{1}_Mx^* \right) \right\| ^2|\mathcal{F}_k \right].
\end{aligned}
\end{equation}
\end{Lem}

\begin{proof}
By the definition of $\sigma_k$ and $t_k$, for $C_k\equiv \mathbf{I}_M$ in the first case, which implies $t_k=0$, we obtain by noticing $Y_0=\nabla F\left( X_0 \right)$ and $b_k=b$ that
\begin{equation}
\begin{aligned}
&\sigma _k=\frac{1}{n}\mathbb{E}\left[ \left\| \mathbf{1}_{n}^{T}S_k\nabla F\left( \mathbf{1}_Mx^* \right) -\mathbf{1}_{n}^{T}\left( \mathbf{I}_n\otimes \frac{\mathbf{1}_{m}^{T}}{m} \right) \nabla F\left( \mathbf{1}_Mx^* \right) \right\| ^2|\mathcal{F}_k \right] 
\\
&=\frac{1}{n}\sum_{i=1}^n{\mathbb{E}\left( \left\| \frac{1}{b_k}\sum_{j\in \xi _{i,k}}{\left( \nabla f_{ij}\left( x^* \right) -\nabla f_i\left( x^* \right) \right)} \right\| ^2|\mathcal{F}_k \right)}
\\
&\leqslant \frac{1}{n}\sum_{i=1}^n{\mathbb{E}\left( \frac{1}{b_{k}^{2}}\sum_{j\in \xi _{i,k}}{\left\| \nabla f_{ij}\left( x^* \right) -\nabla f_i\left( x^* \right) \right\| ^2}|\mathcal{F}_k \right)}
\leqslant \frac{\sigma ^*}{{b}},
\end{aligned}
\end{equation}
where we used the fact that $\mathbb{E}\left[ \nabla f_{ij}\left( x^* \right) -\nabla f_i\left( x^* \right) \right] =0, \,j\in \xi _{i}$.

For the case that $C_k=\mathbf{I}_n\otimes V_k$ with $p>0$, 
recalling that $C_k$ is column-stochastic, then we have
\begin{equation}
\begin{aligned}
\sigma _k&=n\mathbb{E}\left[ \left\| \frac{\mathbf{1}_{n}^{T}}{n}\left( \left( \mathbf{I}_n\otimes \frac{\mathbf{1}_{m}^{T}}{m} \right) \nabla F\left( X_{t_k} \right) -S_k\nabla F\left( X_{t_k} \right) +S_k\nabla F\left( \mathbf{1}_Mx^* \right) -\left( \mathbf{I}_n\otimes \frac{\mathbf{1}_{m}^{T}}{m} \right) \nabla F\left( \mathbf{1}_Mx^* \right) \right) \right\| ^2|\mathcal{F}_k \right] 
\\
&=\frac{1}{n}\mathbb{E}\left[ \left\| \left( \mathbf{I}_n\otimes \frac{\mathbf{1}_{m}^{T}}{m} \right) \nabla F\left( X_{t_k} \right) -S_k\nabla F\left( X_{t_k} \right) +S_k\nabla F\left( \mathbf{1}_Mx^* \right) -\left( \mathbf{I}_n\otimes \frac{\mathbf{1}_{m}^{T}}{m} \right) \nabla F\left( \mathbf{1}_Mx^* \right) \right\| ^2|\mathcal{F}_k \right] ,
\end{aligned}
\end{equation}
where in the second equality we used $\mathbb{E}\left[ S_k \right] =\mathbf{I}_n\otimes \frac{\mathbf{1}_{m}^{T}}{m} $ induced by Assumption \ref{Ass_uniformly_sampling_without_replacement}. Furthermore, noticing the fact that $\mathbb{E}\left[ \left\| X-\mathbb{E}\left[ X \right] \right\| ^2 \right] \leqslant \mathbb{E}\left[ \left\| X \right\| ^2 \right]$, we get
\begin{equation}
\begin{aligned}
\sigma _k&\leqslant \frac{1}{n}\mathbb{E}\left[ \left\| S_k\nabla F\left( X_{t_k} \right) -S_k\nabla F\left( \mathbf{1}_Mx^* \right) \right\| ^2|\mathcal{F}_k \right] 
\\
&\leqslant \frac{1-p}{M}\mathbb{E}\left[ \left\| \nabla F\left( X_{t_{k-1}} \right) -\nabla F\left( \mathbf{1}_Mx^* \right) \right\| ^2|\mathcal{F}_k,V_{k-1}=\mathbf{I}_m \right] 
\\
&+\frac{p}{M}\mathbb{E}\left[ \left\| \nabla F\left( X_{k-1} \right) -\nabla F\left( \mathbf{1}_Mx^* \right) \right\| ^2|\mathcal{F}_k,V_{k-1}=\mathbf{J}_m \right],
\end{aligned}
\end{equation}
where in the last inequality we have used the law of total probability.

For $C_k=W_k\otimes \mathbf{J}_m$ in the third case, we have $p=1$ and
\begin{equation*}
\begin{aligned}
\sigma _k\leqslant \frac{1}{M}\mathbb{E}\left[ \left\| \nabla F\left( X_{k-1} \right) -\nabla F\left( \mathbf{1}_Mx^* \right) \right\| ^2|\mathcal{F}_k \right].
\end{aligned}
\end{equation*}
which completes the proof.
\end{proof}

Lemma \ref{Lem_sigma_k} shows that the error term $\sigma_k$ is related to the optimization error  when the variance-reduction methods are adopted (see Lemma \ref{Lem_VR_error} and \ref{Lem_VR_error_delay}), otherwise, there will be a fixed upper bound on $\sigma_k$ as defined in Assumption \ref{Ass_sampling}.

Now we are going to bound each error term in the Lyapunov function (\ref{Lyapunov_func}) and then  combine them together to obtain a recursion of the constructed Lyapunov function with proper choice of $c_0-c_4$.

\subsection{The Recursion of Optimality Gap}
\begin{Lem}\label{Lem_opt_gap}
Suppose Assumption \ref{Ass_smoothness}-\ref{Ass_algoirthm} hold. Then, we have for all $k>0$,
\begin{equation}\label{Cons_error_X}
\begin{aligned}
\mathbb{E}\left[ \left\| \bar{x}_{k+1}-x^* \right\| ^2|\mathcal{F}_k \right] &\leqslant \left( 1-\alpha \mu \right)  \left\| \bar{x}_k-x^* \right\| ^2 -\left( 2\alpha -8\alpha ^2L \right) \left( f\left( \bar{x}_k \right) -f\left( x^* \right) \right) 
\\
&+\left( 1-r \right) \frac{\left( 4\alpha ^2L^2+\alpha L \right)}{n}\left\| \hat{X}_k-\mathbf{1}_n\bar{x}_k \right\| ^2 +\frac{2\alpha ^2\sigma _k}{n}.
\end{aligned}
\end{equation}
\end{Lem}

\begin{proof}
Using (\ref{Def_x_y_bar}), we have
$$
\bar{x}_{k+1}=\bar{x}_k-\alpha \bar{y}_k,
$$
which implies that 
\begin{equation*}
\begin{aligned}
\left\| \bar{x}_{k+1}-x^* \right\| ^2=\left\| \bar{x}_k-x^* \right\| ^2-2\alpha \left< \bar{y}_k, \bar{x}_k-x^* \right> +\left\| \bar{y}_k \right\| ^2.
\end{aligned}
\end{equation*}
Then by lemma \ref{Lem_X_by_RW} and \ref{Expt_gradient}, we get
\begin{equation*}
\begin{aligned}
&\mathbb{E}\left[ \left\| \bar{x}_{k+1}-x^* \right\|^{2}|\mathcal{F}_{k} \right]
\\
&\leqslant \left\| \bar{x}_k-x^* \right\|^{2}-2\alpha \mathbb{E}\left[ \left< \bar{y}_k, \bar{x}_k-x^* \right> |\mathcal{F}_{k} \right] +\alpha ^2\mathbb{E}\left[ \left\| \bar{y}_k \right\|^{2}|\mathcal{F}_{k} \right]
\\
&=\left\| \bar{x}_k-x^* \right\|^{2}-2\alpha \left< r\nabla f\left( \bar{x}_k \right) +\left( 1-r \right) g_k,\bar{x}_k-x^* \right> +\alpha ^2\mathbb{E}\left[ \left\| \bar{y}_k \right\|^{2}|\mathcal{F}_{k} \right]
\\
&\overset{(a)}{\leqslant}\left( 1-\alpha \mu r \right) \left\| \bar{x}_k-x^* \right\|^{2}-2\alpha r\left( f\left( \bar{x}_k \right) -f\left( x^* \right) \right)
-2\alpha \left( 1-r \right) \left< g_k ,\bar{x}_k-x^* \right> +\alpha ^2\mathbb{E}\left[ \left\| \bar{y}_k \right\|^{2}|\mathcal{F}_{k} \right]
\\
&\overset{(b)}{\leqslant}\left( 1-\alpha \mu \right) \left\| \bar{x}_k-x^* \right\|^{2}-2\alpha \left( f\left( \bar{x}_k \right) -f\left( x^* \right) \right) +\frac{\alpha L\left( 1-r \right)}{n}\left\| \hat{X}_k-\mathbf{1}_n\bar{x}_k \right\| ^2+\alpha ^2\mathbb{E}\left[ \left\| \bar{y}_k \right\|^{2}|\mathcal{F}_{k} \right],
\end{aligned}
\end{equation*}
where in the inequality (a) we have used the fact that $f$ is $\mu$-strongly convex:
\begin{equation*}
\begin{aligned}
-2\alpha \left< r\nabla f\left( \bar{x}_k \right) ,\bar{x}_k-x^* \right> \,\,\leqslant -2r\alpha \left(f\left( \bar{x}_k \right) -f\left( x^* \right) \right) -r\alpha \mu \left\| \bar{x}_k-x^* \right\| ^2,
\end{aligned}
\end{equation*}
and for the inequality (b), noticing that {$\hat{X}_k:=S_kX_k=\left[ \hat{x}_{1,k}^{T},\hat{x}_{2,k}^{T},\cdots ,\hat{x}_{n,k}^{T} \right] ^T\in \mathbb{R}^{n\times d}$}, and the fact that each $f_i$ is $\mu$-strongly convex and $L$-smooth, we have:
\begin{equation*}
\begin{aligned}
&-2\alpha (1-r)\left< g_k, \bar{x}_k-x^* \right> 
\\
&=-2\alpha \left( 1-r \right) \left< \frac{\mathbf{1}_n^T}{n}\left( \mathbf{I}_n\otimes \frac{\mathbf{1}_m^T}{m} \right) \nabla F\left( \hat{X}_k\otimes \mathbf{1}_m \right) ,\bar{x}_k-x^* \right> 
\\
&=-2\alpha \left( 1-r \right) \frac{1}{n}\sum_{i=1}^n{\left< \nabla f_i\left( \hat{x}_{i,k} \right) ,\bar{x}_k-x^* \right>}
\\
&=-2\alpha \left( 1-r \right) \frac{1}{n}\sum_{i=1}^n{\left( \left< \nabla f_i\left( \hat{x}_{i,k} \right) ,\hat{x}_{i,k}-x^* \right> +\left< \nabla f_i\left( \hat{x}_{i,k} \right) ,\bar{x}_k-\hat{x}_{i,k} \right> \right)}
\\
&\leqslant 2\alpha \left( 1-r \right) \frac{1}{n}\sum_{i=1}^n{\left( f_i\left( x^* \right) -f_i\left( \hat{x}_{i,k} \right) -\frac{\mu}{2}\left\| \hat{x}_{i,k}-x^* \right\| ^2 \right)}
\\
&+2\alpha \left( 1-r \right) \frac{1}{n}\sum_{i=1}^n{\left( f_i\left( \hat{x}_{i,k} \right) -f_i\left( \bar{x}_k \right) +\frac{L}{2}\left\| \hat{x}_{i,k}-\bar{x}_k \right\| ^2 \right)}
\\
&=-2\alpha \left( 1-r \right) \left[ f\left( \bar{x}_k \right) -f\left( x^* \right) \right] -\alpha \mu \left( 1-r \right) \left\| \bar{x}_k-x^* \right\| +\frac{\alpha L\left( 1-r \right)}{n}\left\| \hat{X}_k-\mathbf{1}_n\bar{x}_k \right\| ^2.
\end{aligned}
\end{equation*}
Further, we can bound the last term in the inequality (b) (note that $\nabla f(x^*)=0$)
\begin{equation}\label{Eq_y_bar}
\begin{aligned}
&\mathbb{E}\left\| \bar{y}_k-\nabla f\left( x^* \right) |\mathcal{F}_k \right\|^{2}
\\
&\overset{\left( \ref{Def_x_y_bar} \right)}{=}\mathbb{E}\left[ \left\| \frac{\mathbf{1}_{n}^{T}}{n}S_k\nabla F\left( X_k \right) +\frac{\mathbf{1}_{n}^{T}}{n}S_kC_{k-1}Y_{k-1}-\frac{\mathbf{1}_{n}^{T}}{n}S_k\nabla F\left( X_{k-1} \right) -\nabla f\left( x^* \right) \right\|^{2}|\mathcal{F}_k \right] 
\\
&\leqslant 2\mathbb{E}\left[ \left\| \frac{\mathbf{1}_{n}^{T}}{n}S_k\left[ \nabla F\left( X_k \right) -\nabla F\left( \mathbf{1}_M\bar{x}_k \right) +\nabla F\left( \mathbf{1}_M\bar{x}_k \right) -\nabla F\left( \mathbf{1}_Mx^* \right) \right] \right\|^{2}|\mathcal{F}_k \right] 
\\
&+2\underset{:=\sigma _k /n}{\underbrace{\mathbb{E}\left[ \left\| \frac{\mathbf{1}_{n}^{T}}{n}S_k\nabla F\left( \mathbf{1}_Mx^* \right) +\frac{\mathbf{1}_{n}^{T}}{n}S_kC_{t_k}Y_{t_k}-\frac{\mathbf{1}_{n}^{T}}{n}S_k\nabla F\left( X_{t_k} \right) -\nabla f\left( x^* \right) \right\| ^2|\mathcal{F}_k \right] }}
\\
&\leqslant4\underset{S_1}{\underbrace{\mathbb{E}\left[ \left\| \frac{\mathbf{1}_{n}^{T}}{n}S_k\left( \nabla F\left( X_k \right) -\nabla F\left( \mathbf{1}_M\bar{x}_k \right) \right) \right\| ^2|\mathcal{F}_k \right] }}
\\
&+4\underset{S_2}{\underbrace{\mathbb{E}\left[ \left\| \frac{\mathbf{1}_{n}^{T}}{n}S_k\left( \nabla F\left( \mathbf{1}_M\bar{x}_k \right) -\nabla F\left( \mathbf{1}_Mx^* \right) \right) \right\| ^2|\mathcal{F}_k \right] }}+\frac{2\sigma _k}{n}.
\end{aligned}
\end{equation}
In what follows, under Assumption \ref{Expected smoothness}, we bound the terms $S_1$ and $S_2$ respectively. For $S_1$, we have
\begin{equation}\label{Eq_s_1}
\begin{aligned}
&\mathbb{E}\left[ \left\| \frac{\mathbf{1}_{n}^{T}}{n}S_k\left( \nabla F\left( X_k \right) -\nabla F\left( \mathbf{1}_M\bar{x}_k \right) \right) \right\| ^2|\mathcal{F}_k \right] 
\\
&\overset{(\ref{Eq_X_by_X_hat})}{=}\mathbb{E}\left[ \left\| \frac{\mathbf{1}_{n}^{T}}{n}S_k\left( \nabla F\left( \hat{X}_k\otimes \mathbf{1}_m \right) -\nabla F\left( \mathbf{1}_M\bar{x}_k \right) \right) \right\| ^2|\mathcal{F}_k \right] 
\\
&\leqslant \frac{1}{n}\sum_{i=1}^n{\frac{1}{m}\sum_{j=1}^m{\left\| \nabla f_{ij}\left( \hat{x}_{i,k} \right) -\nabla f_{ij}\left( \bar{x}_k \right) \right\| ^2}}\overset{(\ref{Eq_Expe_smooth})}{\leqslant} \frac{L^2}{n}\left\| \hat{X}_k-\mathbf{1}_n\bar{x}_k \right\| ^2,
\end{aligned}
\end{equation}
and for $S_2$, we have
\begin{equation}\label{Eq_s_2}
\begin{aligned}
&\mathbb{E}\left[ \left\| \frac{\mathbf{1}_{n}^{T}}{n}S_k\left( \nabla F\left( \mathbf{1}_M\bar{x}_k \right) -\nabla F\left( \mathbf{1}_Mx^* \right) \right) \right\| ^2|\mathcal{F}_k \right] 
\\
&\leqslant \frac{1}{n}\sum_{i=1}^n{\frac{1}{m}\sum_{j=1}^m{\left\| \nabla f_{ij}\left( \bar{x}_k \right) -\nabla f_{ij}\left( x^* \right) \right\| ^2}}\overset{(\ref{Eq_Expe_smooth})}{\leqslant} 2L\left( f\left( \bar{x}_k \right) -f\left( x^* \right) \right).
\end{aligned}
\end{equation}
Then, combining \eqref{Eq_y_bar}, \eqref{Eq_s_1} and \eqref{Eq_s_2}, we get
\begin{equation}
\begin{aligned}
\mathbb{E}\left\| \bar{y}_k-\nabla f\left( x^* \right) |\mathcal{F}_k \right\| ^2\leqslant 4\left( 1-r \right) \frac{L^2}{n}\rho _W\left\| \hat{X}_k-\mathbf{1}_n\bar{x}_k \right\| ^2+8L\left( f\left( \bar{x}_k \right) -f\left( x^* \right) \right) +\frac{2\sigma _k}{n}.
\end{aligned}
\end{equation}
which completes the proof.
\end{proof}

\subsection{The Recursion of Consensus Error}
 For simplicity, we first recall the following notations:
\begin{equation*}\label{def_rho_beta} 
\begin{aligned}
r=P\left( W_k=\mathbf{J}_m \right),\quad
\rho _{r,W}=\left( 1-r \right)\rho _W,\quad
\rho _W=\left\| W-\mathbf{J}_n \right\| ^2.
\end{aligned}
\end{equation*}
Then, we bound the consensus error for two types of algorithms respectively. In particular, the results for the algorithms adopting only variance-reduction schemes are provided in Lemma \ref{Lem_Cons_error_1} while the results for those adopting both variance-reduction and gradient-tracking schemes are given in Lemma \ref{Lem_Cons_error_GT}.


\begin{Lem}\label{Lem_Cons_error_1}
Suppose Assumption \ref{Ass_smoothness}-\ref{Ass_algoirthm} hold. Then, if we set $C_k=\mathbf{I}_n\otimes V_k$, we have for all $k>0$
\begin{equation}\label{Cons_error_X} 
\begin{aligned}
&\mathbb{E}\left[ \left\| \hat{X}_{k+1}-\mathbf{1}_n\bar{x}_{k+1} \right\| ^2|\mathcal{F}_k \right] 
\\
&\leqslant \left( \frac{1+\rho _{r,W}}{2}+\frac{4\alpha ^2L^2\rho _{r,W}\left( 1+\rho _{r,W} \right)}{1-\rho _{r,W}} \right) \left\| \hat{X}_k-\mathbf{1}_n\bar{x}_k \right\| ^2 
\\
&+\frac{8n\alpha ^2L\rho _{r,W}\left( 1+\rho _{r,W} \right)}{1-\rho _{r,W}}\left( f\left( \bar{x}_k \right) -f\left( x^* \right) \right) +\alpha ^2\rho _{r,W}\left( \frac{4n\zeta ^*}{1-\rho _{r,W}}+n\sigma _k \right) .
\end{aligned}
\end{equation}
\end{Lem}

\begin{proof}
First of all, recalling the definition of $\hat{X}_k$ in (\ref{Def_X_Y_hat}), $\bar{x}_k$ in (\ref{Def_x_y_bar}) and $\rho _{r,W}:=\left( 1-r \right) \rho _W$, we have
\begin{equation}
\begin{aligned}
&\mathbb{E}\left[ \left\| \hat{X}_{k+1}-\mathbf{1}_n\bar{x}_{k+1} \right\| ^2|\mathcal{F}_k \right] 
\\
&=\left( 1-r \right) \mathbb{E}\left[ \left\| W_k\left[ \hat{X}_k-\alpha \hat{Y}_k \right] -\mathbf{1}_n\left( \bar{x}_k-\alpha \bar{y}_k \right) \right\| ^2|\mathcal{F}_k, W_k=W \right] 
\\
&+r\mathbb{E}\left[ \left\| W\left[ \hat{X}_k-\alpha \hat{Y}_k \right] -\mathbf{1}_n\left( \bar{x}_k-\alpha \bar{y}_k \right) \right\| ^2|\mathcal{F}_k, W_k=\mathbf{J}_n \right] 
\\
&=\left( 1-r \right) \mathbb{E}\left[ \left\| W\left[ \hat{X}_k-\alpha \hat{Y}_k \right] -\mathbf{1}_n\left( \bar{x}_k-\alpha \bar{y}_k \right) \right\| ^2|\mathcal{F}_k \right] 
\\
&\leqslant \left( 1-r \right) \left\| W-\mathbf{J}_n \right\|_2 ^2\mathbb{E}\left[ \left\| \left( \hat{X}_k-\mathbf{1}_n\bar{x}_k \right) -\alpha S_kY_k \right\| ^2|\mathcal{F}_k \right] .
\\
&=\rho _{r,W}\mathbb{E}\left[ \left\| \left( \hat{X}_k-\mathbf{1}_n\bar{x}_k \right) -\alpha S_kY_k \right\| ^2|\mathcal{F}_k \right] .
\end{aligned}
\end{equation}
Then, by Lemma \ref{Lem_Y_t}, we get
\begin{equation}\label{Eq_Consensus_X_hat}
\begin{aligned}
&\mathbb{E}\left[ \left\| \hat{X}_{k+1}-\mathbf{1}_n\bar{x}_{k+1} \right\| ^2|\mathcal{F}_k \right] 
\\
&\leqslant \rho _{r,W}\mathbb{E}\left[ \left\| \left( \hat{X}_k-\mathbf{1}_n\bar{x}_k \right) -\alpha S_k\left( C_{t_k}Y_{t_k}+\nabla F\left( X_k \right) -\nabla F\left( X_{t_k} \right) \right) \right\| ^2|\mathcal{F}_k \right] 
\\
&\overset{\left( c \right)}{\leqslant}\rho _{r,W}\mathbb{E}\left[ \left\| \left( \hat{X}_k-\mathbf{1}_n\bar{x}_k \right) -\alpha \left( S_k\nabla F\left( X_k \right) -S_k\nabla F\left( \mathbf{1}_Mx^* \right) +\left( \mathbf{I}_n\otimes \frac{\mathbf{1}^T}{m} \right) \nabla F\left( \mathbf{1}_Mx^* \right) \right) \right\| ^2|\mathcal{F}_k \right] 
\\
&+\alpha ^2\rho _{r,W}\underset{:=n\sigma _k}{\underbrace{\mathbb{E}\left[ \left\| S_k\left( C_{t_k}Y_{t_k}-\nabla F\left( X_{t_k} \right) \right) +S_k\nabla F\left( \mathbf{1}_Mx^* \right) -\left( \mathbf{I}_n\otimes \frac{\mathbf{1}^T}{m} \right) \nabla F\left( \mathbf{1}_Mx^* \right) \right\| ^2|\mathcal{F}_k \right] }}
\\
&\overset{\left( d \right)}{\leqslant}\rho _{r,W}\left( 1+\beta \right) \left\| \hat{X}_k-\mathbf{1}_n\bar{x}_k \right\| ^2 +2\alpha ^2\rho _{r,W}\left( 1+\frac{1}{\beta} \right) \left( \mathbb{E}\left[ \left\| S_k\nabla F\left( X_k \right) -S_k\nabla F\left( \mathbf{1}_Mx^* \right) \right\| ^2|\mathcal{F}_k \right] \right) 
\\
&+2\alpha ^2\rho _{r,W}\left( 1+\frac{1}{\beta} \right) n\zeta ^*+\alpha ^2\rho _{r,W}n\sigma _k,
\end{aligned}
\end{equation}
where $t_{k} < k$ denotes the latest iteration before $k$ that $V_{t_k}=\mathbf{J}_m$, and thus $C_{t_k}=\mathbf{I}_n\otimes \mathbf{J}_m$. Besides, in the inequality (c) we have used the fact that $\mathbb{E}\left[ \left\| X-\mathbb{E}\left[ X \right] \right\| ^2 \right] \leqslant \mathbb{E}\left[ \left\| X \right\| ^2 \right]$, and in the inequality (d) we have used Young inequality with parameter $\beta =\frac{1-\rho _{{r,W}}}{2\rho _{{r,W}}}$.  
Furthermore, noticing that
\begin{equation}
\begin{aligned}
&\mathbb{E}\left[ \left\| S_k\nabla F\left( X_k \right) -S_k\nabla F\left( \mathbf{1}_Mx^* \right) \right\| ^2|\mathcal{F}_k \right] 
\\
&=\mathbb{E}\left[ \left\| S_k\nabla F\left( X_k \right) -S_k\nabla F\left( \mathbf{1}_M\bar{x}_k \right) +S_k\nabla F\left( \mathbf{1}_M\bar{x}_k \right) -S_k\nabla F\left( \mathbf{1}_Mx^* \right) \right\| ^2|\mathcal{F}_k \right] 
\\
&\leqslant 2L^2\left\| \hat{X}_k-\mathbf{1}_n\bar{x}_k \right\| ^2 +4nL\left( f\left( \bar{x}_k \right) -f\left( x^* \right) \right),
\end{aligned}
\end{equation}
we thus complete the proof of the lemma.
\end{proof}

\begin{Lem}\label{Lem_Cons_error_GT}
Suppose Assumption \ref{Ass_smoothness}-\ref{Ass_algoirthm} hold. Then, if we set $C_k=W_k\otimes \mathbf{J}_m$,
we have for all $k>0$
\begin{equation}\label{Cons_error_X} 
\begin{aligned}
\mathbb{E}\left[ \left\| \hat{X}_{k+1}-\mathbf{1}_n\bar{x}_{k+1} \right\| ^2|\mathcal{F}_k \right] \leqslant \rho _{r,W}\left\| \hat{X}_k-\mathbf{1}_n\bar{x}_k \right\| ^2 +\frac{\alpha ^2\rho _{r,W}\left( 1+\rho _{r,W} \right)}{1-\rho _{r,W}}\mathbb{E}\left[ \left\| \hat{Y}_k-\mathbf{1}_n\bar{y}_k \right\| ^2|\mathcal{F}_k \right]. 
\end{aligned}
\end{equation}
\end{Lem}
\begin{proof}
Recalling the definition of $\hat{X}_k$ in (\ref{Def_X_Y_hat}) and $\bar{x}_k$ in (\ref{Def_x_y_bar}), we have
\begin{equation}\label{Eq_X_hat_x_bar}
\begin{aligned}
&\mathbb{E}\left[ \left\| \hat{X}_{k+1}-\mathbf{1}_n\bar{x}_{k+1} \right\| ^2|\mathcal{F}_k \right] 
\\
&= \left( 1-r \right) \mathbb{E}\left[ \left\| \left( W_k-\mathbf{J}_n \right) \left( \hat{X}_k-\mathbf{1}_n\bar{x}_k \right) -\alpha \left( W-\mathbf{J}_n \right) \left( \hat{Y}_k-\mathbf{1}_n\bar{y}_k \right) \right\| ^2|\mathcal{F}_k \right] 
\\
&\leqslant \rho _W\left( 1-r \right) \left( 1+\beta \right) \left\| \hat{X}_k-\mathbf{1}_n\bar{x}_k \right\| ^2 +\alpha ^2\rho _W\left( 1-r \right) \left( 1+\frac{1}{\beta} \right) \mathbb{E}\left[ \left\| \hat{Y}_k-\mathbf{1}_n\bar{y}_k \right\| ^2|\mathcal{F}_k \right] ,
\end{aligned}
\end{equation}
where in the last inequality we have used Young inequality. Setting $\beta =\frac{1-\rho _{r,W}}{2\rho _{r,W}}$ completes the proof.
\end{proof}

\subsection{The Recursion of Variance Reduction Error}
In the following lemma, we show that the VR error will be decaying when the variance-reduction methods such as SAGA, L-SVRG, are adopted. For ease of presentation, we recall the notions as follows:
$$
q:=\rho \left( \mathbb{E}\left[ \varLambda _{k+1}|V_k=\mathbf{J}_m \right] \right) ,\quad r:=P\left( W_k=\mathbf{J}_n \right).
$$

\begin{Lem}\label{Lem_VR_error}
Suppose Assumption \ref{Ass_smoothness}-\ref{Ass_algoirthm} hold. Then, we have for all $k>0$
\begin{equation}
\begin{aligned}
&\mathbb{E}\left[ \nabla F\left( X_k \right) -\nabla F\left( \mathbf{1}_Mx^* \right) |\mathcal{F}_{k},V_{k-1}=\mathbf{J}_m \right] 
\\
&\leqslant \left( 1-q \right) \mathbb{E}\left[ \left\| \left[ \nabla F\left( X_{k-1} \right) -\nabla F\left( \mathbf{1}_Mx^* \right) \right] \right\|^{2}|\mathcal{F}_{k},V_{k-1}=\mathbf{J}_m \right] 
\\
&+2\left( 1-r \right) mqL^2 \left\| \hat{X}_k-\mathbf{1}_n\bar{x}_k \right\| ^{2} +4MqL\left( f\left( \bar{x}_k \right) -f\left( x^* \right) \right).
\end{aligned}
\end{equation}
\end{Lem}

\begin{proof}
Using Lemma \ref{Lem_Y_t}, we have
$$
\nabla F\left( X_k \right) = \left( \mathbf{I}_M-\varLambda _k \right) \nabla F\left( X_{k-1} \right) +\varLambda _k\nabla F\left( \hat{X}_k\otimes \mathbf{1}_m \right).
$$
Then, we further obtain
\begin{equation*}
\begin{aligned}
&\mathbb{E}\left[ \left\| \nabla F\left( X_k \right) -\nabla F\left( \mathbf{1}_Mx^* \right) \right\|^{2}|\mathcal{F}_{k}, V_{k-1}=\mathbf{J}_m \right] 
\\
&=\mathbb{E}\left[ \left\| \left( \mathbf{I}_M-\varLambda _k \right) \nabla F\left( X_{k-1} \right) +\varLambda _k\nabla F\left( \hat{X}_k\otimes \mathbf{1}_m \right) -\nabla F\left( \mathbf{1}_Mx^* \right) \right\|^{2}|\mathcal{F}_{k}, V_{k-1}=\mathbf{J}_m \right] 
\\
&=\mathbb{E}\left[ \left\| \begin{array}{c}
	\left(\mathbf{I}_M-\varLambda _k \right) \left[ \nabla F\left( X_{k-1} \right) -\nabla F\left( \mathbf{1}_Mx^* \right) \right]\\
\end{array} \right\|^{2}|\mathcal{F}_{k}, V_{k-1}=\mathbf{J}_m \right] 
\\
&+\mathbb{E}\left[ \left\| \varLambda _k\left[ \nabla F\left( \hat{X}_k\otimes \mathbf{1}_m \right) -\nabla F\left( \mathbf{1}_Mx^* \right) \right] \right\|^{2}|\mathcal{F}_{k}, V_{k-1}=\mathbf{J}_m \right] 
\\
&\overset{(\ref{Eq_s_1}),(\ref{Eq_s_2})}{\leqslant} \left( 1-q \right) \mathbb{E}\left[ \left\| \nabla F\left( X_{k-1} \right) -\nabla F\left( \mathbf{1}_Mx^* \right) \right\|^{2}|\mathcal{F}_{k}, V_{k-1}=\mathbf{J}_m\, \right] 
\\
&+2\left( 1-r \right) mqL^2\mathbb{E}\left[ \left\| \hat{X}_k-\mathbf{1}_n\bar{x}_k \right\| ^{2}|\mathcal{F}_{k},V_{k-1}=\mathbf{J}_m \right]+4MqL\left( f\left( \bar{x}_k \right) -f\left( x^* \right) \right).
\end{aligned}
\end{equation*}
\end{proof}

\subsection{The Recursion of Delayed Variance Reduction Error}

When the variance-reduction methods with random local loops, such as L-SVRG, are adopted, an additional delay error term will appear and should be dealt with properly for convergence analysis.

\begin{Lem}\label{Lem_VR_error_delay}
Suppose Assumption \ref{Ass_smoothness}-\ref{Ass_algoirthm} hold. Then, we have for all $k>0$,
\begin{equation}
\begin{aligned}
&\mathbb{E}\left[ \left\| \nabla F\left( X_{t_k} \right) -\nabla F\left( \mathbf{1}_Mx^* \right) \right\|^{2}|\mathcal{F}_{k} \right] 
\\
&\leqslant \left( 1-p \right) \mathbb{E}\left[ \left\| \nabla F\left( X_{t_{k-1}} \right) -\nabla F\left( \mathbf{1}_Mx^* \right) \right\|^{2}|\mathcal{F}_{k},V_{k-1}=\mathbf{I}_m \right] 
\\
&+p\mathbb{E}\left[ \left\| \left[ \nabla F\left( X_{k-1} \right) -\nabla F\left( \mathbf{1}_Mx^* \right) \right] \right\|^{2}|\mathcal{F}_{k},V_{k-1}=\mathbf{J}_m \right].
\end{aligned}
\end{equation}
\end{Lem}
The proof for the above lemma is straightforward leveraging the definition of $t_k$ as well as the conditional expectation on the choices of $V_k-1$, which is thus omitted.

\subsection{The Recursion of Gradient Tracking Error}
The following lemma shows that the error caused by the data heterogeneity across devices is decaying with iteration when gradient-tracking schemes are adopted.

\begin{Lem}\label{Lem_GT_error}
Suppose Assumption \ref{Ass_smoothness}-\ref{Ass_algoirthm} hold, if we set $C_k=W_k\otimes \mathbf{J}_m$, then for all $k>0$
\begin{equation}
\begin{aligned}
&\mathbb{E}\left[ \left\| \hat{Y}_{k+1}-\mathbf{1}_n\bar{y}_{k+1} \right\| ^2|\mathcal{F}_k \right] 
\\
&\leqslant \left( \frac{1+\rho _{r,W}}{2}+\frac{8\alpha ^2L^2\rho _{r,W}\left( 1+\rho _{r,W} \right)}{1-\rho _{r,W}} \right) \mathbb{E}\left[ \left\| Y_k-\mathbf{1}_M\bar{y}_k \right\| ^2|\mathcal{F}_k \right] 
\\
&+\frac{4\left( 1+\rho _{r,W} \right) L^2}{1-\rho _{r,W}}\left( 4+\left( 1-r \right) q+4\left( 1-r \right) \alpha ^2L^2 \right) \left\| \hat{X}_k-\mathbf{1}_n\bar{x}_k \right\| ^2 
\\
&+\frac{2\left( 1+\rho _{r,W} \right)}{m\left( 1-\rho _{r,W} \right)}\left( 1-q+\frac{4\alpha ^2L^2}{n} \right) \mathbb{E}\left[ \left\| \nabla F\left( X_{k-1} \right) -\nabla F\left( \mathbf{1}_Mx^* \right) \right\| _{2}^{2}|\mathcal{F}_k \right] 
\\
&+\frac{4n\left( 1+\rho _{r,W} \right)}{1-\rho _{r,W}}\left( 8\alpha ^2L^3+4L+2qL \right) \left( f\left( \bar{x}_k \right) -f\left( x^* \right) \right). 
\end{aligned}
\end{equation}
\end{Lem}
\begin{proof}
Recalling the definition of $\hat{Y}_k$ in (\ref{Def_X_Y_hat}) and $\bar{y}_k$ in (\ref{Def_x_y_bar}), we have
\begin{equation*}
\begin{aligned}
&\mathbb{E}\left[ \left\| \hat{Y}_{k+1}-\mathbf{1}_n\bar{y}_{k+1} \right\| ^2|\mathcal{F}_k \right] 
\\
&=\mathbb{E}\left[ \left\| \left( \mathbf{I}_n-\mathbf{J}_n \right) \left( W_k\hat{Y}_k+S_{k+1}\left( \nabla F\left( X_{k+1} \right) -\nabla F\left( X_k \right) \right) \right) \right\| ^2|\mathcal{F}_k \right] 
\\
&\leqslant \frac{1+\rho _{r,W}}{2}\mathbb{E}\left[ \left\| \hat{Y}_k-\mathbf{1}_n\bar{y}_k \right\| ^2|\mathcal{F}_k \right] +\frac{1+\rho _{r,W}}{1-\rho _{r,W}}\mathbb{E}\left[ \left\| S_{k+1}\left( \nabla F\left( X_{k+1} \right) -\nabla F\left( X_k \right) \right) \right\| ^2|\mathcal{F}_k \right],
\end{aligned}
\end{equation*}
where we have used the Young inequality. Then, we first bound the last term:
\begin{equation*}
\begin{aligned}
&\mathbb{E}\left[ \left\| S_{k+1}\left( \nabla F\left( X_{k+1} \right) -\nabla F\left( X_k \right) \right) \right\| ^2|\mathcal{F}_k \right] 
\\
&=\mathbb{E}\left[ \left\| S_{k+1}\nabla F\left( X_{k+1} \right) -S_{k+1}\nabla F\left( \mathbf{1}_Mx^* \right) +S_{k+1}\nabla F\left( \mathbf{1}_Mx^* \right) -S_{k+1}\nabla F\left( X_k \right) \right\| ^2|\mathcal{F}_k \right] 
\\
&\leqslant 2\mathbb{E}\left[ \left\| S_{k+1}\left( \nabla F\left( \hat{X}_{k+1}\otimes \mathbf{1}_m \right) -\nabla F\left( \hat{X}_k\otimes \mathbf{1}_m \right) +\nabla F\left( \hat{X}_k\otimes \mathbf{1}_m \right) -\nabla F\left( \mathbf{1}_Mx^* \right) \right) \right\| ^2|\mathcal{F}_k \right] 
\\
&+\frac{2}{m}\mathbb{E}\left[ \left\| \nabla F\left( X_k \right) -\nabla F\left( \mathbf{1}_Mx^* \right) \right\| ^2|\mathcal{F}_k \right] 
\\
&\leqslant 4L^2\underset{S_3}{\underbrace{\mathbb{E}\left[ \left\| \hat{X}_{k+1}-\hat{X}_k \right\| ^2|\mathcal{F}_k \right] }}+4\underset{S_4}{\underbrace{\mathbb{E}\left[ \left\| S_{k+1}\nabla F\left( \hat{X}_k\otimes \mathbf{1}_m \right) -S_{k+1}\nabla F\left( \mathbf{1}_Mx^* \right) \right\| ^2|\mathcal{F}_k \right] }}
\\
&+\frac{2}{m}\mathbb{E}\left[ \left\| \nabla F\left( X_k \right) -\nabla F\left( \mathbf{1}_Mx^* \right) \right\| ^2|\mathcal{F}_k \right],
\end{aligned}
\end{equation*}
In what follows, we further bound the terms of $S_3$ and $S_4$, respectively. For $S_3$, we have
\begin{equation*}
\begin{aligned}
&\mathbb{E}\left[ \left\| \hat{X}_{k+1}-\hat{X}_k \right\| ^2|\mathcal{F}_k \right] 
\\
&=\mathbb{E}\left[ \left\| W_k\hat{X}_k-\hat{X}_k-\alpha W_k\hat{Y}_k \right\| ^2|\mathcal{F}_k \right] 
\\
&\leqslant \mathbb{E}\left[ \left\| W_k\hat{X}_k-\hat{X}_k \right\| ^2|\mathcal{F}_k \right] +\alpha ^2\mathbb{E}\left[ \left\| W_k\hat{Y}_k-\mathbf{1}_n\bar{y}_k+\mathbf{1}_n\bar{y}_k \right\| ^2|\mathcal{F}_k \right] -2\alpha \left< W_k\hat{X}_k-\hat{X}_k, W_k\hat{Y}_k \right> 
\\
&=\mathbb{E}\left[ \left\| W_k\hat{X}_k-\hat{X}_k \right\| ^2|\mathcal{F}_k \right] +\alpha ^2\mathbb{E}\left[ \left\| W_k\hat{Y}_k-\mathbf{1}_n\bar{y}_k \right\| ^2|\mathcal{F}_k \right] 
\\
&+n\alpha ^2\mathbb{E}\left[ \left\| \bar{y}_k \right\| ^2|\mathcal{F}_k \right] -2\alpha \left< W_k\hat{X}_k-\hat{X}_k, W_k\hat{Y}_k-\mathbf{1}_n\bar{y}_k \right> 
\\
&\leqslant 2\mathbb{E}\left[ \left\| W_k-\mathbf{I}_n \right\| \right]_2 ^2 \left\| \hat{X}_k-\mathbf{1}_n\bar{x}_k \right\| ^2 +2\alpha ^2\rho _{r,W}\mathbb{E}\left[ \left\| \hat{Y}_k-\mathbf{1}_n\bar{y}_k \right\| ^2|\mathcal{F}_k \right] +n\alpha ^2\mathbb{E}\left[ \left\| \bar{y}_k \right\| ^2|\mathcal{F}_k \right].
\end{aligned}
\end{equation*}
Then, for $S_4$, recalling the relations in (\ref{Eq_s_1}) and (\ref{Eq_s_2}), we have
\begin{equation*}
\begin{aligned}
&\mathbb{E}\left[ \left\| S_{k+1}\nabla F\left( \hat{X}_k\otimes \mathbf{1}_m \right) -S_{k+1}\nabla F\left( \mathbf{1}_Mx^* \right) \right\| ^2|\mathcal{F}_k \right] 
\\
&=\mathbb{E}\left[ \left\| S_{k+1}\left( \nabla F\left( \hat{X}_k\otimes \mathbf{1}_m \right) -\nabla F\left( \mathbf{1}_M\bar{x}_k \right) +\nabla F\left( \mathbf{1}_M\bar{x}_k \right) -\nabla F\left( \mathbf{1}_Mx^* \right) \right) \right\| ^2|\mathcal{F}_k \right] 
\\
&\leqslant 2\mathbb{E}\left[ \left\| S_{k+1}\left( \nabla F\left( \hat{X}_k\otimes \mathbf{1}_m \right) -\nabla F\left( \mathbf{1}_M\bar{x}_k \right) \right) \right\| ^2|\mathcal{F}_k \right] +2\mathbb{E}\left[ \left\| S_{k+1}\left( \nabla F\left( \mathbf{1}_M\bar{x}_k \right) -\nabla F\left( \mathbf{1}_Mx^* \right) \right) \right\| ^2|\mathcal{F}_k \right] 
\\
&{\leqslant}2L^2\left\| \hat{X}_k-\mathbf{1}_n\bar{x}_k \right\| ^2+4nL\left( f\left( \bar{x}_k \right) -f\left( x^* \right) \right).
\end{aligned}
\end{equation*}
Combining all these inequalities above, together with the inequality of $\mathbb{E}\left[ \left\| \bar{y}_k \right\| ^2|\mathcal{F}_k \right]$ in (\ref{Eq_y_bar}), we complete the proof.
\end{proof}
\begin{Rem}
It follows from \ref{Lem_GT_error}, together with Lemma \ref{Lem_Cons_error_1} with Lemma \ref{Lem_Cons_error_GT}, that the error caused by the data heterogeneity across devices is decaying when gradient-tracking methods are adopted, otherwise the heterogeneity constant  $\zeta^*$ as defined in Assumption \ref{Ass_sampling} will not be eliminated.
\end{Rem}

\section{Proofs  in Section \ref{Sec_Con_analysis}}
\label{Appe_Con_analysis}
In this section, we provide the proofs of Theorem \ref{Thm_Without_GT_VR}-\ref{Thm_With_GT_VR},  followed by the corresponding Corollary \ref{Cor_Complexcity_1}-\ref{Cor_Complexcity_3} for the strongly convex and smooth objectives.

\subsection{Proof of Theorem \ref{Thm_Without_GT_VR}} \label{Proof of Thm1}

\begin{proof}
For algorithms $\mathcal{A}\left( \cdot , \cdot , C_k\equiv \mathbf{I}_M \right) $, recalling the following parameter settings of Lyapunov function (\ref{Lyapunov_func})
$$
c_0=1, \,\,c_1= \frac{8\left( 1-r \right)\alpha L\left( 4\alpha L+1 \right)}{n\left( 1-\rho _{r,W} \right)},\,\, c_2=c_3=c_4=0,
$$
and using Lemma \ref{Lem_sigma_k}, \ref{Lem_opt_gap}, \ref{Lem_Cons_error_1}, we can derive by properly rearranging terms that
\begin{equation}
    \begin{aligned}
&\mathbb{E}\left[ T_{k+1} \right] =c_0\mathbb{E}\left[ \left\| \bar{x}_{k+1}-x^* \right\| ^2 \right] +c_1\mathbb{E}\left[ \left\| \hat{X}_{k+1}-\mathbf{1}_n\bar{x}_{k+1} \right\| ^2 \right] 
\\
&\leqslant \left( 1-\alpha \mu \right) c_0\mathbb{E}\left[ \left\| \bar{x}_k-x^* \right\| _{2}^{2} \right] +\left( 1-\frac{1-\rho _{{r,W}}}{8} \right) \mathbb{E}\left[ \left\| \hat{X}_k-\mathbf{1}_n\hat{x}_k \right\| ^2 \right] 
\\
&+\underset{e_1}{\underbrace{\left( \left( 1-r \right) \frac{64\alpha ^3L^2\left( 4\alpha L+1 \right) \rho _{r,W}\left( 1+\rho _{r,W} \right)}{\left( 1-\rho _{r,W} \right) ^2}-\left( 2\alpha -8\alpha ^2L \right) \right) }}\mathbb{E}\left[ f\left( \bar{x}_k \right) -f\left( x^* \right) \right]
\\
&+\underset{e_2}{\underbrace{\left( \left( 1-r \right) \frac{\left( 4\alpha ^2L^2+\alpha L \right)}{n}-\left( 1-r \right) \frac{8\alpha L\left( 4\alpha L+1 \right)}{n\left( 1-\rho _{r,W} \right)}\frac{1-\rho _{{r,W}}}{8} \right) }}\mathbb{E}\left[ \left\| \hat{X}_k-\mathbf{1}_n\bar{x}_k \right\| ^2 \right] 
\\
&+\frac{2\alpha ^2\sigma ^*}{nb}+\left( 1-r \right) \frac{8\alpha ^3L\left( 4\alpha L+1 \right) \rho _{r,W}}{\left( 1-\rho _{r,W} \right)}\left( \frac{4\zeta ^*}{1-\rho _{r,W}}+\frac{\sigma ^*}{b} \right) .
    \end{aligned}
\end{equation}
It can be thus verified that if the step-size satisfies
\begin{equation}\label{stepsize_1}
    \begin{aligned}
\alpha \leqslant \min \left\{ \frac{1}{5L}, \frac{1-\rho _{{r,W}}}{4L\sqrt{\rho _{r,W}\left( 1+\rho _{r,W} \right)}}, \frac{1-\rho _{r,W}}{12L\sqrt{2\rho _{r,W}\left( 1+\rho _{r,W} \right)}} \right\} =\min \left\{ \mathcal{O}\left( \frac{1}{L} \right) , \mathcal{O}\left( \frac{1-\rho _{r,W}}{L\sqrt{\rho _{r,W}}} \right) \right\},
\end{aligned}
\end{equation}
then the coefficient $e_0, e_1\leqslant0$, and we thus have
\begin{equation*}
\begin{aligned}
\mathbb{E}\left[ T_{k+1} \right] \leqslant \left( 1-\min \left\{ \alpha \mu , \frac{1-\rho _{{r,W}}}{8} \right\} \right) \mathbb{E}\left[ T_k \right] +\frac{2\alpha ^2\sigma ^*}{nb}+\left( 1-r \right) \frac{16\alpha ^3L\rho _{r,W}}{1-\rho _{r,W}}\left( \frac{4\zeta ^*}{1-\rho _{r,W}}+\frac{\sigma ^*}{{b}} \right),
\end{aligned}
\end{equation*}
which completes the proof.
\end{proof}

\subsection{{Proof of  Corollary \ref{Cor_Complexcity_1}}}

\begin{proof}
According to Theorem \ref{Thm_Without_GT_VR}, we have
\begin{equation}
\begin{aligned}
\mathbb{E}\left[ T_k \right] \leqslant \left( 1-\min \left\{ \alpha \mu ,1-\frac{1-\rho _{r,W}}{8} \right\} \right) ^k\mathbb{E}\left[ T_0 \right] +\frac{1}{\min \left\{ \alpha \mu ,\frac{1-\rho _{r,W}}{8} \right\}}\left( \alpha ^2D_1+\alpha ^3D_2 \right),
\end{aligned}
\end{equation}
where 
$$
D_1=\frac{2\sigma ^*}{nb}, \,\,
D_2=\left( 1-r \right) \frac{8L\rho _{r,W}\left( 4\alpha L+1 \right)}{1-\rho _{r,W}}\left( \frac{4\zeta ^*}{1-\rho _{r,W}}+\frac{\sigma ^*}{{b}} \right).
$$
Let $\gamma$ denote the resulting upper bound on step size $\alpha$ in (\ref{stepsize_1}), and give another extra upper bound as follow:
\begin{equation}
\begin{aligned}
\alpha \leqslant \min \left\{ \gamma , \frac{\ln \left( \max \left\{ 2, \frac{\mu ^2K^2}{D_1}, \frac{\mu ^3K^3}{D_2} \right\} \right)}{\mu K} \right\},
\end{aligned}
\end{equation}
Then, we consider two possible situations.
First, if
$
\gamma \leqslant \frac{\ln \left( \max \left\{ 2, \min \left\{ \frac{\mu ^2K^2}{D_1}, \frac{\mu ^3K^3}{D_2} \right\} \right\} \right)}{\mu K}
$
holds, then we let $\alpha =\gamma$, and obtain
\begin{equation}
\begin{aligned}
&\mathbb{E}\left[ T_k \right] \leqslant \left( 1-\min \left( \gamma \mu ,\frac{1-\rho _{r,W}}{8} \right) \right) ^K\mathbb{E}\left[ T_0 \right] +\frac{1}{\min \left( \gamma \mu ,\frac{1-\rho _{r,W}}{8} \right)}\left( \gamma ^2D_1+\gamma ^3D_2 \right) 
\\
&\leqslant \exp \left( -\min \left( \gamma \mu ,\frac{1-\rho _{r,W}}{8} \right) K \right) \mathbb{E}\left[ T_0 \right] +\frac{1}{\min \left( \gamma \mu ,\frac{1-\rho _{r,W}}{8} \right)}\left( \gamma ^2D_1+\gamma ^3D_2 \right) 
\\
&\leqslant \exp \left( -\min \left( \gamma \mu ,\frac{1-\rho _{r,W}}{8} \right) K \right) +\frac{D_1}{\mu ^2K}+\frac{D_2}{\mu ^3K^2}+\frac{D_1}{\left( 1-\rho _{r,W} \right) \mu ^2K^2}+\frac{D_2}{\left( 1-\rho _{r,W} \right) \mu ^3K^3}.
\end{aligned}
\end{equation}
To get $\mathbb{E}\left[ T_K \right] \leqslant \varepsilon$, we have
\begin{equation}
\begin{aligned}
K\geqslant {\mathcal{O}}\left( \left( \frac{1}{\gamma \mu}+\frac{1}{1-\rho _{r,W}} \right) \log \frac{\mathbb{E}\left[ V_0 \right]}{\varepsilon}+\frac{\sigma ^*}{nb\mu ^2\varepsilon}+\sqrt{\frac{\rho _{r,W}L}{\mu ^3\left( 1-\rho _{r,W} \right) \varepsilon}\left( \frac{4\zeta ^*}{1-\rho _{r,W}}+\frac{\sigma ^*}{{b}} \right)} \right) \geqslant \frac{1}{1-\rho _{r,W}}.
\end{aligned}
\end{equation}

Second, if 
$
\left\{ \gamma ,\frac{1-\rho _{r,W}}{8\mu} \right\}  \geqslant \frac{\ln \left( \max \left\{ 2,\min \left\{ \frac{\mu ^2K^2}{D_1},\frac{\mu ^3K^3}{D_2} \right\} \right\} \right)}{\mu K}
$
holds, we set
$$
\alpha =\frac{\ln \left( \max \left\{ 2, \min \left\{ \frac{\mu ^2K^2}{D_1}, \frac{\mu ^3K^3}{D_2} \right\} \right\} \right)}{\mu K},
$$
and then obtain
\begin{equation*}
\begin{aligned}
&\mathbb{E}\left[ T_K \right] \leqslant \left( 1-\alpha \mu \right) ^K\mathbb{E}\left[ T_0 \right] +\frac{1}{\mu}\left( \gamma D_1+\gamma ^2D_2 \right) 
\\
&=\exp \left( -\frac{\ln \left( \max \left\{ 2, \min \left\{ \frac{\mu ^2K^2}{D_1}, \frac{\mu ^3K^3}{D_2} \right\} \right\} \right)}{\mu K}\mu K \right) +\frac{1}{\mu}\left( \gamma D_1+\gamma ^2D_2 \right) 
\\
&={\mathcal{O}}\left( \frac{D_1}{\mu ^2K}+\frac{D_2}{\mu ^3K^2} \right).
\end{aligned}
\end{equation*}

In all, substituting the value of $\gamma$ and hiding the constants,  we have $\mathbb{E}\left[ T_K \right] \leqslant \varepsilon$ after at most the following number of iterations:
$$
K\geqslant {\mathcal{O}}\left( \left( \frac{L}{\mu \left( 1-\rho _{r,W} \right)} \right) \log \frac{\mathbb{E}\left[ T_0 \right]}{\varepsilon}+\frac{\sigma ^*}{nb\mu ^2\varepsilon}+\sqrt{\frac{\rho _{r,W}L}{\mu ^3\left( 1-\rho _{r,W} \right) \varepsilon}\left( \frac{4\zeta ^*}{1-\rho _{r,W}}+\frac{\sigma ^*}{{b}} \right)} \right) .
$$
which completes the proof.
\end{proof}

\subsection{{Proof of Theorem \ref{Thm_With_VR_only}}}

\begin{proof} \label{Proof_2}

In order to obtain the convergence rate, we choose the following parameters for the Lyapunov function:
\begin{equation}
\begin{aligned}
c_0=1, c_1=\frac{20L\alpha}{n\left( 1-\rho _{r,W} \right)}, c_2=\frac{5\alpha ^2}{Mp},  c_3=\frac{16\alpha ^2}{Mq}, c_4=0.
\end{aligned}
\end{equation}

Then, using Lemma \ref{Lem_sigma_k}, \ref{Lem_opt_gap}, \ref{Lem_Cons_error_1}, \ref{Lem_VR_error} and \ref{Lem_VR_error_delay}, if
$$
\frac{1+\rho _{r,W}}{2}+\frac{4\alpha ^2L^2\rho _{r,W}\left( 1+\rho _{r,W} \right)}{1-\rho _{r,W}}\leqslant \frac{3+\rho _{r,W}}{4},
$$
which implies
$$
\alpha \leqslant \frac{1-\rho _{r,W}}{4L\sqrt{\rho _{r,W}\left( 1+\rho _{r,W} \right)}},
$$
we can derive the recursion of Lyapunov function as follows:
\begin{equation*}
\begin{aligned}
\mathbb{E}\left[ T_{k+1} \right] &=\mathbb{E}\left[ \left\| \bar{x}_{k+1}-x^* \right\| _{2}^{2} \right] +c_1\mathbb{E}\left[ \left\| \hat{X}_{k+1}-\mathbf{1}_n\bar{x}_{k+1} \right\| ^2 \right] 
\\
&+c_2\mathbb{E}\left[ \left\| \nabla F\left( X_{t_k} \right) -\nabla F\left( \mathbf{1}_Mx^* \right) \right\| _{2}^{2} \right] +c_3\mathbb{E}\left[ \nabla F\left( X_k \right) -\nabla F\left( \mathbf{1}_Mx^* \right) \right] 
\\
&\leqslant \underset{e_1}{\underbrace{\left( \frac{160\alpha ^3L^2\rho _{r,W}\left( 1+\rho _{r,W} \right)}{\left( 1-\rho _{r,W} \right) ^2}-\left( 2\alpha -72\alpha ^2L \right) \right) }}\mathbb{E}\left[ f\left( \bar{x}_k \right) -f\left( x^* \right) \right] 
\\
&+\underset{e_2}{\underbrace{\left( \left( 1-r \right) \frac{\left( 4\alpha ^2L^2+\alpha L \right)}{n}+\left( 1-r \right) \frac{32\alpha ^2L^2}{n}-\frac{5L\alpha}{2n} \right) }}\mathbb{E}\left[ \left\| \hat{X}_k-\mathbf{1}_n\bar{x}_k \right\| ^2 \right] 
\\
&+\underset{e_3}{\underbrace{\left( \frac{2\alpha ^2}{n}\frac{1-p}{M}+\frac{20\alpha ^3L\rho _{r,W}}{\left( 1-\rho _{r,W} \right)}\frac{1-p}{M}-\frac{5\alpha ^2}{2M} \right) }}\mathbb{E}\left[ \left\| \nabla F\left( X_{t_{k-1}} \right) -\nabla F\left( \mathbf{1}_Mx^* \right) \right\| _{2}^{2} \right] 
\\
&+\underset{e_4}{\underbrace{\left( \frac{2\alpha ^2}{n}\frac{p}{M}+\frac{20\alpha ^3L\rho _{r,W}}{\left( 1-\rho _{r,W} \right)}\frac{p}{M}+\frac{5\alpha ^2}{M}-\frac{8\alpha ^2}{M} \right) }}\mathbb{E}\left[ \left\| \left[ \nabla F\left( X_{k-1} \right) -\nabla F\left( \mathbf{1}_Mx^* \right) \right] \right\| _{2}^{2} \right] 
\\
&+\left( 1-\alpha \mu \right) \mathbb{E}\left[ \left\| \bar{x}_k-x^* \right\| _{2}^{2} \right] +\left( 1-\frac{1-\rho _{r,W}}{8} \right) c_1\mathbb{E}\left[ \left\| \hat{X}_k-\mathbf{1}_n\bar{x}_k \right\| ^2 \right] 
\\
&+\left( 1-\frac{p}{2} \right) c_2\mathbb{E}\left[ \left\| \nabla F\left( X_{t_{k-1}} \right) -\nabla F\left( \mathbf{1}_Mx^* \right) \right\| _{2}^{2} \right] 
\\
&+\left( 1-\frac{q}{2} \right) c_3\mathbb{E}\left[ \left\| \left[ \nabla F\left( X_k \right) -\nabla F\left( \mathbf{1}_Mx^* \right) \right] \right\| _{2}^{2}\,\, \right] +\frac{80\alpha ^3L\rho _{r,W}}{\left( 1-\rho _{r,W} \right) ^2}\zeta ^*.
\end{aligned}
\end{equation*}
It can be verified that if the step-size satisfies:
\begin{equation}\label{stepsize_2}
\begin{aligned}
\alpha \leqslant \min \left\{ \frac{1}{64L},\frac{1-\rho _{r,W}}{40L},\frac{1-\rho _{r,W}}{16L\sqrt{\rho _{r,W}\left( 1+\rho _{r,W} \right)}} \right\} =\mathcal{O}\left( \frac{1-\rho _{r,W}}{L} \right),
\end{aligned}
\end{equation}
then we have $e_1, e_2, e_3, e_4 \leqslant 0$, and
\begin{equation*}
\begin{aligned}
\mathbb{E}\left[ T_{k+1} \right] \leqslant \left( 1-\min \left\{ \alpha \mu , \frac{pq}{2}, \frac{1-\rho _{r,W}}{8} \right\} \right) \mathbb{E}\left[ T_k \right] +\frac{80\alpha ^3L\rho _{r,W}}{\left( 1-\rho _{r,W} \right) ^2}\zeta ^*,
\end{aligned}
\end{equation*}
which completes the proof.
\end{proof}

\subsection{{Proof of Corollary \ref{Cor_Complexcity_2}}}
\begin{proof}
Similar to the proof of Corollary \ref{Cor_Complexcity_1}, let $\gamma$ denote the resulting upper bound on step size $\alpha$ in (\ref{stepsize_2}), and give another extra upper bound as follow:
\begin{equation}
\begin{aligned}
\alpha \leqslant \min \left\{ \gamma , \frac{\ln \left( \max \left\{ 2, \frac{\mu ^3K^2}{D_1} \right\} \right)}{\mu K} \right\},
\end{aligned}
\end{equation}
where $D_1=\frac{80L\rho _{r,W}}{\left( 1-\rho _{r,W} \right) ^2}\zeta ^*$. Then, for a constant iteration $K$ such that either
$
\gamma \leqslant \frac{\ln \left( \max \left\{ 2, \frac{\mu ^3K^2}{D_1} \right\} \right)}{\mu K}
$
or
$$
\frac{\ln \left( \max \left\{ 2, \frac{\mu ^3K^2}{D_1} \right\} \right)}{\mu K}\leqslant \min \left\{ \frac{1-\rho_{r,W} }{8\mu}, \frac{pq}{2\mu} \right\},
$$
Then, substituting the value of $\gamma$ and hiding the logarithmic factors and constants, we have $\mathbb{E}\left[ T_K \right] \leqslant \varepsilon $ after at most the following number of iterations
$$
K\geqslant \mathcal{O}\left( \left(\frac{1}{pq}+\frac{L}{\mu \left( 1-\rho _{r,W} \right)} \right) \log  \frac{\mathbb{E}\left[ T_0 \right]}{\varepsilon} \right) +\mathcal{O}\left( \sqrt{\frac{\rho _{r,W}L\zeta ^*}{\left( 1-\rho _{r,W} \right) ^2\mu ^3\varepsilon}} \right) ,
$$
which completes the proof.
\end{proof}

\subsection{{Proof of Theorem \ref{Thm_With_GT_VR}}}
\begin{proof}
For algorithms  $\mathcal{A}\left( \cdot ,\cdot ,C_k = W_k\otimes \mathbf{J}_m \right)$, recalling the following parameter settings of Lyapunov function (\ref{Lyapunov_func})
$$
c_0=1, \,\,
c_1=\frac{1-\rho _{r,W}}{n\rho _{r,W}\left( 1+\rho _{r,W} \right)},\,\,
c_2=0, \,\,
c_3=\frac{20\alpha ^2}{Mq\left( 1-\rho _{r,W} \right) ^2},\,\, c_4=\frac{8\alpha ^2}{n\left( 1-\rho _{r,W} \right)}
$$
and using Lemma \ref{Lem_sigma_k}, \ref{Lem_opt_gap}, \ref{Lem_Cons_error_GT}, \ref{Lem_VR_error}, \ref{Lem_GT_error}, we have
\begin{equation}
\begin{aligned}
&\mathbb{E}\left[ T_{k+1} \right] \leqslant \left( 1-\min \left\{ \alpha \mu , \frac{q}{2}, \frac{1-\rho _{r,W}}{8} \right\} \right) \mathbb{E}\left[ T_k \right] 
\\
&+e_1\mathbb{E}\left[ f\left( \bar{x}_k \right) -f\left( x^* \right) \right] +e_2\mathbb{E}\left[ \left\| \hat{X}_k-\mathbf{1}_n\bar{x}_k \right\| ^2 \right] 
\\
&+e_3\mathbb{E}\left[ \left\| \nabla F\left( X_{k-1} \right) -\nabla F\left( \mathbf{1}_Mx^* \right) \right\| ^2 \right] +e_4\mathbb{E}\left[ \left\| \hat{Y}_k-\mathbf{1}_n\bar{y}_k \right\| ^2 \right] ,
\end{aligned}
\end{equation}
where
\begin{equation*}
\begin{aligned}
e_1&=\frac{32\alpha ^2L\left( 1+\rho _{r,W} \right)}{\left( 1-\rho _{r,W} \right) ^2}\left( 8\alpha ^2L^2+4+2q \right) +\frac{80\alpha ^2L}{\left( 1-\rho _{r,W} \right) ^2}-\left( 2\alpha -8\alpha ^2L \right),
\\
e_2&=\frac{32\alpha ^2L^2\left( 1+\rho _{r,W} \right)}{n\left( 1-\rho _{r,W} \right) ^2}\left( 4+\left( 1-r \right) q+4\left( 1-r \right) \alpha ^2L^2 \right) +\left( 1-r \right) \frac{80\alpha ^2L^2}{n\left( 1-\rho _{r,W} \right) ^2}
\\
&+\left( 1-r \right) \frac{\left( 4\alpha ^2L^2+\alpha L \right)}{n}-\frac{\left( 1-\rho _{r,W} \right) ^2}{4n\rho _{r,W}\left( 1+\rho _{r,W} \right)},
\\
e_3&=\frac{16\alpha ^2\left( 1+\rho _{r,W} \right)}{M\left( 1-\rho _{r,W} \right) ^2}\left( 1-q+\frac{4\alpha ^2L^2}{n} \right) +\frac{2\alpha ^2}{M}-\frac{20\alpha ^2}{2M\left( 1-\rho _{r,W} \right) ^2},
\\
e_4&=\frac{\alpha ^2}{n}-\frac{8\alpha ^2}{8n}=0.
\end{aligned}
\end{equation*}
It can be verified that if the step-size satisfies:
\begin{equation}\label{stepsize_3}
\begin{aligned}
\alpha \leqslant \min \left\{ \frac{1}{8L}, \frac{1-\rho _{r,W}}{4L\sqrt{2\rho _{r,W}\left( 1+\rho _{r,W} \right)}},  \frac{\left( 1-\rho _{r,W} \right) ^2}{528L} \right\} =\mathcal{O}\left( \frac{\left( 1-\rho _{r,W} \right) ^2}{L} \right) ,
\end{aligned}
\end{equation}
then all the coefficients satisfy $e_1, e_2, e_3, e_4 \leqslant 0$, which further leads to
\begin{equation}
\begin{aligned}
\mathbb{E}\left[ T_{k+1} \right] \leqslant \left( 1-\min \left\{ \alpha \mu ,\,\,\frac{q}{2},\,\,\frac{1-\rho _{r,W}}{8} \right\} \right) \mathbb{E}\left[ T_k \right].
\end{aligned}
\end{equation}
We thus complete the proof.

\end{proof}
\subsection{{Proof of Corollary \ref{Cor_Complexcity_3}}}

\begin{proof}
By Theorem \ref{Thm_With_GT_VR}, we have
$$
\mathbb{E}\left[ T_k \right] \leqslant \max \left\{ 1-\alpha \mu , \frac{q}{2},\frac{1-\rho _{r,W}}{8}\right\} ^k\mathbb{E}\left[ T_0 \right],
$$
where
$$
\alpha =\mathcal{O}\left( \frac{\left( 1-\rho _{r,W} \right) ^2}{L} \right),
$$
then similar to the proof of Corollary \ref{Cor_Complexcity_1}, we have $\mathbb{E}\left[ T_K \right] \leqslant \varepsilon$ after at most the following number of iterations:
$$
K\geqslant {\mathcal{O}}\left( \left( \frac{L}{\mu \left( 1-\rho _{r,W} \right) ^2}+\frac{1}{q} \right) \log \frac{\mathbb{E}\left[ T_0 \right]}{\varepsilon} \right).
$$
\end{proof}

\section{{Sub-linear Convergence Analysis for Convex Problems}}\label{App_sublinear_results}
In this section, we give the corresponding sub-linear convergence rate in section \ref{Sec_Con_analysis} for smooth and convex ($\mu=0$) objectives, the proofs are also based on the Lyapunov function (\ref{Lyapunov_func}).

\begin{Thm}\label{Thm_4}
Consider algorithms belonging to $\mathcal{A}\left( \cdot , \cdot , C_k\equiv \mathbf{I}_M \right)$. Suppose Assumption \ref{Ass_smoothness}-\ref{Ass_algoirthm} hold and $\mu =0$. Let
\begin{equation*}
\begin{aligned}
c_0=1, \,\,c_1=\left( 1-r \right) \frac{8L\left( 4\alpha L+1 \right)}{n\left( 1-\rho _{r,W} \right)},\,\, c_2=c_3=c_4=0
\end{aligned}
\end{equation*}
and the step-size satisfy 
\begin{equation}\label{Eq_stepsize_4}
\begin{aligned}
\alpha \leqslant \min \left\{ \frac{1}{5L},\frac{\left( 1-\rho _{r,W} \right)}{24L\sqrt{\rho _{r,W}\left( 1+\rho _{r,W} \right)}} \right\} =\min \left\{ \mathcal{O}\left( \frac{1}{L} \right) , \mathcal{O}\left( \frac{1-\rho _{r,W}}{L\sqrt{\rho _{r,W}}} \right) \right\},
\end{aligned}
\end{equation}
then we have for all $k>0$
\begin{equation}
\begin{aligned}
\frac{1}{K}\sum_{k=0}^{K-1}{\mathbb{E}\left[ f\left( \bar{x}_k \right) -f\left( x^* \right) \right]}\leqslant \frac{5\left\| \bar{x}_0-x^* \right\| ^2}{\alpha K}+\frac{10\alpha \sigma ^*}{nb}+\left( 1-r \right) \frac{72\alpha ^2L\rho _{r,W}}{n\left( 1-\rho _{r,W} \right)}\left( \frac{4n\zeta ^*}{1-\rho _{r,W}}+\frac{n\sigma ^*}{b} \right) .
\end{aligned}
\end{equation}
\end{Thm}

\begin{proof}
According to the proof of Theorem \ref{Thm_Without_GT_VR} and noticing $\mu=0$, we further let the step size satisfies
$$
\alpha \leqslant \frac{\left( 1-\rho _{r,W} \right)}{24L\sqrt{\rho _{r,W}\left( 1+\rho _{r,W} \right)}},
$$
which implies that
$$
\left( 2\alpha -8\alpha ^2L \right) -\left( 1-r \right) \frac{64\alpha ^3L^2\left( 4\alpha L+1 \right) \rho _{r,W}\left( 1+\rho _{r,W} \right)}{\left( 1-\rho _{r,W} \right) ^2}\geqslant \left( \alpha -4\alpha ^2L \right).
$$
Combining the upper bound of step-size in (\ref{stepsize_1}),
then we have
\begin{equation}
\begin{aligned}
&\left( \alpha -4\alpha ^2L \right) \left( f\left( \hat{x}_k \right) -f\left( x^* \right) \right) 
\\
&\leqslant \mathbb{E}\left[ T_k \right] -\mathbb{E}\left[ T_{k+1} \right] +\frac{2\alpha ^2\sigma ^*}{nb}+\left( 1-r \right) \frac{8\alpha L\left( 4\alpha L+1 \right)}{n\left( 1-\rho _{r,W} \right)}\alpha ^2\rho _{r,W}\left( \frac{4n\zeta ^*}{1-\rho _{r,W}}+\frac{n\sigma ^*}{b} \right). 
\end{aligned}
\end{equation}
Then, summing over $k$ from $0$ to $K-1$, we further obtain
\begin{equation*}
\begin{aligned}
&\frac{1}{K}\sum_{k=1}^{K-1}{\left( f\left( \bar{x}_k \right) -f\left( x^* \right) \right)}
\\
&\leqslant \frac{\mathbb{E}\left[ T_0 \right]}{\alpha \left( 1-4\alpha L \right) K}+\frac{2\alpha \sigma ^*}{n b\left( 1-4\alpha L \right)}+\left( 1-r \right) \frac{8\alpha ^2L\left( 4\alpha L+1 \right) \rho _{r,W}}{n\left( 1-\rho _{r,W} \right) \left( 1-4\alpha L \right)}\left( \frac{4n\zeta ^*}{1-\rho _{r,W}}+\frac{n\sigma ^*}{b} \right) 
\\
&\leqslant \frac{5\left\| \bar{x}_0-x^* \right\| ^2}{\alpha K}+\frac{10\alpha \sigma ^*}{n b}+\left( 1-r \right) \frac{72\alpha ^2L\rho _{r,W}}{n\left( 1-\rho _{r,W} \right)}\left( \frac{4n\zeta ^*}{1-\rho _{r,W}}+\frac{n\sigma ^*}{b} \right),
\end{aligned}
\end{equation*}
which completes the proof.
\end{proof}

\begin{Cor}\label{Cor_4}
Under the same conditions in Theorem \ref{Thm_4}, suppose the step-size satisfy
$$
\alpha \leqslant \left\{ \gamma , \frac{1}{\sqrt{H_1K}}, \frac{1}{\sqrt[3]{H_2K}} \right\},
$$
where $\gamma$ denotes the resulting upper bound on $\alpha$ in (\ref{Eq_stepsize_4}) and
$$
H_1=\frac{\sigma ^*}{n b \left\| x_0-x^* \right\| ^2}, \,\, H_2=\frac{72L\rho _{r,W}\left( 1-r \right)}{\left( 1-\rho _{r,W} \right) \left\| x_0-x^* \right\| ^2}\left( \frac{4\zeta ^*}{1-\rho _{r,W}}+\frac{\sigma ^*}{b} \right) .
$$
Then, we have
$$
\frac{1}{K}\sum_{k=0}^{K-1}{\mathbb{E}\left[ f\left( \bar{x}_k \right) -f\left( x^* \right) \right]}\leqslant \varepsilon
$$
after at most the following iterations:
\begin{equation}\label{Comp_4}
\begin{aligned}
K\geqslant {\mathcal{O}}\left( \frac{1}{\gamma \varepsilon}+\frac{\sigma ^*}{nb\varepsilon ^2}+\frac{1}{\varepsilon ^{3/2}}\sqrt{\frac{\rho _{r,W}L\left( b\zeta ^*+\left( 1-\rho _{r,W} \right) \sigma ^* \right)}{b\left( 1-\rho _{r,W} \right) ^2}} \right) \left\| \bar{x}_0-x^* \right\| ^2.
\end{aligned}
\end{equation}
\end{Cor}
\begin{proof}
According to the obtained sub-linear rate in Theorem \ref{Thm_4} and the upper bound on $\alpha$, we have three cases: 

$$
\gamma \leqslant \left\{ \frac{1}{\sqrt{H_1K}}, \frac{1}{\sqrt[3]{H_2K}} \right\} ,\quad
\frac{1}{\sqrt{H_1K}}\leqslant \left\{ \gamma , \frac{1}{\sqrt[3]{H_2K}} \right\} ,\quad
\frac{1}{\sqrt[3]{H_2K}}\leqslant \left\{ \gamma , \frac{1}{\sqrt{H_1K}} \right\}.
$$
Then, we can obtain
\begin{equation*}
\begin{aligned}
&\frac{1}{K}\sum_{k=0}^{K-1}{\begin{array}{c}
	\mathbb{E}\left[ f\left( \bar{x}_k \right) -f\left( x^* \right) \right]\\
\end{array}}
\\
&\leqslant \frac{5\left\| \bar{x}_0-x^* \right\| ^2}{\min \left\{ \gamma , \frac{1}{\sqrt{H_1K}}, \frac{1}{\sqrt[3]{H_2K}} \right\} K}+\frac{10\sigma ^*}{nb}\frac{1}{\sqrt{H_1K}}
\\
&+\left( 1-r \right) \frac{72L\rho _{r,W}}{n\left( 1-\rho _{r,W} \right)}\left( \frac{4n\zeta ^*}{1-\rho _{r,W}}+\frac{n\sigma ^*}{b} \right) \frac{1}{\sqrt[3]{H_2K}} \leqslant \varepsilon .
\end{aligned}
\end{equation*}
Plugging $H_1$ and $H_2$ into the above inequality and merging the similar terms, we obtain the complexity in (\ref{Comp_4}). 
\end{proof}

\begin{Thm}\label{Thm_5}
Consider algorithms $\mathcal{A}\left( \cdot ,\cdot ,C_k= \mathbf{I}_n\otimes V_k \right)$ with $p>0$. Suppose Assumption \ref{Ass_smoothness}-\ref{Ass_algoirthm} hold and $\mu =0$. Let
\begin{equation*}
\begin{aligned}
c_0=1, \,\,c_1=\frac{20L\alpha}{n\left( 1-\rho _{r,W} \right)},\,\, c_2=\frac{5\alpha ^2}{Mp},\,\,  c_3=\frac{16\alpha ^2}{Mq},\,\, c_4=0
\end{aligned}
\end{equation*}
and the step-size satisfy
\begin{equation}\label{eq:upperbound_stepsize}
\begin{aligned}
\alpha \leqslant \min \left\{ \frac{1}{64L}, \frac{1-\rho _{r,W}}{40L}, \frac{1-\rho _{r,W}}{32L\sqrt{\rho _{r,W}\left( 1+\rho _{r,W} \right)}} \right\} =\mathcal{O}\left( \frac{1-\rho _{r,W}}{L} \right).
\end{aligned}
\end{equation}
Then, we have
\begin{equation}
\begin{aligned}
\frac{1}{K}\sum_{k=1}^{K-1}{\left( f\left( \bar{x}_k \right) -f\left( x^* \right) \right)}\leqslant \frac{\mathbb{E}\left[ T_0 \right]}{\alpha K}+\frac{80\alpha ^2L\rho _{r,W}}{\left( 1-\rho _{r,W} \right) ^2}\zeta ^*.
\end{aligned}
\end{equation}
\end{Thm}

\begin{proof}
According to the proof of Theorem \ref{Thm_With_VR_only} and noticing $\mu=0$, let the step size satisfy
$$
\alpha \leqslant \frac{1-\rho _{r,W}}{32L\sqrt{\rho _{r,W}\left( 1+\rho _{r,W} \right)}},
$$
which implies that
$$
\frac{160\alpha ^3L^2\rho _{r,W}\left( 1+\rho _{r,W} \right)}{n\left( 1-\rho _{r,W} \right) ^2}\geqslant \left( \alpha -36\alpha ^2L \right).
$$
Combining the upper bound of step-size in (\ref{stepsize_2}),
then we get
\begin{equation}
\begin{aligned}
\left( \alpha -36\alpha ^2L \right) \left( f\left( \bar{x}_k \right) -f\left( x^* \right) \right) \leqslant \mathbb{E}\left[ T_k \right] -\mathbb{E}\left[ T_{k+1} \right] +\frac{80\alpha ^3L\rho _{r,W}}{\left( 1-\rho _{r,W} \right) ^2}\zeta ^*.
\end{aligned}
\end{equation}
Summing over $k$ from $0$ to $K-1$, we have
\begin{equation}
\begin{aligned}
\frac{1}{K}\sum_{k=1}^{K-1}{\left( f\left( \bar{x}_k \right) -f\left( x^* \right) \right)} &\leqslant \frac{\mathbb{E}\left[ T_0 \right]}{\alpha \left( 1-36\alpha L \right) K}+\frac{80\alpha ^2L\rho _{r,W}}{\left( 1-36\alpha L \right) \left( 1-\rho _{r,W} \right) ^2}\zeta ^*\leqslant \frac{\mathbb{E}\left[ T_0 \right]}{\alpha K}+\frac{80\alpha ^2L\rho _{r,W}}{\left( 1-\rho _{r,W} \right) ^2}\zeta ^*
\end{aligned}
\end{equation}
where
\begin{equation*}
\begin{aligned}
T_0&=\left\| x_0-x^* \right\| ^2+\mathcal{O}\left( \frac{\alpha ^2}{pq} \right) \left\| \nabla f\left( x_0 \right) -\nabla f\left( x^* \right) \right\| ^2
\\
&\leqslant \mathbb{E}\left[ \left\| x_0-x^* \right\| ^2 \right] +\mathcal{O}\left( \frac{\alpha ^2}{pq} \right) L^2\left\| x_0-x^* \right\| ^2=\mathcal{O}\left( \frac{1}{pq} \right) \left\| x_0-x^* \right\| ^2.
\end{aligned}
\end{equation*}
\end{proof}

\begin{Cor}\label{Cor_5}
Under the same conditions in Theorem \ref{Thm_5}, suppose that the step-size
$$
\alpha \leqslant \min\left\{ \gamma,\frac{1}{\sqrt{H_1K}} \right\},
$$
where $\gamma$ is the resulting upper bound of $\alpha$ in~(\ref{eq:upperbound_stepsize}) and
$$
H_1=\frac{80L\rho _{r,W}pq}{\left( 1-\rho _{r,W} \right) ^2\left\| x_0-x^* \right\| ^2}\zeta ^*.
$$
Then, we have
$$
\frac{1}{K}\sum_{k=0}^{K-1}{\mathbb{E}\left[ f\left( \bar{x}_k \right) -f\left( x^* \right) \right]}\leqslant \varepsilon 
$$
after at most the following iterations:
\begin{equation}\label{Comp_5}
\begin{aligned}
K\geqslant {\mathcal{O}}\left( \frac{1}{\gamma pq\varepsilon}+\frac{1}{pq\varepsilon ^{3/2}}\sqrt{\frac{\rho _{r,W}L\zeta ^*}{\left( 1-\rho _{r,W} \right) ^2}} \right) \left\| x_0-x^* \right\| ^2.
\end{aligned}
\end{equation}
\end{Cor}

Proof of Corollary \ref{Cor_5} is similar to the proof of Corollary \ref{Cor_4} and thus omitted.

\begin{Thm}\label{Thm_6}
Consider the algorithms $\mathcal{A}\left( \cdot ,\cdot ,C_k = W_k\otimes \mathbf{J}_m \right) $. Suppose Assumption \ref{Ass_smoothness}-\ref{Ass_algoirthm} hold and $\mu =0$. Let
\begin{equation*}
\begin{aligned}
c_0=1, \,\,
c_1=\frac{1-\rho _{r,W}}{n\rho _{r,W}\left( 1+\rho _{r,W} \right)},\,\,
c_2=0, \,\,
c_3=\frac{20\alpha ^2}{Mq\left( 1-\rho _{r,W} \right) ^2},\,\, c_4=\frac{8\alpha ^2}{n\left( 1-\rho _{r,W} \right)}
\end{aligned}
\end{equation*}
and the step-size satisfy 
\begin{equation}
\begin{aligned}
\alpha \leqslant \min \left\{ \frac{1}{8L}, \frac{1-\rho _{r,W}}{4L\sqrt{2\rho _{r,W}\left( 1+\rho _{r,W} \right)}},  \frac{\left( 1-\rho _{r,W} \right) ^2}{1056L} \right\} =\mathcal{O}\left( \frac{\left( 1-\rho _{r,W} \right) ^2}{L} \right) .
\end{aligned}
\end{equation}
Then, we have
\begin{equation}
\begin{aligned}
\frac{1}{K}\sum_{k=0}^{K-1}{\mathbb{E}\left[ f\left( \bar{x}_k \right) -f\left( x^* \right) \right]}\leqslant \frac{\left( 1-\rho _{W} \right) ^2}{2\alpha \left( \left( 1-\rho _{W} \right) ^2-\left( 640+10\left( 1-\rho _{W} \right) ^2 \right) \alpha L \right) K}\mathbb{E}\left[ T_0 \right].
\end{aligned}
\end{equation}
\end{Thm}

\begin{proof}
According to the proof of Theorem \ref{Thm_With_VR_only} and noticing $\mu=0$, let the step size satisfy
$$
\alpha \leqslant \frac{\left( 1-4\alpha L \right) \left( 1-\rho _{r,W} \right) ^2}{1056L},
$$
which implies that
$$
\frac{32\alpha ^2L\left( 1+\rho _{r,W} \right)}{\left( 1-\rho _{r,W} \right) ^2}\left( 8\alpha ^2L^2+4+2q \right) +\frac{80\alpha ^2L}{\left( 1-\rho _{r,W} \right) ^2}\leqslant \left( \alpha -4\alpha ^2L \right).
$$
Combining the upper bound of step-size in (\ref{stepsize_3}), then we get
$$
\left( \alpha -4\alpha ^2L \right) \left( f\left( \bar{x}_k \right) -f\left( x^* \right) \right) \leqslant \mathbb{E}\left[ T_k \right] -\mathbb{E}\left[ T_{k+1} \right].
$$
\vspace{-0.2cm}
Summing over $k$ from $0$ to $K-1$, we have
\begin{equation*}
\begin{aligned}
\frac{1}{K}\sum_{k=0}^{K-1}{\left( f\left( \bar{x}_k \right) -f\left( x^* \right) \right)}\leqslant \frac{\mathbb{E}\left[ T_0 \right]}{\left( 1-\alpha L \right) \alpha K},
\end{aligned}
\end{equation*}
where
\begin{equation}
\begin{aligned}
&\mathbb{E}\left[ T_0 \right] =\left\| x_0-x^* \right\| ^2+\mathcal{O}\left( \frac{\alpha ^2}{Mq\left( 1-\rho _{r,W} \right) ^2} \right) \left\| \nabla F\left( \mathbf{1}_Mx_0 \right) \right\| ^2
\\
&=\left\| x_0-x^* \right\| ^2+\mathcal{O}\left( \frac{\alpha ^2L^2}{q\left( 1-\rho _{r,W} \right) ^2} \right) \left\| x_0-x^* \right\| ^2
=\mathcal{O}\left( \frac{1}{q}\left\| x_0-x^* \right\| ^2 \right).
\end{aligned}
\vspace{-0.8cm}
\end{equation}

\end{proof}

\begin{Cor}\label{Cor_6}
Under the conditions same in Theorem \ref{Thm_6}, we have
$$
\frac{1}{K}\sum_{k=0}^{K-1}{\mathbb{E}\left[ f\left( \bar{x}_k \right) -f\left( x^* \right) \right]}\leqslant \varepsilon 
$$
after at most the following iterations:
\begin{equation}\label{Comp_5}
\begin{aligned}
K\geqslant \mathcal{O}\left( \frac{1}{\gamma q\varepsilon} \right) \left\| x_0-x^* \right\| ^2.
\end{aligned}
\end{equation}
\end{Cor}
Proof of Corollary \ref{Cor_6} is similar to the proof of Corollary \ref{Cor_4} and thus omitted.

\section{Additional Experiments} \label{Appe_add_experiments}

\begin{table}[h]
 \begin{center}
 \vspace{-0.5cm}
  \caption{Summary of the experimental setup.}
  \label{Exp_Setting_2}
  \vspace{0.1cm}
  \small
  \begin{tabular}{|c|c|c|c|c|c|c|c|} 
   \hline
   \rule {0pt}{10pt}
   \textbf{Dataset} & \textbf{Node} $\left( n \right) $ & \# \textbf{Train } & \# \textbf{Test} & $\textbf{Dimension}$ & \textbf{BS ($n\times b$)} & \textbf{SS ($\alpha$)} & $\lambda$
   \rule {0pt}{10pt}
   \\
   \hline
   \rule {0pt}{10pt}
   F-MNIST & \{8, 20, 50\} & 60000 & 10000 & 784 & 200 &  0.05 & 0.001\\
   \hline
   \rule {0pt}{10pt}
   CIFAR-10 & \{8, 20, 50\} & 50000 & 10000 & 3072 & 400 &  0.008 & 0.001 \\
   \hline
  \end{tabular}
 \end{center}
 \vspace{-0.25cm}
\end{table}

In this section, we further verify our theoretical findings by several extra experiments. 
The experiment setting is recalled here. 
We train a regularized logistic regression classifier on both CIFAR-10 and Fashion-MNIST (F-MNIST) datasets over a network of $n$ nodes each of which locally stores $m$ data samples, as defined in (\ref{prob:logistic}) and (\ref{cross-entropy loss}):
\begin{equation*}
\begin{aligned}
\vspace{-0.2cm}
\underset{x\in \mathbb{R}^{d}}{\min}f\left( x \right) :=\frac{1}{n}\sum_{i=1}^n{\frac{1}{m}\sum_{j=1}^m{\underset{f_{ij}}{\left( \underbrace{\ell \left( x,\xi _{ij} \right) +\frac{\lambda}{2}\left\| x \right\| ^2} \right)}}},
\end{aligned}
\vspace{-0.5cm}
\end{equation*}
with the cross-entropy loss $\ell$:
\begin{equation*}
\begin{aligned}
\vspace{-0.2cm}
\ell \left( x,\xi _{ij} \right) :=-\sum_{c=1}^{10}{\phi _{ij}^{c}}\log \left( 1+\exp \left( -x^T\theta _{i,j} \right) \right) ^{-1},
\end{aligned}
\vspace{-0.2cm}
\end{equation*}
where $\lambda$ is a regularization parameter, $\xi _{ij}$ represents the $j$-th sample of node $i$ with feature vector $\theta _{i,j} \in \mathbb{R}^d$ and label $\phi _{ij}^{c}\in \left\{ -1, 1 \right\}$ of class $c$. 
The datasets and parameters we used are summarized in Table~\ref{Exp_Setting_2} with $r=0.05$ for corresponding algorithms. 
All the algorithms are executed on a server with 8 GPUs (NVIDIA RTX 2080Ti).

\begin{table}[h]
 \begin{center}
  \caption{Unbalanced label distribution of training samples with $h=20$ for CIFAR-10.}
  \label{data_split_h20}
  \vspace{0.1cm}
  \small
  \begin{tabular}{|c|c|c|c|c|c|c|c|c|c|c|} 
   \hline
   \rule {0pt}{10pt}
   { \diagbox[]{Node}{Label}} & 1 & 2 & 3 & 4 & 5 & 6 & 7 & 8 & 9 & 10
   \rule {0pt}{10pt}
   \\
   \hline
   \rule {0pt}{10pt}
    1 & 555 & 575 & 595 & 615 & 635 & 655 & 675 & 695 & 625 & 625\\
   \hline
   \rule {0pt}{10pt}
    2 & 575 & 595 & 615 & 635 & 655 & 675 & 695 & 555 & 625 & 625\\
    \hline
   \rule {0pt}{10pt}
    3 & 595 & 615 & 635 & 655 & 675 & 695 & 555 & 575 & 625 & 625\\
    \hline
   \rule {0pt}{10pt}
    4 & 615 & 635 & 655 & 675 & 695 & 555 & 575 & 595 & 625 & 625\\
    \hline
   \rule {0pt}{10pt}
    5 & 635 & 655 & 675 & 695 & 555 & 575 & 595 & 615 & 625 & 625\\
    \hline
   \rule {0pt}{10pt}
    6 & 655 & 675 & 695 & 555 & 575 & 595 & 615 & 635 & 625 & 625\\
    \hline
   \rule {0pt}{10pt}
    7 & 675 & 695 & 555 & 575 & 595 & 615 & 635 & 655 & 625 & 625\\
    \hline
   \rule {0pt}{10pt}
    8 & 695 & 555 & 575 & 595 & 615 & 635 & 655 & 675 & 625 & 625\\
    \hline
  \end{tabular}
 \end{center}
 \vspace{-0.5cm}
\end{table}

\begin{table}[h]

 \begin{center}
  \caption{A highly unbalanced label distribution of training samples denoted by $h=h_{\max}$ for CIFAR-10.}
  \label{data_split_hmax}
  \vspace{0.1cm}
  \small 
  \begin{tabular}{|c|c|c|c|c|c|c|c|c|c|c|} 
   \hline
   \rule {0pt}{10pt}
   {\diagbox[]{Node}{Label}} & 1 & 2 & 3 & 4 & 5 & 6 & 7 & 8 & 9 & 10
   \rule {0pt}{10pt}
   \\
   \hline
   \rule {0pt}{10pt}
    1 & 1000 & 0 & 0 & 0 & 1000 & 1000 & 1000 & 1000 & 625 & 625\\
   \hline
   \rule {0pt}{10pt}
    2 & 1000 & 1000 & 0 & 0 & 0 & 1000 & 1000 & 1000 & 625 & 625\\
    \hline
   \rule {0pt}{10pt}
    3 & 1000 & 1000 & 1000 & 0 & 0 & 0 & 1000 & 1000 & 625 & 625\\
    \hline
   \rule {0pt}{10pt}
    4 & 1000 & 1000 & 1000 & 1000 & 0 & 0 & 0 & 1000 & 625 & 625\\
    \hline
   \rule {0pt}{10pt}
    5 & 1000 & 1000 & 1000 & 1000 & 1000 & 0 & 0 & 0 & 625 & 625\\
    \hline
   \rule {0pt}{10pt}
    6 & 0 & 1000 & 1000 & 1000 & 1000 & 1000 & 0 & 0 & 625 & 625\\
    \hline
   \rule {0pt}{10pt}
    7 & 0 & 0 & 1000 & 1000 & 1000 & 1000 & 1000 & 0 & 625 & 625\\
    \hline
   \rule {0pt}{10pt}
    8 & 0 & 0 & 0 & 1000 & 1000 & 1000 & 1000 & 1000 & 625 & 625\\
    \hline
  \end{tabular}
 \end{center}
 \vspace{-0.5cm}
\end{table}

\hy{\textbf{Heterogeneously split dataset.}  In order to generate different levels of data heterogeneity among devices, we split the samples of 10 classes into $n=\{8, 20, 50\}$ nodes with different \textit{label distributions} for both Fashion-MNIST and CIFAR-10 datasets. 
In particular, we denote by $m_{i,c}$ the number of samples of class $c\in[10]$ allocated to node $i\in [n]$, then we generate the unbalanced label distribution through cyclic arithmetic sequences with a difference of $h$, i.e, for the case $n=8$, we have}

\begin{equation}\label{Eq_data_splite_1}
\begin{aligned}
m_{i,c}=\begin{cases}
	m_0+h (i+c-2), \,\, \forall i, c\in \left[ n \right]\\
	\frac{M}{10n},\,\, \forall i\in \left[ n \right] , c \notin  \left[ n \right]\\
\end{cases},\quad
\mathrm{s}.\mathrm{t}.,\,\, \sum_{i=1}^n{m_{i,c}}=\frac{M}{10},\,\, \sum_{c=1}^{10}{m_{i,c}=}\frac{M}{n},
\end{aligned}
\end{equation}
where $M$ denotes the total number of samples to be allocated, $m_0$ denotes the initial term of the arithmetic sequence, and $h$ is the difference between the consecutive terms which we used to represent the level of data heterogeneity. 
We also consider the cases $n=20, 50$ which are larger than the number of classes, then the label distribution $\{m_{i,c}\}$ is generated as follow:
\begin{equation}\label{Eq_data_splite_2}
\begin{aligned}
m_{i,c}=m_0+h (\left( i+c-2 \right) \%10),  \forall i\in \left[ n \right] , c\in \left[ 10 \right], \quad
\mathrm{s}.\mathrm{t}.,\,\, \sum_{i=1}^n{m_{i,c}}=\frac{M}{10},\,\, \sum_{c=1}^{10}{m_{i,c}=}\frac{M}{n},
\end{aligned}
\end{equation}
where $\%$ denotes the remainder operator. Intuitively, the larger the value of $h$, the more heterogeneous the local datasets will be. For example, let $n=8$, $h=20$ and $m_0=555$, we get the allocation strategy of training samples on CIFAR-10 dataset by \eqref{Eq_data_splite_1} in Table \ref{data_split_h20}. Furthermore, we also design a
highly unbalanced label distribution of CIFAR-10 denoted by $h_{\max}$ in Table \ref{data_split_hmax} (can be adapted to F-MNIST), which indicates that each node only has samples of 7 classes while other classes of samples are inaccessible. As a result, it is very difficult to learn a global model for 10-class image classification task in a distributed manner  (c.f. Fig.~\ref{Expe}).

\textbf{Topology dependence.} We provide more experiments to verify the topology dependence of different algorithms, as we reported in the theoretical results in the main text. We compare various algorithms that can be recovered in the proposed SPP framework on heterogeneously split datasets ($h=20$) over graphs: i) directed ring with $n=8$; ii) directed ring with $n=20$; iii) geometric graph with $n=50$. The results are summarized in Fig.~\ref{Expe_3}. \finetune{It follows from the figure that the algorithms adopting VR and/or GT schemes (solid lines) outperform the others both in terms of testing accuracy and training loss, especially on the graph with worse connectivity (i.e., $n=50, \rho_W \approx 0.99$), which verifies the dependency of the performance of the algorithms on the sampling variance,  data heterogeneity, and the connectivity of the graph.}

\begin{figure}[t]
    \centering
    \subfigure 
    {
        \begin{minipage}[t]{0.3\textwidth}
            \centering          
            \includegraphics[width=\textwidth]{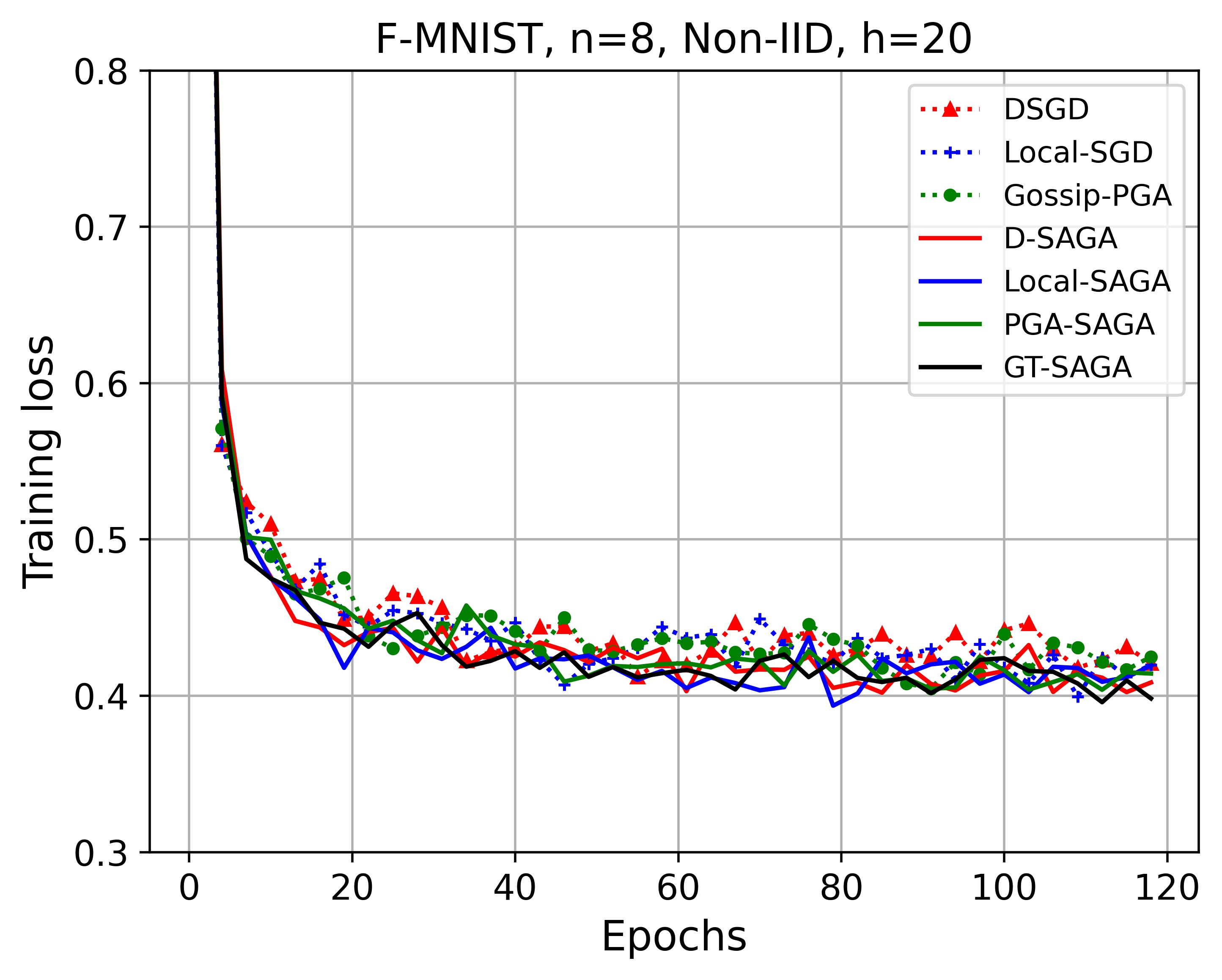}   
        \end{minipage}%
    }
        \subfigure 
    {
        \begin{minipage}[t]{0.3\textwidth}
            \centering          
            \includegraphics[width=\textwidth]{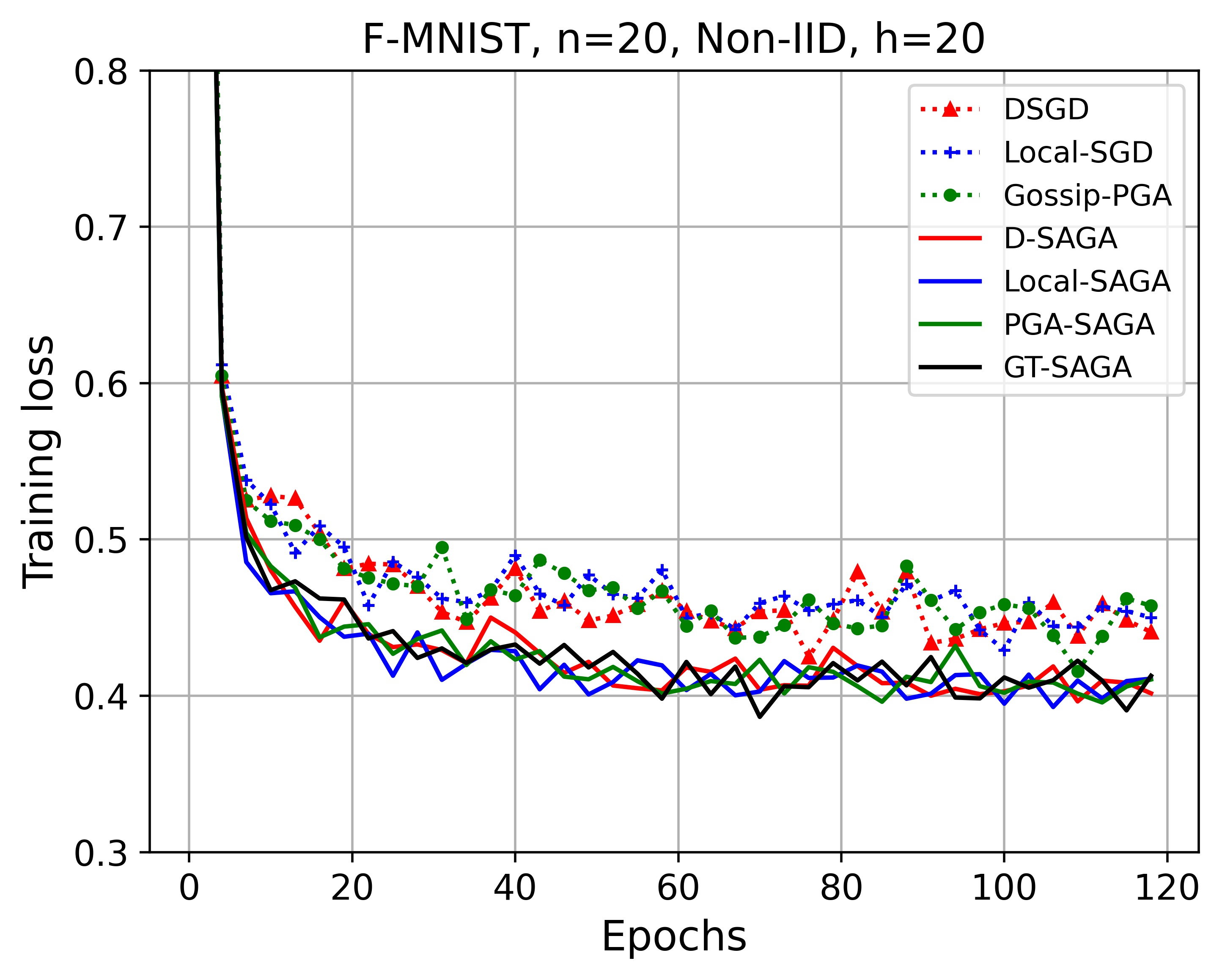}   
        \end{minipage}%
    }
        \subfigure 
    {
        \begin{minipage}[t]{0.3\textwidth}
            \centering          
            \includegraphics[width=\textwidth]{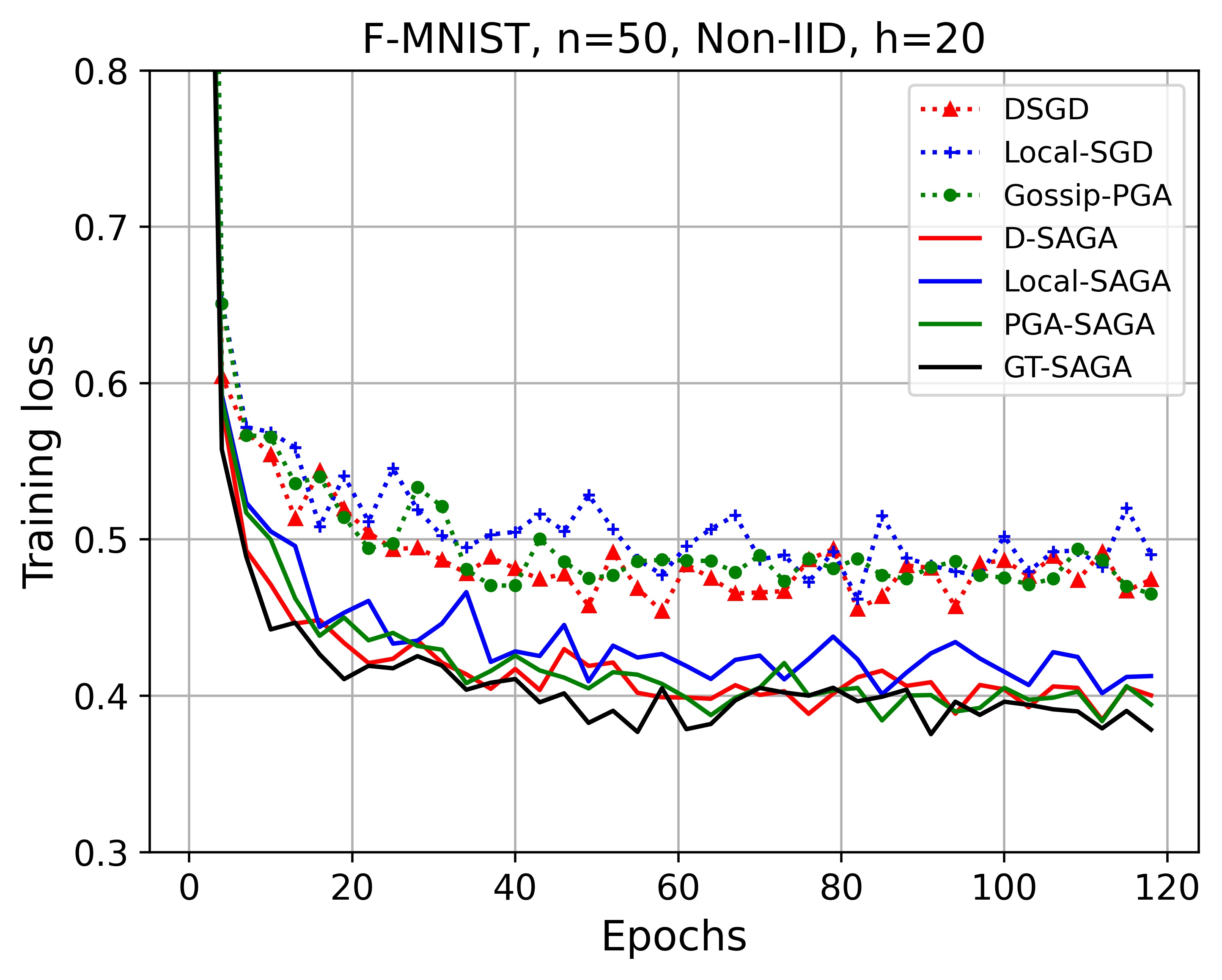}   
        \end{minipage}%
    }\\
    \subfigure 
    {
        \begin{minipage}[t]{0.3\textwidth}
            \centering          
            \includegraphics[width=\textwidth]{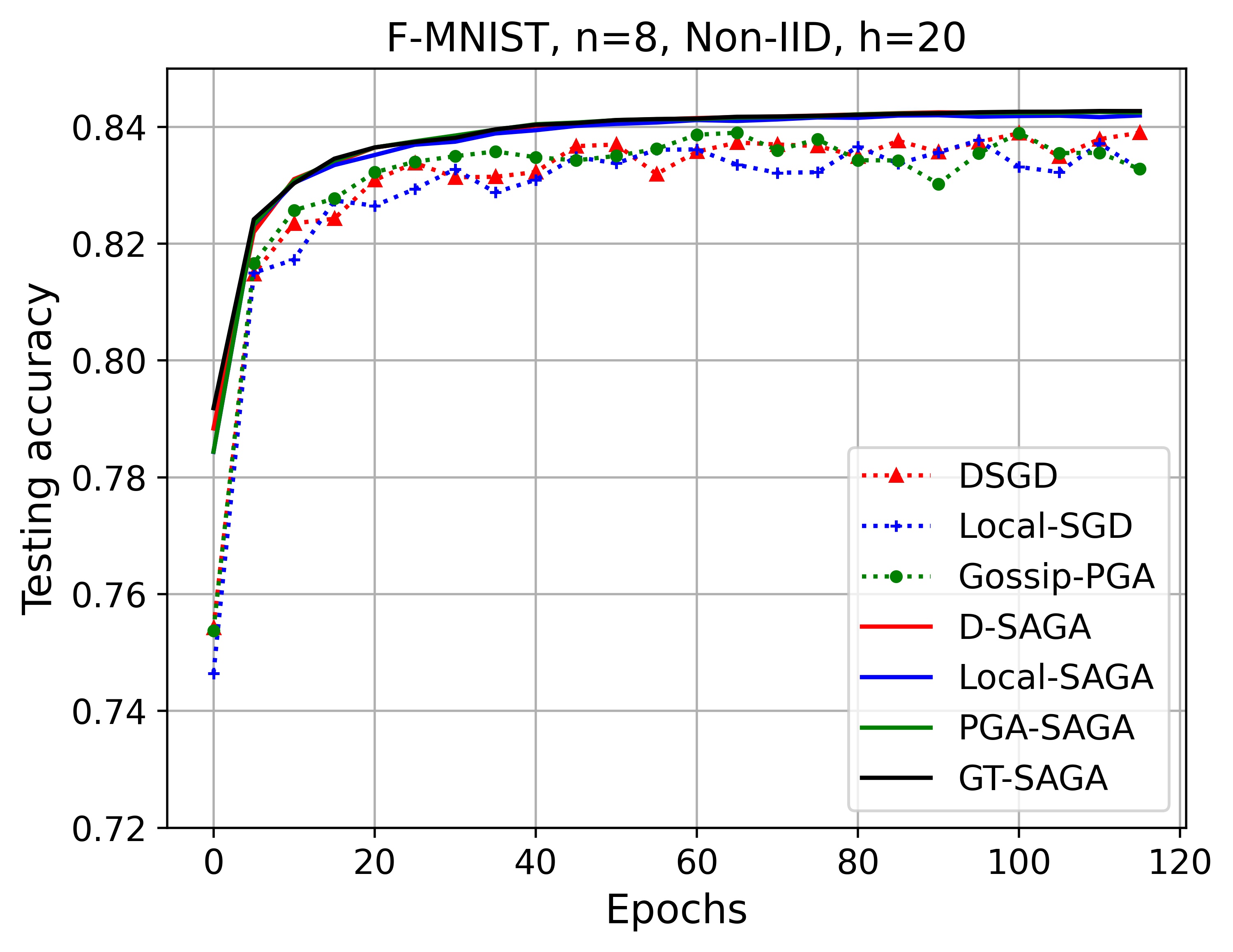}   
        \end{minipage}%
    }
    \subfigure 
    {
        \begin{minipage}[t]{0.3\textwidth}
            \centering      
            \includegraphics[width=\textwidth]{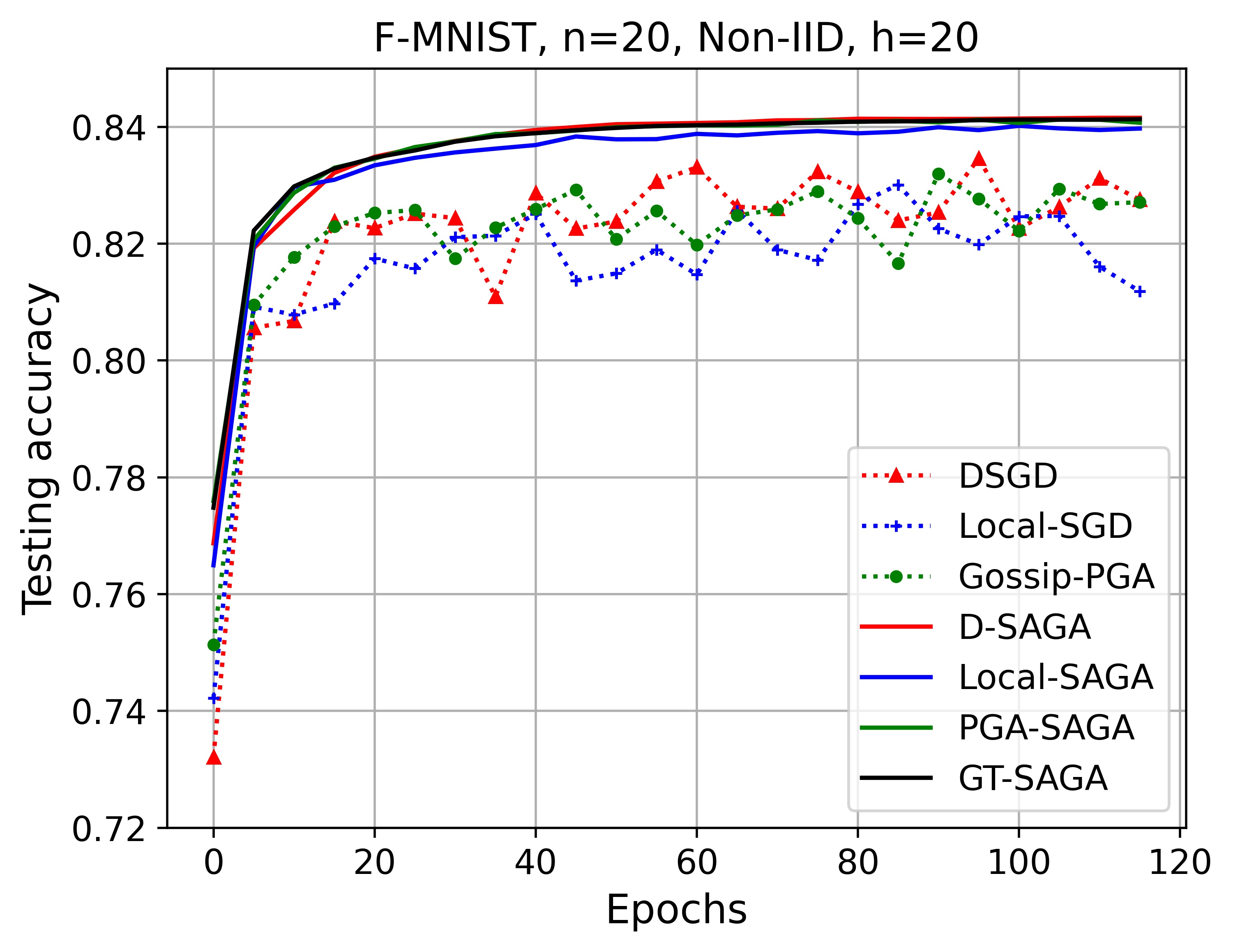}    
        \end{minipage}
    }
    \subfigure 
    {
        \begin{minipage}[t]{0.3\textwidth}
            \centering      
            \includegraphics[width=\textwidth]{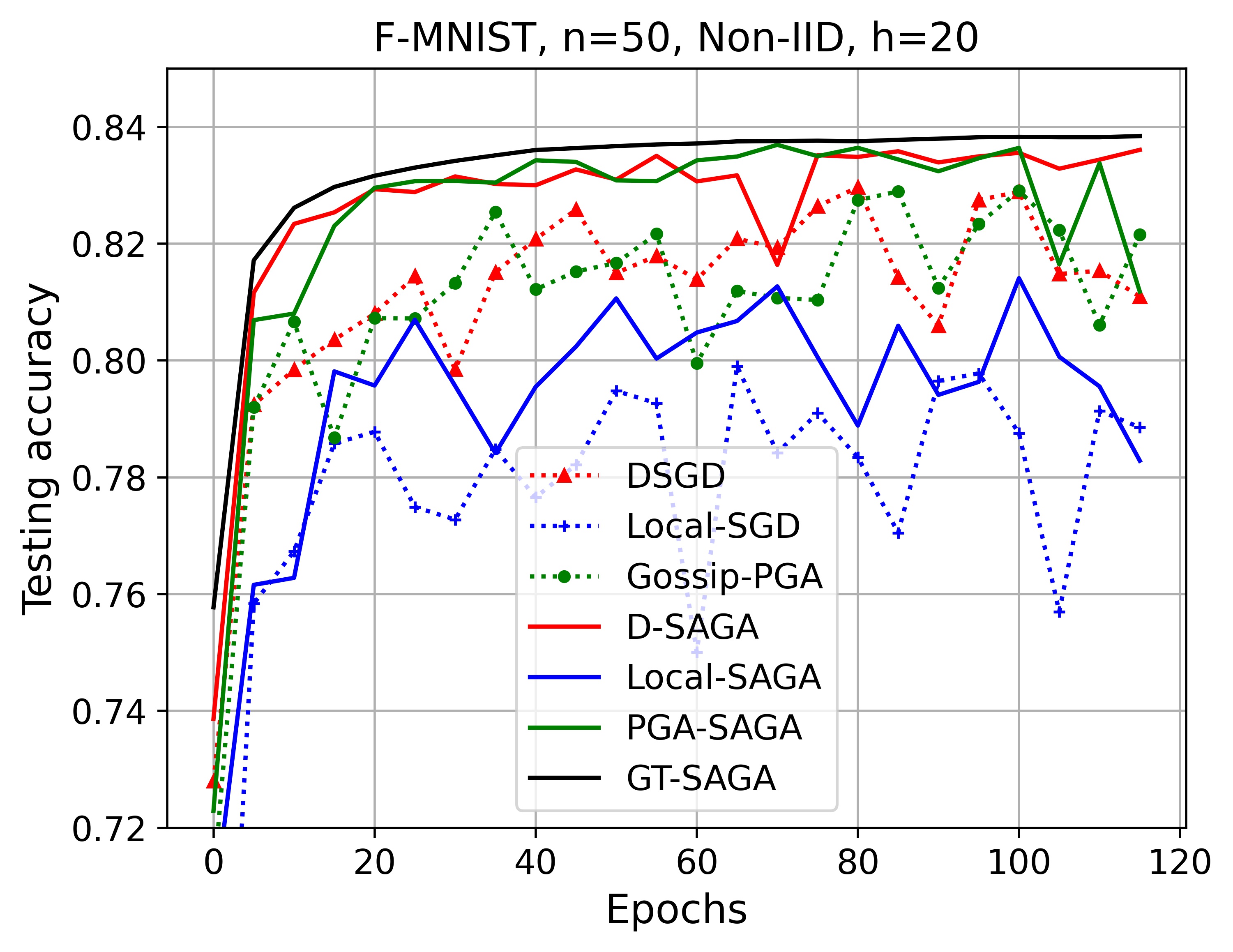}    
        \end{minipage}
    }\\
    \subfigure 
    {
        \begin{minipage}[t]{0.3\textwidth}
            \centering          
            \includegraphics[width=\textwidth]{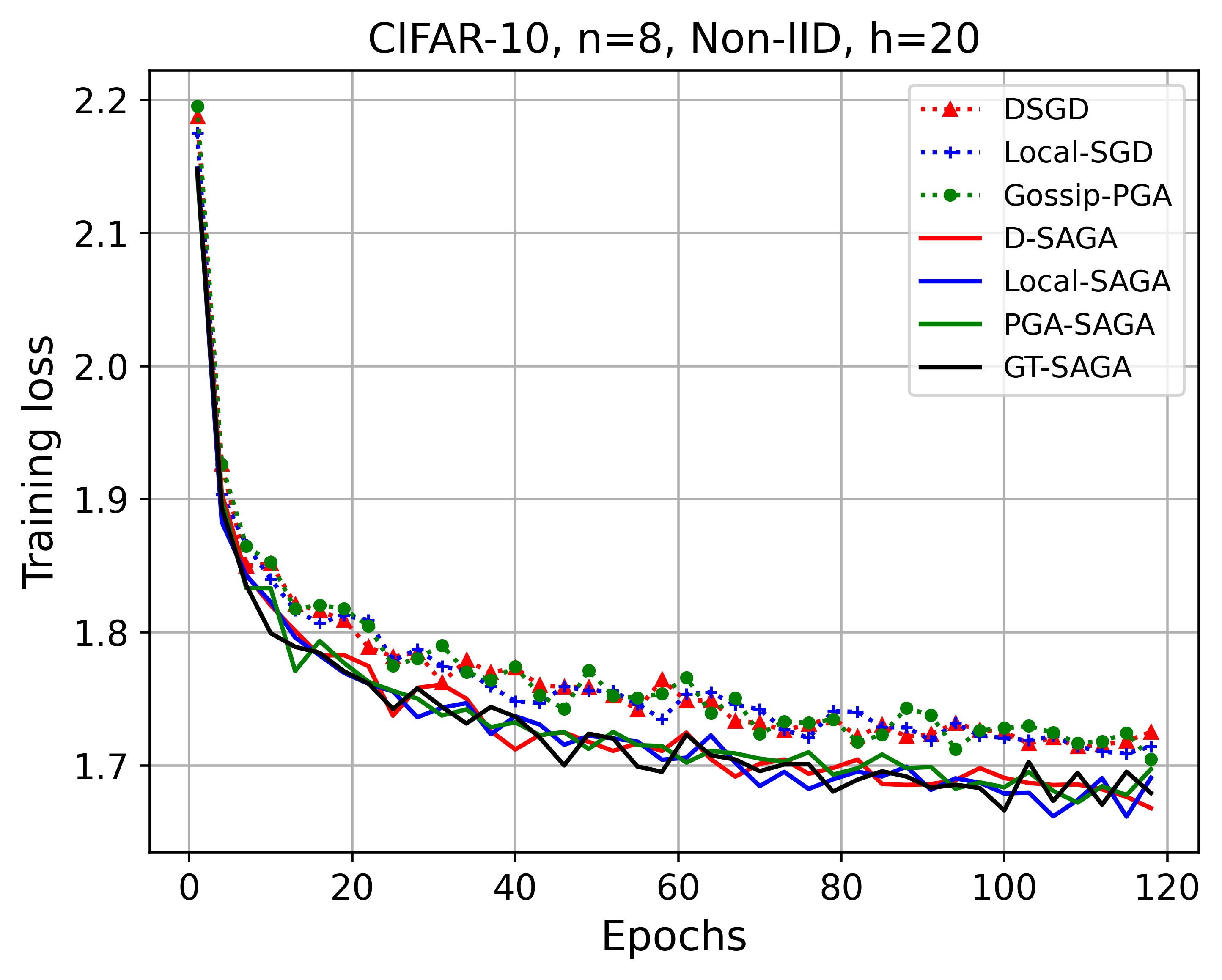}   
        \end{minipage}%
    }
    \subfigure 
    {
        \begin{minipage}[t]{0.3\textwidth}
            \centering          
            \includegraphics[width=\textwidth]{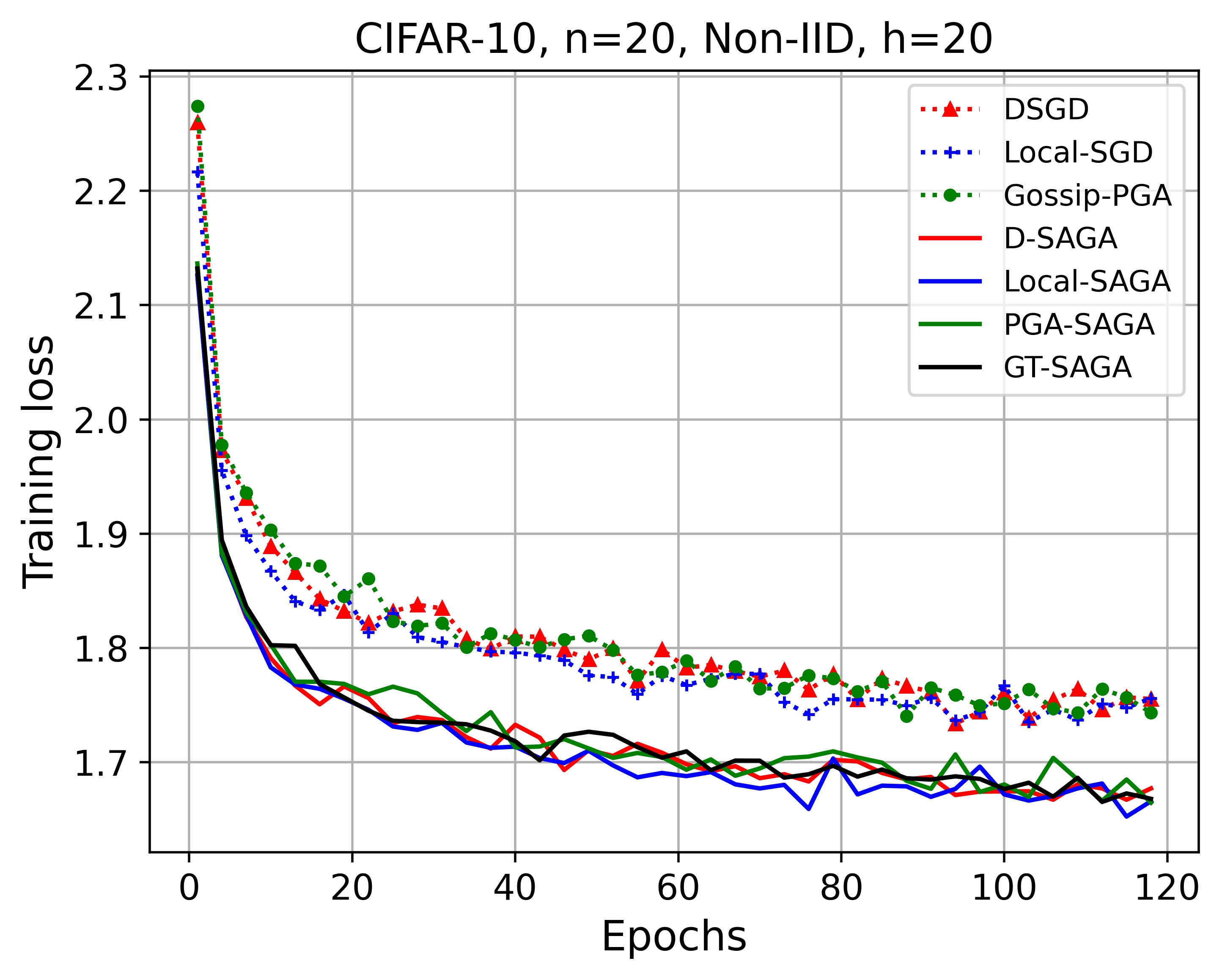}   
        \end{minipage}%
    }
    \subfigure 
    {
        \begin{minipage}[t]{0.3\textwidth}
            \centering          
            \includegraphics[width=\textwidth]{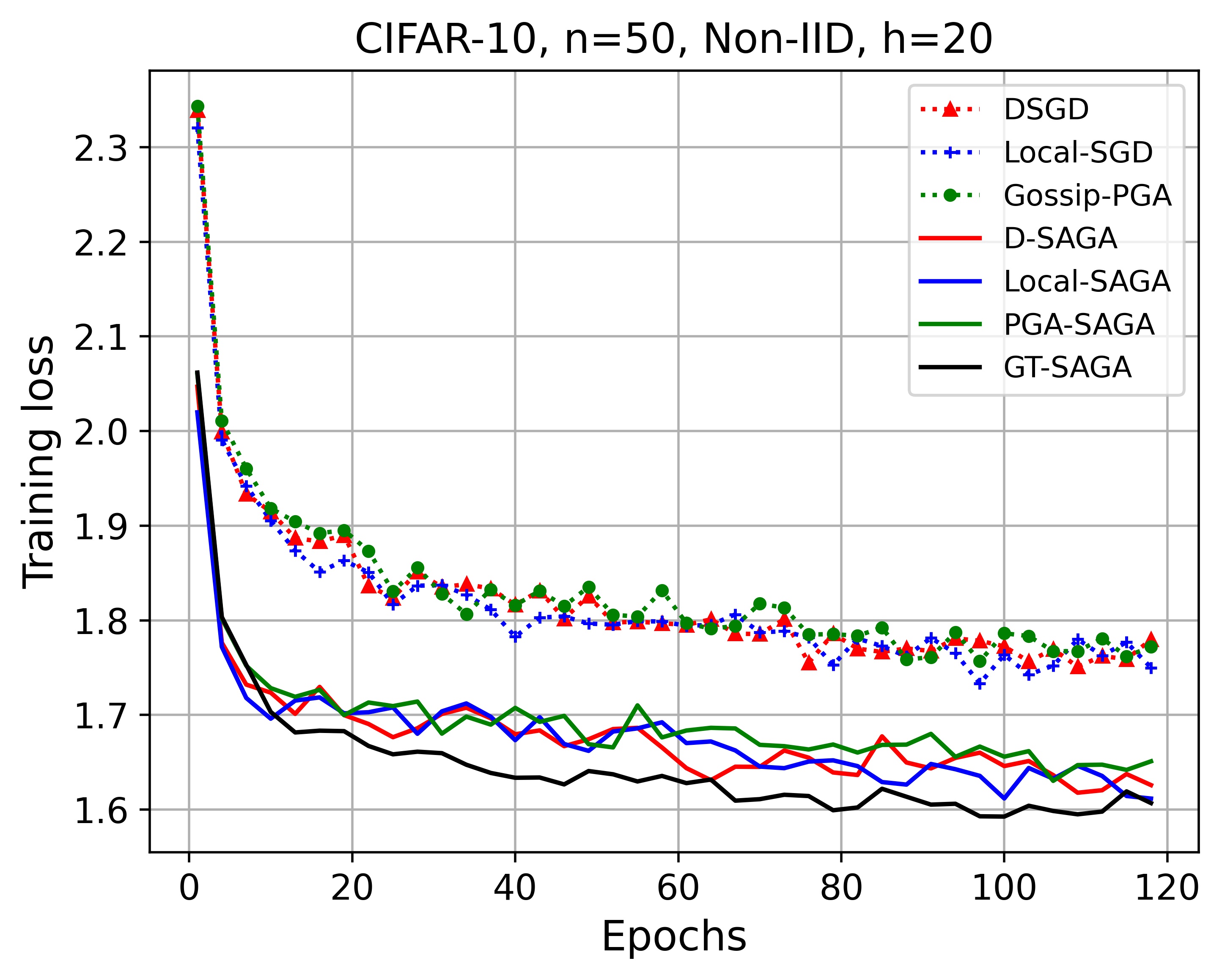}   
        \end{minipage}%
    }\\
    \subfigure 
    {
        \begin{minipage}[t]{0.3\textwidth}
            \centering          
            \includegraphics[width=\textwidth]{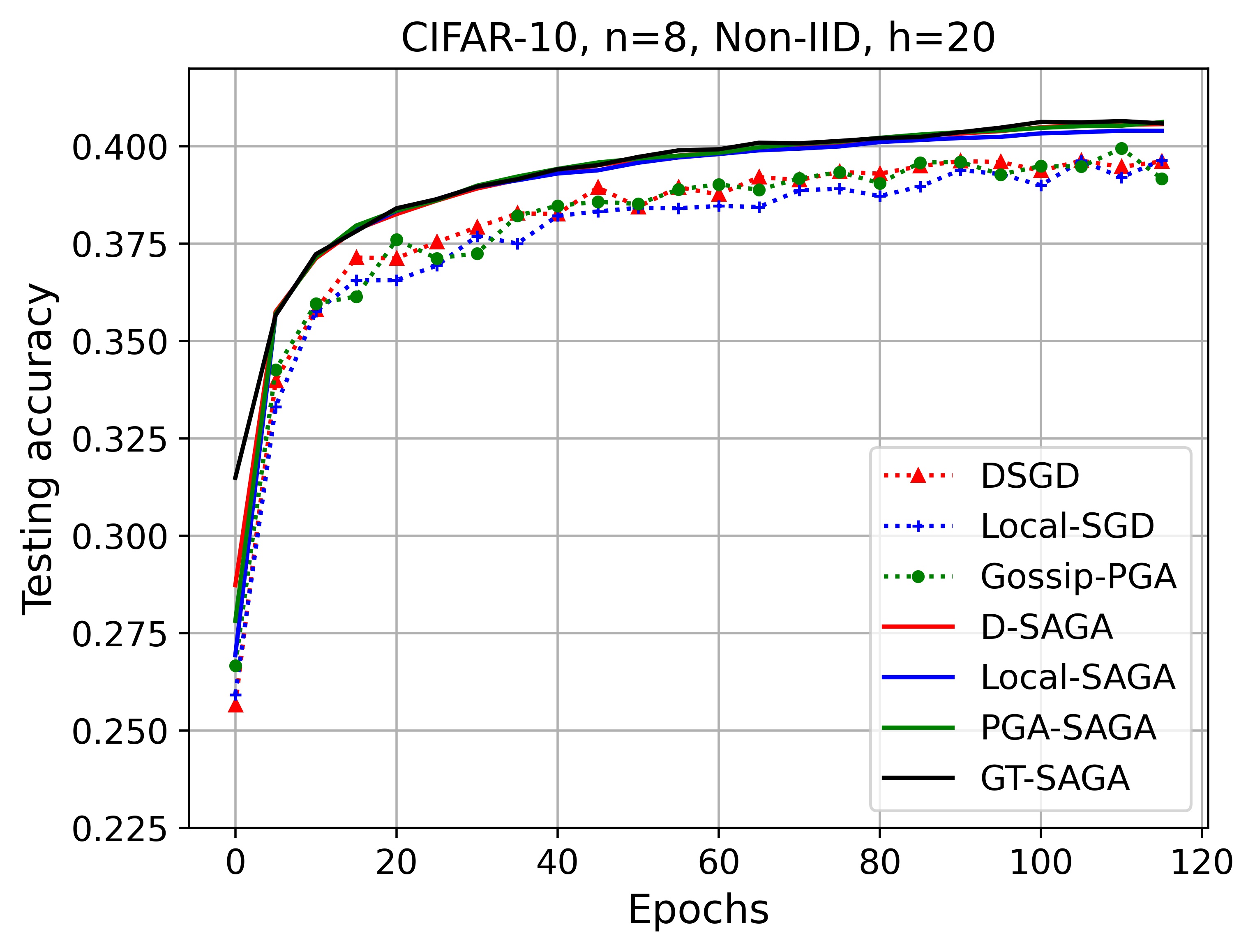}   
        \end{minipage}%
    }
   \subfigure 
    {
        \begin{minipage}[t]{0.3\textwidth}
            \centering      
            \includegraphics[width=\textwidth]{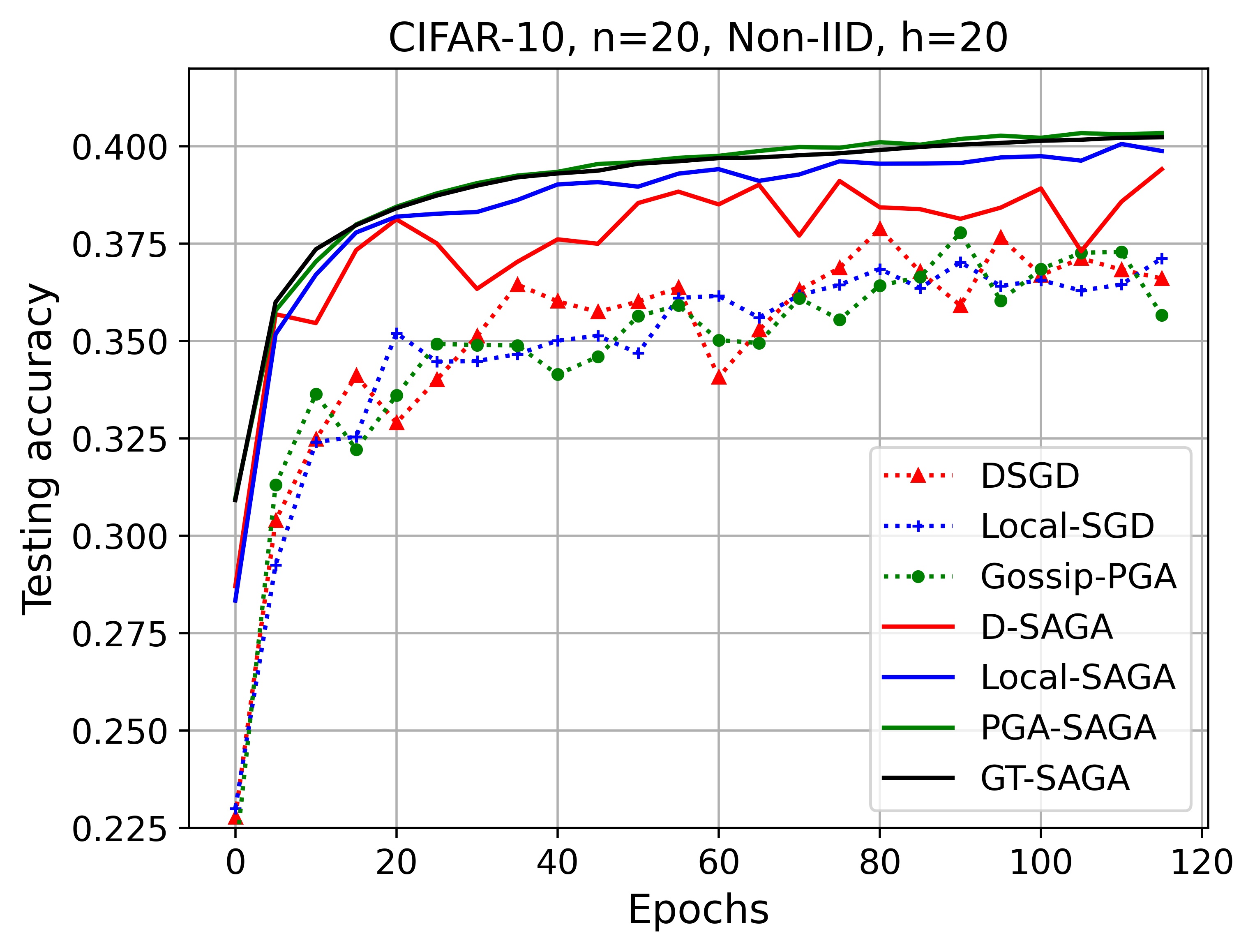}    
        \end{minipage}
    }
    \subfigure 
    {
        \begin{minipage}[t]{0.3\textwidth}
            \centering      
            \includegraphics[width=\textwidth]{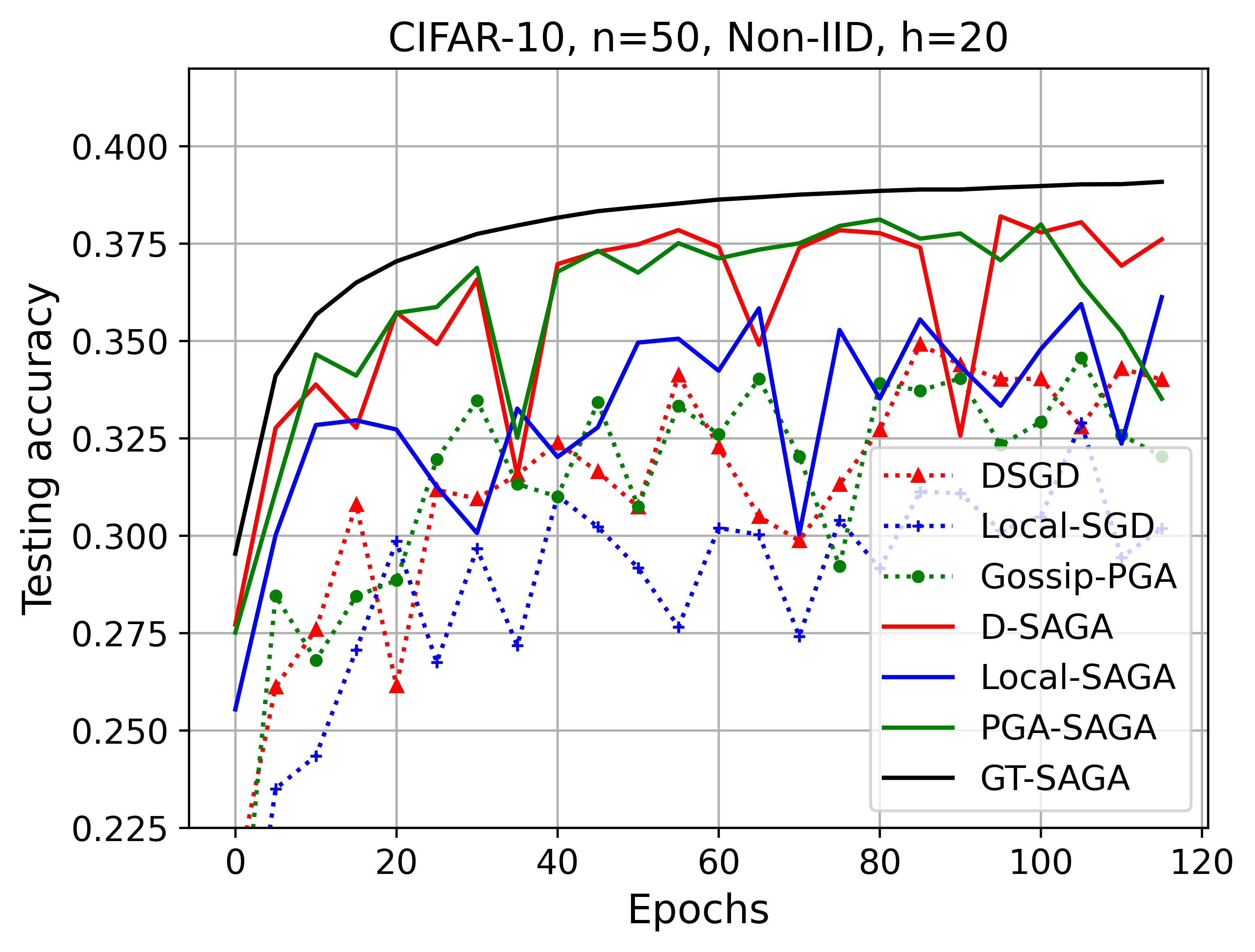}    
        \end{minipage}
    }
    \vspace{-0.25cm}
    \caption{\hy{Performance comparison of DSGD, Local-SGD, Gossip-PGA, D-SAGA, Local-SAGA, PGA-SAGA and GT-SAGA over three graphs: i) directed ring with $n=8$ (first column); ii) directed ring with $n=20$ (second column); iii) geometric graph with $n=50$ (third column). The sub-figures on the first two rows plot the training loss and testing accuracy of the algorithms on Fashion-MNIST dataset, respectively, and the sub-figures on the last two rows plot the training loss and testing accuracy on CIFAR-10 dataset.}}
    \label{Expe_3}
    \vspace{-0.25cm}
\end{figure}

\end{document}